\documentclass[11pt,reqno]{amsart}
\usepackage{amsfonts, amsmath, amssymb, amscd, amsthm, bm}
\usepackage{enumerate}
\usepackage{url}
\usepackage[%backref=page,
linktocpage=true,colorlinks,citecolor=magenta,linkcolor=blue,urlcolor=magenta]{hyperref}
\usepackage{multicol}
\usepackage{cite}
\usepackage{comment}
\usepackage{graphicx}
\usepackage{xcolor}
\usepackage{verbatim}
\usepackage{textcomp}
\usepackage{latexsym}
\usepackage{mathrsfs}
\usepackage{xypic}
\usepackage{pgfplots}
\usepackage{tikz}
\usetikzlibrary{arrows,shapes,trees,backgrounds}

\usepackage[margin=1.1in]{geometry}
 \newtheorem{thm}{Theorem}[section]
 \newtheorem{cor}[thm]{Corollary}
  
 \newtheorem{lem}[thm]{Lemma}
 
 \newtheorem{prop}[thm]{Proposition}
 \theoremstyle{definition}
 \newtheorem{defn}[thm]{Definition}
 \newtheorem{rem}{Remark}[section]
 \newtheorem*{ack}{Acknowledgments}

\numberwithin{equation}{section}
\allowdisplaybreaks

\def\R{\mathbb R}

\def\SS{\mathbb S}

\def\pt{\partial}

\begin{document}
\title{Anisotropic curvature measures and volume preserving flows}
\author[B. Andrews]{Ben~Andrews}
\author[Y. Lei]{Yitao~Lei}
\address{Mathematical Sciences Institute, Australian National University, Canberra, ACT 2601, Australia}
\email{\href{mailto:Ben.Andrews@anu.edu.au}{Ben.Andrews@anu.edu.au}}
\email{\href{mailto:yitao.lei@anu.edu.au}{yitao.lei@anu.edu.au}}
\author[Y. Wei]{Yong Wei}
\address{School of Mathematical Sciences, University of Science and Technology of China, Hefei 230026, P.R. China}
\email{\href{mailto:yongwei@ustc.edu.cn}{yongwei@ustc.edu.cn}}
\author[C. Xiong]{Changwei~Xiong}
\address{College of Mathematics, Sichuan University, Chengdu 610065,  P.R. China}
\email{\href{mailto:changwei.xiong@scu.edu.cn}{changwei.xiong@scu.edu.cn}}
%
%\date{\today}
%\thanks {}
\keywords {Convex body; Anisotropic curvature measure; Volume preserving flow.}
\subjclass[2010]{53C44; 52A39}

%%% ----------------------------------------------------------------------
%
\begin{abstract}
In the first part of this paper, we develop the theory of anisotropic curvature measures for convex bodies in the Euclidean space. It is proved that any convex body whose boundary anisotropic curvature measure equals a linear combination of other lower order anisotropic curvature measures with nonnegative coefficients is a scaled Wulff shape. This generalizes the classical results by Schneider [Comment. Math. Helv. \textbf{54} (1979), 42--60] and by Kohlmann [Arch. Math. (Basel) \textbf{70} (1998), 250--256] to the anisotropic setting. The main ingredients in the proof are the generalized anisotropic Minkowski formulas and an inequality of Heintze--Karcher type for convex bodies.

In the second part, we consider the volume preserving flow of smooth closed convex hypersurfaces in the Euclidean space with speed given by a  positive power $\alpha $ of the $k$th anisotropic mean curvature plus a global term chosen to preserve the enclosed volume of the evolving hypersurfaces. We prove that if the initial hypersurface is strictly convex, then the solution of the flow exists for all time and converges to the Wulff shape in the Hausdorff sense. The characterization theorem for Wulff shapes via the anisotropic curvature measures will be used crucially in the proof of the convergence result. Moreover, in the cases $k=1$, $n$ or $\alpha\geq k$, we can further improve the Hausdorff convergence to the smooth and exponential convergence.
\end{abstract}

\maketitle
\tableofcontents

\section{Introduction}\label{sec:1}

The curvature measures for sets of positive reach (i.e., those sets admitting a locally unique metric projection; see, e.g., \cite[Def.~1.1]{Koh91}) in the Euclidean space have been extensively studied since they were introduced by Federer \cite{Fed59} in 1959. Schneider \cite{Sch78} in 1978 developed the theory of curvature measures for convex bodies, an important subclass of sets with positive reach. Here by convex body we mean compact subset of $\mathbb{R}^{n+1}$ with nonempty interior. Given a convex body $K$ in $\mathbb{R}^{n+1}$, $\varepsilon>0$ and any Borel set $\beta\in \mathcal{B}(\mathbb{R}^{n+1})$, the local $\varepsilon$-parallel set of $K$ is defined by
 \begin{equation*}
   A_{\varepsilon}(K,\beta):=\{x\in \mathbb{R}^{n+1}:~0<d(K,x)\leq\varepsilon,~p(K,x)\in\beta\}
 \end{equation*}
which is the set of all points $x\in \mathbb{R}^{n+1}$ such that the distance $0<d(K,x)\leq \varepsilon$ and the nearest point $p(K,x)$ to $x$ on $K$ belongs to $\beta$. The volume of $A_\varepsilon(K,\beta)$ is a polynomial of degree $n+1$ in the parameter $\varepsilon$ (see \cite[\S 4.2]{Schn}):
 \begin{align}\label{s1.Ste}
   \mathcal{H}^{n+1}( A_{\varepsilon}(K,\beta))=&\frac 1{n+1}\sum_{k=0}^n\varepsilon^{n+1-k}{\binom{n+1}{k}}\mathcal{C}_k(K,\beta)
 \end{align}
for $\beta\in \mathcal{B}(\mathbb{R}^{n+1})$ and $\varepsilon>0$. The equation \eqref{s1.Ste} can be viewed as a localized version of the classical Steiner formula. The coefficients $\mathcal{C}_0(K,\cdot), \dots, \mathcal{C}_n(K,\cdot)$ of \eqref{s1.Ste} are called the \emph{curvature measures} of the convex body $K$, which are Borel measures on $\mathbb{R}^{n+1}$. A special one is $\mathcal{C}_n$ which is the boundary area measure given by $ \mathcal{C}_n(K,\beta)=~\mathcal{H}^n(\beta\cap\partial K)$ for $\beta\in \mathcal{B}(\mathbb{R}^{n+1})$.

The curvature measures, as a measure theoretical replacement of the mean curvatures $E_k$ ($1\leq k\leq n$) in the smooth differential geometry, have many applications in the non-smooth case. In 1979, Schneider \cite{Sch79a} applied the curvature measures to characterize the Euclidean ball, and proved the following generalization of the classical Alexandrov Theorem in the differential geometry.
\begin{thm}[Theorem 8.5.7 in \cite{Schn}]\label{s1.thm-Sch}
Let $m\in\{0,\dots,n-1\}$. If $K$ is a convex body in $\mathbb{R}^{n+1}$ and its curvature measure $\mathcal{C}_{m}(K,\cdot)$ satisfies
\begin{equation}\label{s1.cur-meas1}
  \mathcal{C}_{m}(K,\beta)=~c~ \mathcal{C}_{n}(K,\beta)
\end{equation}
for any Borel set $\beta\in \mathcal{B}(\mathbb{R}^{n+1})$ and a constant $c>0$, then $K$ is a ball.
\end{thm}

The classification result in Theorem \ref{s1.thm-Sch} was generalized later by Kohlmann \cite{Koh98} to the case where a linear combination of curvature measures with nonnegative coefficients equals the boundary area measure. If $K$ is a smooth convex body, then the curvature measures can be given as the following integrals
\begin{align}\label{s1.cur-meas2}
 \mathcal{ C}_{k}(K,\beta)=&\int_{\beta\cap \partial K}E_{n-k}(\lambda)d\mathcal{H}^n,\quad k=0,\dots,n
\end{align}
for any Borel set $\beta\in \mathcal{B}(\mathbb{R}^{n+1})$, where $E_{n-k}(\lambda)$ is the $(n-k)$th mean curvature of the hypersurface $\partial K$ and is defined as the normalized $(n-k)$th elementary symmetric polynomial of the principal curvatures $\lambda=(\lambda_1,\dots,\lambda_n)$ of $\partial K$. The equation \eqref{s1.cur-meas1} for a smooth convex body $K$ is equivalent to that  $E_{n-m}(\lambda)$ is a constant on the boundary $\partial K$, which implies that $K$ is a Euclidean ball by the classical Alexandrov theorem.

The curvature measures are useful in the study of the geometry of hypersurfaces. Recently, the first and the third authors \cite{AW17} developed a new application of the curvature measures in the study of volume preserving curvature flows. The classification result in Theorem \ref{s1.thm-Sch} was used crucially in \cite{AW17} to establish the Hausdorff convergence of a class of volume preserving curvature flows in the Euclidean space by positive powers of higher order mean curvatures. The main objective of this paper is to establish the theory of curvature measures in the anisotropic geometry, and then to develop their applications in the study of volume preserving anisotropic curvature flows.

\subsection{Anisotropic curvature measures}
In this paper, we first develop the analogues of the curvature measures by incorporating some explicit anisotropy, which we call the anisotropic curvature measures. The concept of anisotropy originates from crystallography and dates back to Wulff \cite{Wulff}. Let $\gamma\in C^{\infty}(\mathbb{S}^n)$ be a smooth positive function such that the square of its positively $1$-homogeneous extension on $\mathbb{R}^{n+1}\setminus \{0\}$ is uniformly convex.  Consider the anisotropic surface energy functional
\begin{equation}\label{s1:PK}
  \mathcal{E}(K)=\int_{\partial^*K}\gamma(\nu)d\mathcal{H}^n
\end{equation}
among sets of finite perimeter, where $\partial^*K$ denotes the reduced boundary of $K$ and $\nu$ denotes the outward unit normal vector field. The isoperimetric problem for this energy functional is uniquely solved by a scaled translated Wulff shape $W$, which is the set
\begin{equation*}
  W=\bigcap_{\nu\in \mathbb{S}^n}\left\{x\in \mathbb{R}^{n+1}: x\cdot\nu\leq \gamma(\nu)\right\}
\end{equation*}
determined by the function $\gamma$, see \cite{BM94,FM91,Tay78}. A sharp quantitative version of the anisotropic isoperimetric problem was established by Figalli, Maggi and Pratelli \cite{FMP2010}, corresponding to the stability for the Wulff shape of a given surface energy \eqref{s1:PK}. The function $\gamma\in C^{\infty}(\mathbb{S}^n)$ is called the support function of the Wulff shape $W$. The assumption on $\gamma$ ensures that $W$ is a smooth uniformly convex body in $\mathbb{R}^{n+1}$, and its boundary $\Sigma=\partial W$ is also called the Wulff shape determined by the support function $\gamma$. See \S \ref{sec:2-1} for details.

The anisotropy is considered in the relative (or Minkowski) differential geometry as well, where the Wulff shape plays the similar role as the unit sphere plays in the Euclidean geometry; see \cite{BF87} and references therein.  In fact, the Wulff shape allows us to define the anisotropic normal of any smooth orientable hypersurface $M^n$ in $\mathbb{R}^{n+1}$ by a map $\nu_{\gamma}: M\to \Sigma$, which takes each point $x\in M$ to the point in $\Sigma$ with the same oriented tangent hyperplane.  The anisotropic Weingarten map ${\mathcal W}_{\gamma}$ is the derivative of the anisotropic normal $\nu_{\gamma}$, which is a linear map from $T_xM$ to itself at each point.  The eigenvalues $\kappa=(\kappa_1,\dots,\kappa_n)$ of ${\mathcal W}_{\gamma}$ are called the anisotropic principal curvatures. We define the $k$th anisotropic mean curvature of $M$ as the normalized $k$th elementary symmetric function $E_k$ of the anisotropic principal curvatures $\kappa$:
 \begin{equation}\label{eq:defEk}
   {E}_{k}=E_k(\kappa)={\binom{n}{k}}^{-1}\sum_{1\leq i_1<\cdots<i_k\leq n}\kappa_{i_1}\cdots \kappa_{i_k},\quad k=1,\dots,n.
 \end{equation}
It was proved in \cite{HLMG09} that any smooth embedded closed hypersurface in $\R^{n+1}$ with constant $k$th anisotropic mean curvature for some $1\leq k\leq n$ is a scaled Wulff shape. Note that when the Wulff shape is the unit sphere, the $k$th anisotropic mean curvature defined in \eqref{eq:defEk} is just the $k$th mean curvature of the hypersurface in the Euclidean space, and the result in \cite{HLMG09} reduces to the classical Alexandrov theorem. 

%The anisotropic curvature measure theory we will develop in this paper allows us to generalize the result in \cite{HLMG09} to non-smooth setting.

To introduce the anisotropic curvature measures, we need a notion of local ``$\varepsilon$-parallel set'' of a convex body in the anisotropic setting. This will be defined via the following relative distance function. Given a Wulff shape $W$, we define a (non-symmetric) distance function in $\mathbb{R}^{n+1}$ relative to $W$, by setting $d_W(x,y) = \inf\{r>0:\ y\in x+rW\}$. Then the relative distance of a point $y\in \R^{n+1}$ from a set $K$ is defined as
\begin{equation*}
  d_W(K,y) = \inf\{d_W(x,y):\ x\in K\}.
\end{equation*}
If $K$ is a convex body, then the infimum in the above definition is attained at a single point which we denote by $f_K(y)\in\partial K$.  For any open set $\beta$ in $\R^{n+1}$ and $\varepsilon>0$, we define the local $\varepsilon$-parallel set of $K$ as follows
$$
A_{\varepsilon}(K,\beta) = \{x\in\R^{n+1}:\ 0<d_W(K,x)\leq \varepsilon,\ f_K(x)\in\beta\}.
$$
The anisotropic curvature measures can be formally defined as the coefficients of the polynomial
\begin{equation}\label{s1:Sten}
  {\mathcal H}^{n+1}(A_\varepsilon(K,\beta)) = \frac{1}{n+1}\sum_{r=0}^n\varepsilon^{n+1-r}{n+1\choose r}\Phi_r(K;\beta),
\end{equation}
which is called the Steiner formula for the volume of the anisotropic local parallel set of $K$. Strictly speaking,  as $A_{\varepsilon}(K,\beta)$ and $\Phi_r(K;\beta)$ depend on the Wulff shape $W$ and it would be better to use the notation $A_{\varepsilon,W}(K,\beta)$ and $\Phi_{r,W}(K;\beta)$, but we will omit the subscript $W$ for the simplicity of the notation.  We will discuss in \S \ref{sec:ACM} the derivation of the formula \eqref{s1:Sten}, and prove that $\Phi_r(K;\beta)$ indeed defines a measure on $\R^{n+1}$ and has the weak continuity property with respect to the Hausdorff distance. As a key ingredient, we introduce the generalized anisotropic principal curvatures of a convex body $K$ living on the anisotropic normal bundle of $K$. In \S \ref{sec:Mink}, using the anisotropic curvature measures we prove a generalized version of the Minkowski formulas, which hold for any convex body in the Euclidean space. Moreover, the anisotropic curvature measures allow us to derive in \S \ref{sec:Vol} a representation formula for the volume of a convex body involving the anisotropic interior reach, which leads to an inequality of Heintze--Karcher type. Combining the anisotropic Minkowski formulas and the inequality of Heintze--Karcher type, we are able to prove the  following theorem on the anisotropic curvature measures, which is the content of \S \ref{sec:proof_thm1}.
\begin{thm}\label{thm-main}
Let $W$ be a Wulff shape in $\mathbb{R}^{n+1}$ with the support function $\gamma\in C^{\infty}(\mathbb{S}^n)$. For a convex body $K$ and some constants $\lambda_0,\dots,\lambda_{n-1}\geq 0$, assume that the anisotropic curvature measures of $K$ satisfy
\begin{equation}
\Phi_n(K;\cdot)=\sum_{r=0}^{n-1} \lambda_r \Phi_r(K;\cdot).
\end{equation}
Then $K$ is a scaled Wulff shape.
\end{thm}

Theorem \ref{thm-main} is the anisotropic analogue of the results proved by Schneider \cite{Sch79a} and Kohlmann \cite{Koh98}. In particular, any convex body with anisotropic curvature measures satisfying
\begin{equation}\label{s1:Phk-cst}
\Phi_n(K;\cdot)=c\Phi_{n-k}(K;\cdot)
\end{equation}
for some $k=1,\dots, n$ and some constant $c>0$ must be a scaled Wulff shape. When the Wulff shape $W$ is the unit Euclidean ball, Theorem \ref{thm-main} with anisotropic curvature measures satisfying \eqref{s1:Phk-cst} reduces to Theorem \ref{s1.thm-Sch} by Schneider \cite{Sch79a}.  When the convex body $K$ is smooth, Equation \eqref{s1:Phk-cst} implies that the $k$th anisotropic mean curvature of $\partial K$ is a constant. Then in this case the conclusion of Theorem \ref{thm-main} reduces to the result proved by He, Li, Ma, and Ge \cite{HLMG09} for hypersurfaces with constant $k$th anisotropic mean curvature.

\subsection{Volume preserving anisotropic curvature flows} In the second part of this paper, we will apply Theorem \ref{thm-main} to study the anisotropic analogue of the volume preserving curvature flow. Let $W$ be a Wulff shape in $\mathbb{R}^{n+1}$ with the support function $\gamma\in C^{\infty}(\mathbb{S}^n)$, and let $X_0: M^n\to \mathbb{R}^{n+1}$ be a smooth embedding such that $M_0=X_0(M)$ is a strictly convex hypersurface in $\mathbb{R}^{n+1}$ enclosing a smooth convex body $K_0$. We consider the smooth family of embeddings $X:M^n\times [0,T)\rightarrow \mathbb{R}^{n+1}$ satisfying
\begin{equation}\label{flow-VMCF}
 \left\{\begin{aligned}
 \frac{\partial}{\partial t}X(x,t)=&~(\phi(t)-{E}_k^{{\alpha}/k}(\kappa))\nu_{\gamma}(x,t),\\
 X(\cdot,0)=&~X_0(\cdot),
  \end{aligned}\right.
 \end{equation}
where $\alpha>0$, $E_k(\kappa)$ is the $k$th anisotropic mean curvature defined in \eqref{eq:defEk} and $\nu_{\gamma}(x,t)$ is the anisotropic normal of the evolving hypersurface $M_t=X(M,t)$.
The global term $\phi(t)$ in the flow \eqref{flow-VMCF} is chosen as
 \begin{align}\label{s1:phi-1}
  \phi(t) =& \frac{1}{|M_t|_{\gamma}}\int_{M_t}{E}_k^{{\alpha}/k}(\kappa)d\mu_{\gamma}
\end{align}
 to preserve the volume of the closure $K_t$ of the domain enclosed by $M_t$, where $d\mu_{\gamma}=\gamma(\nu)d\mu$ is the anisotropic area form on $M_t$ and $|M_t|_{\gamma}$ is the anisotropic area of $M_t$ (i.e., the anisotropic perimeter defined in \eqref{s1:PK} for smooth convex bodies).

We will prove the following convergence result for the flow \eqref{flow-VMCF}.
\begin{thm}\label{thm-2}
Let $W$ be a Wulff shape in $\mathbb{R}^{n+1}$ with the smooth support function $\gamma\in C^{\infty}(\mathbb{S}^n)$, and $X_0: M^n\to \mathbb{R}^{n+1}$ be a smooth embedding such that $M_0=X_0(M)$ is a closed strictly convex hypersurface in $\mathbb{R}^{n+1}$ enclosing a convex body $K_0$.
\begin{enumerate}
  \item For any $k\in \{1,\dots,n\}$ and $\alpha>0$, the volume preserving flow \eqref{flow-VMCF} has a smooth strictly convex solution $M_t$ for all time $t\in [0,\infty)$, and $M_t$ converges as $t\to\infty$ to a scaled Wulff shape $M_{\infty}=\bar{r}\Sigma$ in the Hausdorff sense, where $\bar{r}$ is the radius such that $\mathrm{Vol}(K_0)=\bar{r}^{n+1}\mathrm{Vol}(W)$.
  \item For (i) $k=1$ and $\alpha>0$, or (ii) $k=n$ and $\alpha>0$, or (iii) $\alpha\geq k=1,\dots,n$,  we can improve the Hausdorff convergence in (1) to the smooth and exponential convergence.
\end{enumerate}
\end{thm}

Anisotropic analogue of mean curvature flow has been studied as models of crystal growth \cite{TCH92} and other phenomena involving the motion of interfaces \cite{AG89,AG94}. The graphical anisotropic mean curvature flow has been studied by the first author and Clutterbuck \cite{And-C09,Clu07}, but the convergence result of the anisotropic mean curvature flow for smooth, closed and convex hypersurfaces (anisotropic analogue of Huisken's result \cite{Hu84} in 1984) has not yet been proved. The difficulty lies in the increased complexity of the evolution equations of geometric quantities due to the presence of the anisotropy and thus it is hard to obtain a priori curvature pinching ratio estimates.  In 2001, the first author \cite{And01} considered the anisotropic analogue of Huisken's volume-preserving mean curvature flow \cite{Hu87} (corresponding to the case $k=1$ and $\alpha=1$ of the flow \eqref{flow-VMCF}), and proved that the solutions converge to the Wulff shape corresponding to the anisotropy. Though the curvature pinching is still not available, the convergence result was proved by exploring the monotonicity of the isoperimetric ratio. Our Theorem \ref{thm-2} is a generalization of the main result in \cite{And01} to the flow by arbitrary positive powers of higher order anisotropic mean curvatures.

In the proof of Theorem \ref{thm-2}, we exploit the isoperimetric bounds on the evolving domains. For a smooth convex body $K$ in $\mathbb{R}^{n+1}$, the mixed volume relative to $W$ is defined as
\begin{align*}%\label{def-quermass}
 & V_{n+1}(K,W)=(n+1)\mathrm{Vol}(K),\qquad  V_{0}(K,W)=(n+1)\mathrm{Vol}(W),\\
  & V_{n+1-k}(K,W)=\int_{\partial K}{E}_{k-1}d\mu_{\gamma},\quad k=1,\dots,n,
\end{align*}
where ${E}_{k-1}$ is the $(k-1)$st anisotropic mean curvature of $\partial K$ relative to the Wulff shape $W$. We consider the following isoperimetric ratio
\begin{equation}\label{s1:Iso-def}
{\mathcal I}_\ell(K,W) = \frac{V_\ell(K,W)^{n+1}}{V_{n+1}(K,W)^\ell ~V_0(K,W)^{n+1-\ell}},\quad 1\leq \ell\leq n,
\end{equation}
for the convex body $K$. We will prove that $\mathcal{I}_{n+1-k}(K_t,W)$ is monotone non-increasing in time along the flow \eqref{flow-VMCF}, and thus is bounded from above by its initial value. On the other hand, the Alexandrov--Fenchel inequality \eqref{AF-k=n+1} implies
\begin{equation*}%\label{s5:Il-2}
{\mathcal I}_\ell(K,W)\geq 1,\quad 1\leq \ell\leq n,
\end{equation*}
with the equality if and only if $K$ is homothetic to $W$. These can be used to control the geometry of the evolving hypersurface $M_t$, including the $n$-dimensional anisotropic area, the anisotropic inner radius and outer radius.

To study the long-time behavior of the flow \eqref{flow-VMCF}, we employ the anisotropic Gauss map parametrization which we will review in \S\ref{sec:gauss map} and we rewrite the flow as a scalar parabolic equation of the ``anisotropic support function $s$'' on the Wulff shape $\Sigma$. This allows us to use the inner radius bound and Tso's technique \cite{Tso85} to give an upper bound on the $k$th anisotropic mean curvature and the two-sided positive bounds on the global term $\phi(t)$. We can also obtain the convexity estimate for the evolving hypersurfaces by estimating an upper bound on the anisotropic principal radii $\tau=(\tau_1,\dots,\tau_n)$, which are eigenvalues of the matrix
\begin{equation*}
  \tau_{ij}[s]=\bar{\nabla}_i\bar{\nabla}_js+\bar{g}_{ij}s-\frac 12Q_{ijk}\bar{\nabla}_ks.
\end{equation*}
Here $Q_{ijk}$ comes from the anisotropy which creates lots of extra terms in the evolution equation of $\tau_{ij}$. Thus the maximum principle for tensors may not be applied directly to the evolution of $\tau_{ij}$ to deduce its upper bound. To overcome this problem, we notice that the evolution equation of $s$ has some terms which could be used to control the bad terms in the evolution of $\tau_{ij}$, and we can combine them together to get a time dependent upper bound on the largest anisotropic principal radius. Then a standard contradiction argument can be used to conclude the long-time existence of the flow.

From the monotonicity of $\mathcal{I}_{n+1-k}(K_t,W)$ and the long-time existence of the flow, we show that the anisotropic curvature measures of the limit convex body satisfy the equation \eqref{s1:Phk-cst} and thus the limit convex body is a scaled Wulff shape by our Theorem \ref{thm-main}. This proves the Hausdorff convergence of $K_t$ to a scaled Wulff shape. Then we can apply the argument in \cite[\S 7]{AW17} (a variation of Smoczyk's method \cite{Smo98}) to conclude a uniform lower bound on the $k$th anisotropic mean curvature. This allows us to obtain uniform regularity estimates for the flow \eqref{flow-VMCF} in the cases $k=1$, $\alpha>0$ and $k=n$, $\alpha>0$. The smooth convergence of the flow then follows in these cases.

To deal with the case $\alpha\geq k=1,\dots,n$, we need to further explore the Hausdorff convergence of the solution. We first improve the $C^0$ and $C^1$ estimates and show that the anisotropic support function is close to a constant (with respect to some moving center) and its gradient is almost zero. Since $\alpha\geq k$, by the Haussdorff convergence and H\"{o}lder inequality we also show that the global term $\phi(t)$ satisfies the property
\begin{equation*}
  \liminf_{t\to\infty}\phi(t)\geq \bar{r}^{-\alpha},
\end{equation*}
where $\bar{r}$ is the anisotropic radius of the limit Wulff shape. To derive the $C^2$ estimate, we again employ the anisotropic Gauss map parametrization. As we already have the $C^1$ estimate, the $C^2$ estimate is equivalent to the bounds on the anisotropic principal curvatures. Equivalently, it suffices to estimate the upper bound on the anisotropic principal radii of curvature  $\tau_1,\dots,\tau_n$. For this purpose, we apply the maximum principle to the evolution equation of an auxiliary function involving $\tau_{ij}[s]$, the anisotropic support function $s$ and its gradient, and use the improved $C^0$, $C^1$ estimates of $s$ and the above-mentioned property of $\phi(t)$. This is the most technical part in the proof, where an observation on smooth symmetric functions due to Guan, Shi, and Sui \cite{Guanbo15} will be used. Finally, we show that the smooth convergence is in the exponential rate by studying the linearization of the flow around the Wulff shape.

\begin{rem}
We remark that the smooth convergence in Theorem \ref{thm-2} holds for all $\alpha>0$ and all $k=1,\cdots,n$ in the isotropic case (i.e., in the case that the Wulff shape $W$ is the unit Euclidean ball), see Theorem 1.1 in \cite{AW17}. It's a natural question whether this is still true in the general anisotropic setting for $0<\alpha<k$, $k=2,\cdots,n-1$.
\end{rem}

\subsection{Organization of the paper} In \S\ref{sec:2} we review some basic facts on the Wulff shape, the anisotropic geometry, the mixed volumes and the Alexandrov--Fenchel inequalities. \S\S\ref{sec:ACM}--\ref{sec:proof_thm1} comprise the first part of the paper. In \S\ref{sec:ACM} we introduce the anisotropic curvature measures and study their basic properties. In \S\ref{sec:Mink} we prove the generalized anisotropic Minkowski integral formulas. Next in \S\ref{sec:Vol} we discuss the volume representation of convex bodies in terms of the anisotropic interior reach. Finally in \S\ref{sec:proof_thm1} we first use the volume representation in \S\ref{sec:Vol} to establish an inequality of Heintze--Karcher type, and then combine this inequality with the Minkowski formulas in \S\ref{sec:Mink} to complete the proof of Theorem~\ref{thm-main}. After \S\ref{sec:proof_thm1} we come to the second part of the paper, which consists of \S\S\ref{sec:VPACF}--\ref{sec:SC}. In \S\ref{sec:VPACF} we review the variation equations along a general anisotropic curvature flow and study basic properties of our volume preserving anisotropic curvature flow. In \S\ref{sec:gauss map} we introduce the anisotropic Gauss map parametrization of convex hypersurfaces and derive evolution equations for various geometric quantities under this parametrization. Next in \S\ref{sec:LTE} we show the long-time existence of our anisotropic curvature flow by investigating the upper bound for the anisotropic principal radii of curvature of
the evolving hypersurfaces. Then in \S\ref{sec:HC} we use the characterization of Wulff shapes via anisotropic curvature measures in Theorem~\ref{thm-main} to obtain the Hausdorff convergence of the flow. This could be used to improve the estimate on the global term $\phi(t)$. By the Hausdorff convergence and a variant of Smoczyk's method \cite{Smo98}, in \S\ref{sec:HC} we also prove a uniform positive lower bound on the $k$th anisotropic mean curvature. In the last section \S\ref{sec:SC} we show the smooth and exponential convergence of the flow for some special cases as mentioned above, and so finish the proof of Theorem~\ref{thm-2}. Throughout the paper the Einstein convention on the summation for indices is used unless otherwise indicated.

\begin{ack}
The research of the first two authors was supported by Laureate Fellowship FL150100126 of the Australian Research Council. The third author was supported by National Key Research and Development Project SQ2020YFA070080 and Research grant KY0010000052 from University of Science and Technology of China. The fourth author was supported by the funding (no. 1082204112549) from Sichuan University.
\end{ack}

\section{Preliminaries on anisotropic geometry}\label{sec:2}
In this section, we briefly review the anisotropic geometry and mixed volumes in convex geometry. We refer the readers to \cite{And01,Schn,Xia13} for more details.
\subsection{The Wulff shape}\label{sec:2-1}
Let $\gamma$ be a smooth positive function on the sphere $\mathbb{S}^n$ such that the matrix
\begin{equation}\label{s2:A_F}
  A_{\gamma}(x)~=~\nabla^{\mathbb{S}}\nabla^{\mathbb{S}} {\gamma}(x)+{\gamma}(x)g_{\mathbb{S}^n},\quad x\in \mathbb{S}^n
\end{equation}
is positive definite, where $\nabla^{\mathbb{S}}$ denotes the covariant derivative on $\mathbb{S}^n$. Then there exists a unique smooth strictly convex hypersurface $\Sigma$ given by
\begin{equation*}
  \Sigma=\left\{\phi(x)|\phi(x):={\gamma}(x)x+\nabla^{\mathbb{S}} {\gamma}(x),~x\in \mathbb{S}^n\right\}
\end{equation*}
whose support function is given by ${\gamma}$. Denote the closure of the enclosed domain of $\Sigma$ by $W$. We call $W$ (and $\Sigma$) the Wulff shape determined by the function $\gamma\in C^{\infty}(\mathbb{S}^n)$. When $\gamma$ is a constant, the Wulff shape is just a round sphere.

The smooth function ${\gamma}$ on $\mathbb{S}^n$ can be extended homogeneously to a $1$-homogeneous function on $\mathbb{R}^{n+1} $ by
\begin{equation*}
  {\gamma}(x)=|x|{\gamma}({x}/{|x|}), \quad x\in \mathbb{R}^{n+1}\setminus\{0\}
\end{equation*}
and setting ${\gamma}(0)=0$. Then it is easy to show that $\phi(x)=D\gamma(x)$ for $x\in \mathbb{S}^n$, where $D$ denotes the gradient on $\mathbb{R}^{n+1}$. The homogeneous extension ${\gamma}$ defines a Minkowski norm on $\mathbb{R}^{n+1}$, that is, $\gamma$ is a norm on $\mathbb{R}^{n+1}$ and $D^2(\gamma^2)$ is uniformly positive definite in $\mathbb{R}^{n+1}\setminus\{0\}$. We can define a dual Minkowski norm ${\gamma}^0$ on $\mathbb{R}^{n+1}$ by
\begin{equation}\label{s2:gam0}
  {\gamma}^0(z):=\sup_{x\neq 0}\frac{\langle x,z\rangle}{{\gamma}(x)},\quad z\in \mathbb{R}^{n+1}.
\end{equation}
Then the Wulff shape $W$ can be written as
\begin{equation*}
W=\left\{z\in \R^{n+1}:{\gamma}^0(z)\leq 1\right\},
\end{equation*}
and $\Sigma=\partial W=\{z\in \R^{n+1}: \gamma^0(z)=1\}$.

\subsection{Anisotropic geometry}\label{sec:2-2}
Let $M$ be a smooth orientable hypersurface in the Euclidean space $\mathbb{R}^{n+1}$ with a unit  normal vector field $\nu$, and $\Sigma$ be the Wulff shape defined in \S \ref{sec:2-1}. We define the anisotropic normal of $M$ as the map $\nu_{\gamma}: M\to \Sigma$ given by
\begin{equation*}
  \nu_{\gamma}(x)~=~\phi(\nu(x))={\gamma}(\nu(x))\nu(x)+\nabla^{\mathbb{S}} {\gamma}(\nu(x))~\in ~\Sigma
\end{equation*}
for any point $x\in M$. It follows from the $1$-homogeneity of $\gamma$ that $\nu_{\gamma}(x)=D\gamma(\nu(x))$. The anisotropic Weingarten map is a linear map
$$\mathcal{W}_{\gamma}=d\nu_{\gamma}: T_xM\to T_{\nu_{\gamma}(x)}\Sigma.$$
Note that $\mathcal{W}_{\gamma}=A_{\gamma}\circ \mathcal{W}$, where $\mathcal{W}=d\nu=(h_i^j)$ is the Weingarten map of the hypersurface $M$. Recall that the eigenvalues of $\mathcal{W}$ are the principal curvatures $\lambda=(\lambda_1,\dots,\lambda_n)$. The eigenvalues of $\mathcal{W}_{\gamma}$ are called the anisotropic principal curvatures of $M$, and we denote them by $\kappa=(\kappa_1,\dots,\kappa_n)$.  If we write $\mathcal{W}_{\gamma}=(\hat{h}_i^j)$ in local coordinates, then
\begin{equation}\label{s2:hat-h}
  \hat{h}_i^j(x)=(A_{\gamma}(\nu(x)))_i^kh_k^j(x).
\end{equation}
In particular, it follows from \eqref{s2:hat-h} that $E_n(\kappa)=\det(A_{\gamma}(\nu))E_n(\lambda)$.

For any smooth orientable hypersurface $M$ in $\mathbb{R}^{n+1}$, we define the anisotropic area functional as
\begin{equation*}
  |M|_{\gamma}:=\int_M\gamma(\nu)d\mu,
\end{equation*}
where $\nu$ is  a fixed unit normal vector field of $M$ and $d\mu$ is the area form of the induced metric on $M$ from the Euclidean space $\mathbb{R}^{n+1}$. We set $d\mu_{\gamma}=\gamma(\nu)d\mu$ and call it the anisotropic area form of $M$.

There is another formulation of the anisotropic principal curvatures, which was introduced by the first author in \cite{And01} and reformulated by Xia in \cite{Xia13}. The idea is to define a new metric on $\mathbb{R}^{n+1}$ using the second derivatives of the dual function ${\gamma}^0$ of the homogeneous extension of $\gamma$. This new metric on $\mathbb{R}^{n+1}$ is defined as
\begin{equation}\label{s2:G}
  G(z)(U,V):=\frac 12\sum_{i,j=1}^{n+1}\frac{\partial^2({\gamma}^0)^2(z)}{\partial z^i\partial z^j}U^iV^j
\end{equation}
for $z\in \mathbb{R}^{n+1}\setminus\{0\}$ and $U,V\in T_{z}\mathbb{R}^{n+1}$.  The third derivatives of ${\gamma}^0$ can be used to define
\begin{equation}\label{s2:Q-def}
  Q(z)(U,V,Y):=\frac 12\sum_{i,j,k=1}^{n+1}\frac{\partial^3({\gamma}^0)^2(z)}{\partial z^i\partial z^j\partial z^k}U^iV^jY^k
\end{equation}
for $z\in \mathbb{R}^{n+1}\setminus\{0\}$ and $U,V,Y\in T_{z}\mathbb{R}^{n+1}$. The $1$-homogeneity of ${\gamma}^0$ implies that
\begin{align*}
  G(z)(z,z) =& 1,\quad G(z)(z,V)=0,\quad \mathrm{for}\quad z\in \Sigma ~\mathrm{and} ~V\in T_{z}\Sigma, \\
  Q(z)(z,U,V) =& 0,\qquad \mathrm{for}\quad z\in \Sigma ~\mathrm{and}~U,V\in \mathbb{R}^{n+1}.
\end{align*}

For a smooth hypersurface $M$ in $\mathbb{R}^{n+1}$, the anisotropic normal $\nu_{\gamma}$ lies in $\Sigma$. Then
\begin{align*}
  G(\nu_{\gamma})(\nu_{\gamma},\nu_{\gamma}) =& 1,\quad G(\nu_{\gamma})(\nu_{\gamma},U)=0,\quad \mathrm{for} ~U\in TM, \\
   Q(\nu_{\gamma})(\nu_{\gamma},U,V)=&0,\qquad \mathrm{for} ~U,V\in \mathbb{R}^{n+1}. %\label{s2:Q1}
\end{align*}
Thus $\nu_{\gamma}$ is perpendicular to $TM$ with respect to the metric $G(\nu_{\gamma})$. This induces a Riemannian metric $\hat{g}$ on $M$ from $(\mathbb{R}^{n+1},G)$ by
\begin{equation}\label{s2:g-hat}
\hat{g}(x):=G(\nu_{\gamma}(x))\big|_{T_xM},\quad x\in M. %\hat{g}_{ij}=G(\nu_{\gamma})(\partial_iX,\partial_jX)
\end{equation}
The second fundamental form of $(M,\hat{g})\subset (\mathbb{R}^{n+1},G)$ is defined by
\begin{equation}\label{s2:sff}
  \hat{h}_{ij}=-G(\nu_{\gamma})(\nu_{\gamma},\partial_i\partial_jX),
\end{equation}
where $X$ is the position vector of the hypersurface $M$ in $\mathbb{R}^{n+1}$. Then the anisotropic Weingarten map has the expression $\mathcal{W}_{\gamma}=(\hat{h}_i^j)=(\sum_k\hat{g}^{jk}\hat{h}_{ik})$.  The anisotropic Weingarten formula says that
\begin{equation}\label{s2:AWeingart}
  \partial_i\nu_F=\hat{h}_{i}^k\partial_kX.
\end{equation}
We can state the anisotropic analogue of the Gauss equation and Codazzi equation:
\begin{align}\label{s2:Gaus1}
  \hat{R}_{ijk\ell} =& ~\hat{h}_{ik}\hat{h}_{j\ell}-\hat{h}_{i\ell}\hat{h}_{jk}+\hat{\nabla}_{\ell}A_{jki}  -\hat{\nabla}_kA_{j\ell i}\nonumber\\
 & \qquad +\hat{g}^{pm}A_{jkp}A_{m\ell i}-\hat{g}^{pm}A_{j\ell p}A_{mki},\\
  \hat{\nabla}_k\hat{h}_{ij}+&\hat{h}_j^\ell{A}_{\ell ki}~=~\hat{\nabla}_j\hat{h}_{ik}+\hat{h}_k^\ell{A}_{\ell ji},\nonumber
\end{align}
where $\hat{\nabla}$ and $\hat{R}$ are the Levi-Civita connection and Riemannian curvature tensor of $\hat{g}$ respectively, and $A$ is a $3$-tensor
\begin{equation*}%\label{s2:A-def}
{A}_{ijk}=-\frac 12\left(\hat{h}_i^\ell Q_{jk\ell}+\hat{h}_j^\ell Q_{i\ell k}-\hat{h}_k^\ell Q_{ij\ell}\right),
\end{equation*}
with $Q_{ijk}=Q(\nu_{\gamma})(\partial_iX,\partial_jX,\partial_kX)$.  Note that by definition \eqref{s2:Q-def}, $Q$ is totally symmetric in all three indices. Hence the tensor ${A}_{ijk}$ is symmetric for the first two indices.

When the hypersurface $M$ is the Wulff shape $\Sigma$, the anisotropic normal $\nu_{\gamma}$ is just the position vector and the anisotropic principal curvatures are all equal to $1$. Then $A_{ijk}=-Q_{ijk}/2$ which is symmetric in all indices.  We use the notations $\bar{g}$, $\bar{\nabla}$ and $\bar{R}$ to denote the induced metric on the Wulff shape $\Sigma$ from $(\mathbb{R}^{n+1},G)$, and its Levi-Civita connection and curvature tensor, respectively. Then $\bar{\nabla}_\ell Q_{ijk}$ is totally symmetric in four indices (Prop.~2.2 in \cite{xia-2}). The Gauss equation \eqref{s2:Gaus1} is simplified as
 \begin{equation}\label{s2:gauss-2}
  \bar{R}_{ijk\ell}=\bar{g}_{ik}\bar{g}_{j\ell}-\bar{g}_{i\ell}\bar{g}_{jk}+\frac 14\bar{g}^{pq}Q_{i\ell p}Q_{qjk}-\frac 14\bar{g}^{pq}Q_{ikp}Q_{j\ell q}.
\end{equation}

\subsection{Mixed volumes and Alexandrov--Fenchel inequalities}\label{sec:2-3}
Let $W$ be the Wulff shape defined in \S \ref{sec:2-1}, which is a smooth strictly convex body in $\mathbb{R}^{n+1}$ with boundary $\Sigma=\partial W$ and with the support function given by a smooth function $\gamma\in C^{\infty}(\mathbb{S}^n)$. Associated with $W$,  the anisotropic distance in $\mathbb{R}^{n+1}$ is defined by
\begin{equation}\label{s2:dW}
d_W(x,y)={\gamma}^0(y-x),\qquad x,y\in \mathbb{R}^{n+1},
\end{equation}
where $\gamma^0$ is the dual norm defined in \eqref{s2:gam0}.  Equivalently, we have
\begin{equation*}
d_W(x,y)=\inf \left\{r>0:y\in x+rW\right\}.
\end{equation*}
This distance is not symmetric (unless $W$ is centrosymmetric), but satisfies the triangle inequality
\begin{equation*}
  d_W(x,z)\leq d_W(x,y)+d_W(y,z)
\end{equation*}
with the equality if and only if $x,y,z$ are collinear in that order.

Let $K$ be a convex body in $\R^{n+1}$. For any $y\in \mathbb{R}^{n+1}$ we define the anisotropic distance of $y$ to $K$ by
\begin{equation}\label{s2:dWK}
d_W(K,y)=\inf\left\{d_W(x,y):x\in K\right\}.
\end{equation}
For any $\varepsilon>0$, we define the $\varepsilon$-parallel set of $K$ by
\begin{equation*}
  K^{\varepsilon}:=\left\{x\in \mathbb{R}^{n+1}~|~d_W(K,x)\leq\varepsilon\right\}
\end{equation*}
which is also equal to the Minkowski sum of the two convex bodies $K$ and $W$
\begin{equation*}
 K^\varepsilon= K+\varepsilon W=\left\{x+\varepsilon y~|~x\in K,y\in W\right\}.
\end{equation*}
The mixed volumes of a convex body $K$ relative to $W$ are defined as the coefficients of the volume of the above $\varepsilon$-parallel set of $K$, which is a polynomial of $\varepsilon$:
\begin{equation}\label{s2:Ste}
  \mathrm{Vol}(K^{\varepsilon})=\frac 1{n+1}\sum_{k=0}^{n+1}\varepsilon^{n+1-k}\binom{n+1}kV_{k}(K,W).
\end{equation}
Here the left-hand side of \eqref{s2:Ste} is the volume of $K^{\varepsilon}$ with respect to the Euclidean metric, i.e., the $(n+1)$-dimensional Lebesgue measure (or $(n+1)$-dimensional Hausdorff measure) of $K^{\varepsilon}$.  The formula \eqref{s2:Ste} is called the Steiner formula for the convex body $K$. In particular,
\begin{align*}
  V_{n+1}(K,W)=&(n+1)\mathrm{Vol}(K), \\
 V_0(K,W)=&(n+1)\mathrm{Vol}(W).
\end{align*}
If $\partial K$ is smooth,  we can express the mixed volumes in terms of the anisotropic curvature integrals
\begin{equation}\label{def-quermass}
   V_{n+1-k}(K,W)=\int_{\partial K}{E}_{k-1}d\mu_{\gamma},\quad k=1,\dots,n,
\end{equation}
where ${E}_{k-1}$ is the $(k-1)$st anisotropic mean curvature of $\partial K$ relative to $W$ and $d\mu_{\gamma}={\gamma}(\nu)d\mu$ denotes the anisotropic area form.

Moreover, the mixed volumes satisfy the following Alexandrov--Fenchel inequalities:
\begin{thm}[{see \cite[\S 7]{Schn}}]
For any convex body $K$ in $\mathbb{R}^{n+1}$, there holds
\begin{equation}\label{AF-00}
  V_{n+1-j}^{k-i}(K,W)~\geq~V_{n+1-i}^{k-j}(K,W)V_{n+1-k}^{j-i}(K,W)
\end{equation}
for all $0\leq i<j<k\leq n+1$.  In particular, for $k=n+1$, \eqref{AF-00} reduces to
\begin{equation}\label{AF-k=n+1}
  V_{n+1-j}^{n+1-i}(K,W)~\geq~V_{n+1-i}^{n+1-j}(K,W)V_0(K,W)^{j-i}
\end{equation}
for $0\leq i<j<n+1$. Furthermore, the equality of \eqref{AF-k=n+1} holds if and only if $K$ is a scaled translate of $W$.
\end{thm}
%Another special case of \eqref{AF-00} is
%\begin{equation}\label{AF-1}
%  V_{n+1-k}^2(\Omega,W)~\geq~V_{n-k}(\Omega,W)V_{n+2-k}(\Omega,W),\quad k=1,\cdots,n.
%\end{equation}
%When $k=n$, the equality holds in \eqref{AF-1} if and only if $\Omega$ is a scaled translate of $W$.
We note that the inequality \eqref{AF-k=n+1} implies that the isoperimetric ratio $\mathcal{I}_\ell(K,W)$ defined in \eqref{s1:Iso-def} satisfies
\begin{equation}\label{s2:Iso}
\mathcal{I}_\ell(K,W)\geq 1,\quad 1\leq \ell\leq n,
\end{equation}
with the equality holding if and only if $K$ is homothetic to $W$.

\section{Anisotropic curvature measures}\label{sec:ACM}
In this section, we introduce the anisotropic curvature measures for convex bodies in $\R^{n+1}$. For that purpose, we first define the generalized anisotropic principal curvatures on the unit anisotropic normal bundle of a convex body $K$. Then we calculate the volume of the local $\varepsilon$-parallel set of $K$, and express it as a polynomial of degree $n+1$ in $\varepsilon$. The anisotropic curvature measures will be defined as the coefficients of the expression. At the end of this section, we will prove that such defined anisotropic curvature measures are weakly continuous with respect to the Hausdorff distance.

\subsection{Generalized anisotropic principal curvatures}\label{sec:3-1}

Let $W$ be a Wulff shape in $\R^{n+1}$ with the support function ${\gamma}\in C^{\infty}(\mathbb{S}^n)$.  Associated with $W$,  the anisotropic distance function $d_W(\cdot,\cdot)$ in $\mathbb{R}^{n+1}$ is defined by \eqref{s2:dW}. Let $K$ be a convex body in $\R^{n+1}$. For any $y\in \mathbb{R}^{n+1}$ the anisotropic distance of $y$ to $K$ is defined by
\begin{equation*}
d_W(K,y)=\inf\{d_W(x,y):x\in K\}.
\end{equation*}
Since $K$ is convex,  for any $p\in \R^{n+1}$ there exists a unique $p^*\in K$ so that $d_W(K,p)=d_W(p^*,p)$. If $p\in \mathbb{R}^{n+1}\setminus K$, the nearest point $p^*$ lies on the boundary of $K$.
\begin{defn}
For any $p\in \mathbb{R}^{n+1}$, we call $f_K(p):=p^*$ the \emph{anisotropic metric projection} of $p$ on $K$ with respect to $W$. For any $p\in \mathbb{R}^{n+1}\setminus K$, we denote by
\begin{equation*}
  v_K(p):=\dfrac{p-f_K(p)}{\gamma^0(p-f_K(p))}
\end{equation*}
the vector pointing from the nearest point $f_K(p)$ to $p$ with unit $\gamma^0$-norm.  Then $v_K(p)$ lies on the boundary $\Sigma$ of the Wulff shape $W$.
\end{defn}

\begin{defn}
Define $F_K:\R^{n+1}\setminus K\rightarrow T\R^{n+1}$ by
\begin{equation}
F_K(p):=\left(f_K(p),v_K(p)\right)~\in \mathbb{R}^{n+1}\times\Sigma.
\end{equation}
We call the image $F_K(\R^{n+1}\setminus K)=:\mathcal{N}_K$ the unit anisotropic normal bundle of $K$.
%Sometimes by $F_K(p)$ we only mean the vector part $(p-f_K(p))/\gamma^0(p-f_K(p))$ whenever no confusion arises.
\end{defn}

We have the following analytic properties for the maps $f_K$ and $F_K$.
\begin{lem}\label{s3:1-lem1}
The function $f_K$ is Lipschitz continuous, and so is $F_K$.
\end{lem}
\proof
Take any two points $p,q\in \R^{n+1}$ with $p\neq q$. There are three cases to be considered.

Case 1: $p\in K$ and $q\in K$. In this case we have
\begin{equation*}
|f_K(p)-f_K(q)|=|p-q|\leq |p-q|.
\end{equation*}

Case 2: $p\in K$ and $q\in K^c=\mathbb{R}^{n+1}\setminus K$. So $f_K(p)=p\in K$ and by the definition of $f_K$ we have
\begin{equation*}
d_W(f_K(q),q)\leq d_W(p,q).
\end{equation*}
Since $\gamma$ is a smooth positive function on $\mathbb{S}^n$, the anisotropic distance function $d_W$ in $\mathbb{R}^{n+1}$ is equivalent to the Euclidean distance, i.e., there exists a positive constant $C$ depending only on $\gamma$ such that
\begin{equation*}
  \frac 1C|x-y|\leq d_W(x,y)\leq C|x-y|
\end{equation*}
for any two points $x,y\in \mathbb{R}^{n+1}$.  Therefore, we have
\begin{align*}
|f_K(p)-f_K(q)|& =|f_K(q)-f_K(p)|\\
&\leq Cd_W(f_K(q),f_K(p))\\
&\leq C(d_W(f_K(q),q)+d_W(q,f_K(p)))\\
&\leq C(d_W(p,q)+d_W(q,p))\\
&\leq 2C^2|p-q|.
\end{align*}
\begin{figure}
\centering
\begin{tikzpicture}[scale=1.2]
\path[fill=green!5,draw]
(-3,-3.4) [rounded corners=10pt] -- (-3.6,-2.6) -- (-3.5,-1.4)
 [rounded corners=12pt]--(-2.8,-0.5)
 [rounded corners=12pt]--(-1.7,-0.6)
 [rounded corners=12pt]--(-0.3,-1.2)
 [rounded corners=9pt]--(-0.1,-1.9)
  [rounded corners=12pt]--(-0.4,-2.7)--(-1.9,-3.6)
  [rounded corners=10pt]--cycle;
\node at (-1.2,-2.6) {$K$};
\draw[blue,thick] (-4.5,-0.4) circle (0.02) node[left] {$p$};
\draw[dashed] (-4.5,-0.4) -- (-3.5,-1.6);
\draw[blue,thick] (-3.5,-1.6) circle (0.02) node[left] {$f_K(p)$};
\draw[blue,thick] (-1.35,0.6) circle (0.02) node[left] {$q$};
\draw[dashed] (-1.35,0.6) -- (-1.75,-0.62);
\draw[blue,thick] (-1.75,-0.62) circle (0.02) node[right] {$f_K(q)$};
\draw[red](-3.5,-1.6) -- (-1.75,-0.62);
\node at (-2.3,-1.2) {$\rho(t)$};
\draw[red,thick] (-2.66,-1.13) circle (0.02);
\draw[dashed, red] (-4.5,-0.4) -- (-2.66,-1.13);
\end{tikzpicture}
\caption{Lipschitz continuity of $f_K$}%\label{}
\end{figure}
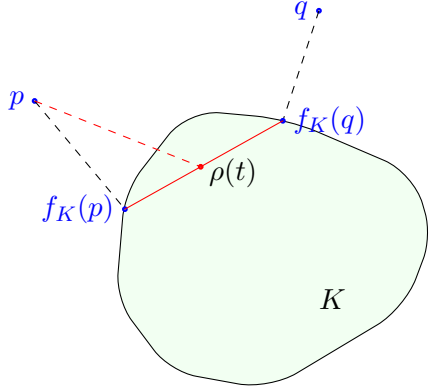

Case 3: $p\in K^c$ and $q\in K^c$. In this case we adapt the idea in the proof of  \cite[Theorem~1.2.1]{Schn} and consider the line segment from $f_K(p)$ to $f_K(q)$, which is given by
\begin{equation*}
\rho(t)=f_K(p)+te,
\end{equation*}
where $e=(f_K(q)-f_K(p))/|f_K(q)-f_K(p)|$ is the unit vector pointing from $f_K(p)$ to $f_K(q)$ and $0\leq t\leq |f_K(q)-f_K(p)|$. Since $K$ is convex, $\rho(t)$ lies in $K$ for all $0\leq t\leq |f_K(q)-f_K(p)|$. The definitions of $f_K$ and $d_W$ imply that $\gamma^0(p-f_K(p))\leq \gamma^0(p-\rho(t))$ for $0\leq t\leq |f_K(q)-f_K(p)|$. It follows that
\begin{align*}
0&\leq  \frac{d}{dt}\bigg|_{t=0} ( \gamma^0(p-\rho(t)))=-\langle D\gamma^0(p-f_{K}(p)),e\rangle.
\end{align*}
Similarly, we have
\begin{align*}
0&\geq  \frac{d}{dt}\bigg|_{t=|f_K(q)-f_K(p)|} ( \gamma^0(q-\rho(t)))=-\langle D\gamma^0(q-f_{K}(q)),e\rangle.
\end{align*}

For a fixed unit vector $e\in \mathbb{S}^n$, let us consider the set
\begin{equation*}
C(e):=\{\xi\in \R^{n+1}\setminus \{0\}: \langle D\gamma^0(\xi),e\rangle \geq 0\}.
\end{equation*}
Since $D\gamma^0(t\xi)=D\gamma^0(\xi)$ for any $t>0$, we know that $C(e)$ is a cone and so it is determined by its subset $C(e)\cap \pt W$. Then it is easy to see that $C(e)\cap \pt W$ consists of all the points $\xi \in \pt W$ such that the isotropic normal $\nu(\xi)$ ($=D\gamma^0(\xi)$) of $\pt W$ at $\xi$ satisfies $\langle \nu(\xi),e\rangle\geq 0$. It follows that the cone $C(e)$ is determined only by $e$. (Note that generally $C(e)$ is not a half-space.) Moreover, $e$ lies in the interior of $C(e)$ since $\langle D\gamma^0(e),e\rangle =\gamma^0(e)>0$.

Back to our setting, we see that $q$ lies in the set $f_K(q)+C(e)$ with vertex $f_K(q)$, while $p$ lies in the closure of the set $f_K(p)+C(e)^c$. Since the boundary $Bd(q,e)$ of $f_K(q)+C(e)$ is the translation of the boundary $Bd(p,e)$ of $f_K(p)+C(e)^c$ along the line segment $\rho(t)$, $0\leq t\leq |f_K(q)-f_K(p)|$, there exists a constant $c_0=c_0(e)$ such that
\begin{equation*}
|f_K(p)-f_K(q)|= c_0(e) \mathrm{dist}(Bd(p,e),Bd(q,e)).
\end{equation*}
Noting that $\mathrm{dist}(Bd(p,e),Bd(q,e))\leq |p-q|$, we derive
\begin{equation*}
|f_K(p)-f_K(q)|\leq c_0(e) |p-q|.
\end{equation*}
Finally since $e\in \SS^n$, the Lipschitz constant of $f_K$ in this case can be taken as $\max_{e\in \SS^n} c_0(e)$.

In summary, $f_K$ is Lipschitz continuous with Lipschitz constant $\max\{1,2C^2,\max_{e\in \SS^n} c_0(e)\}$.
%Following this, $F_A: p\mapsto (f_A(p),\frac{p-f_A(p)}{F^0(p-f_A(p))}) \subset \mathcal{N}_A$ is locally Lipschitz, because $q\mapsto\frac{q}{F^0(q)}$ is locally Lipschitz. In fact, it is locally bi-Lipschitz.
\endproof

We can equip the unit anisotropic normal bundle $\mathcal{N}_K$ of a convex body $K$ in $\mathbb{R}^{n+1}$ with a new metric $\tilde{G}((p,v(p)))(\cdot,\cdot)$ for $(p,v(p))\in \mathcal{N}_K$, using the formulation in \cite{And01,Xia13}.
\begin{defn}
Let $(p,v(p))\in T\R^{n+1}\setminus \{0\}$ be any element such that $p\in \R^{n+1}$ and $v(p)\in T_p\R^{n+1}\setminus \{0\}$. Then any element in $T_{(p,v(p))}(T\R^{n+1}\setminus \{0\})$ is a pair of vectors $(u,w)$. We introduce a metric $\tilde{G}$ on $T\R^{n+1}\setminus \{0\}$ such that
\begin{equation}
\tilde{G}((p,v(p)))((u,w),(u,w))=G(v(p))(u,u)+G(v(p))(w,w),
\end{equation}
where $G(v(p))(\cdot,\cdot)$  is the metric defined in \eqref{s2:G}. Note that $\mathcal{N}_K\subset T\R^{n+1}\setminus \{0\}$. Then $\mathcal{N}_K$ will be equipped with the metric induced from $\tilde{G}$.
\end{defn}

\begin{rem}
In the isotropic case (i.e., ${\gamma}^0(v)^2=|v|^2_{\delta_{ij}}$), the metric $\tilde{G}$ on $T\R^{n+1}\setminus \{0\}$ is well-known and called Sasaki metric.
\end{rem}

Let $K$ be a convex body in $\mathbb{R}^{n+1}$. For $\varepsilon>0$, let
\begin{equation*}
  K^\varepsilon=\{x\in \mathbb{R}^{n+1}:~d_W(K,x)\leq \varepsilon\}
\end{equation*}
be the anisotropic parallel body of $K$ with anisotropic distance $\varepsilon$. Then $\pt K^\varepsilon$ is a $C^{1,1}$ hypersurface. In particular, it is $C^2$ a.e., and the anisotropic principal curvatures $\kappa_i(p)$ of $\pt K^\varepsilon$ are well defined for almost every point $p\in \pt K^\varepsilon$. For $p\in \partial K^{\varepsilon}$, we have
\begin{equation*}
  p=f_K(p)+\varepsilon v_K(p),\quad v_{\pt K^{\varepsilon}}(p)=v_K(p),
\end{equation*}
where $v_{\pt K^{\varepsilon}}(p)$ denotes the anisotropic normal of $\pt K^\varepsilon$ at $p$. Since $f_K$ is locally Lipschitz, the analytical properties of $\mathcal{N}_K$ can be derived as in \cite{Wal76}. The restriction of the map $F_K$ on the boundary of $K^{\varepsilon}$
\begin{equation*}
  F_K|_{\pt K^\varepsilon}:\pt K^\varepsilon\rightarrow \mathcal{N}_K
\end{equation*}
is a bi-Lipschitz homeomorphism. This follows from the fact that the inverse of $F_K|_{\partial K^{\varepsilon}}$ is given by $(F_K|_{\partial K^{\varepsilon}})^{-1} (x,v)=x+\varepsilon v\in \partial K^{\varepsilon}$. It follows that $\mathcal{N}_K$ is a Lipschitz submanifold of $T\R^{n+1}$ of dimension $n$.

Since the anisotropic normal vector of a point of the boundary $\partial K$ is in general not unique, we adapt the idea in \cite{Wal76,Zah86,Koh91} in the isotropic case to define the generalized anisotropic principal curvatures associated with $K$ living on the anisotropic normal bundle $\mathcal{N}_K$ rather than on $\partial K$. Since $f_K$ and $F_K$ are Lipschitz maps, they are differentiable almost everywhere on $\mathbb{R}^{n+1}\setminus K$, by Rademacher's theorem. Let $\mathscr{D}_K\subset \mathbb{R}^{n+1}\setminus K$ be the set of points where $f_K$ and $F_K$ are differentiable. Note that if $y\in \mathscr{D}_K$, then the whole ray $\{f_K(y)+\rho v_K(y):~\rho>0\}$ belongs to $\mathscr{D}_K$. It follows that there exists a set $\tilde{\mathcal{N}}_K\subset \mathcal{N}_K$ with $\mathcal{H}^{n}(\mathcal{N}_K\setminus\tilde{\mathcal{N}}_K)=0$ such that $v_K(\cdot)$ is differentiable at $x+\rho v$ for all $(x,v)\in \tilde{\mathcal{N}}_K$ and all $\rho>0$. We fix a point $(x,v)\in \tilde{\mathcal{N}}_K$ and consider the point $p=x+\varepsilon v\in \partial K^{\varepsilon}$.  The anisotropic principal curvatures of $\pt K^{\varepsilon}$ are well defined at the point $p\in \pt K^{\varepsilon}$ for any $\varepsilon>0$. Consider the line segment $\sigma:[0,\varepsilon]\to \overline{f_K(p)p}$ connecting $\sigma(0)=f_K(p)=x$ and $\sigma(\varepsilon)=p$, which can be written as $\sigma(t)=p-(\varepsilon-t)v$. Since
\begin{equation*}
  v_{\pt K^t}(\sigma(t))=v_{\pt K^{\varepsilon}}(\sigma(\varepsilon)),\quad \forall~t\in (0,\varepsilon),
\end{equation*}
a differentiation gives
\begin{equation}\label{s3:1-b}
  dv_{\pt K^t}\big|_{\sigma(t)}\big(\left(1+(t-\varepsilon)\kappa_i(p)\right)e_i\big)=\kappa_i(p)e_i,\quad i=1,\dots,n,
\end{equation}
where $\{e_i, i=1,\dots,n\}$ denotes an orthonormal basis of the anisotropic Weingarten map at $p\in \pt K^\varepsilon$ with respect to the induced metric from $G$. The equation \eqref{s3:1-b} means that the parallel displacement of $e_i$ along $\sigma(t)$ yields an orthonormal basis of the anisotropic Weingarten map of $\partial K^t$ at $\sigma(t)$ for $t\in (0,\varepsilon)$, and the corresponding eigenvalues at $\sigma(t)$ are related to $\kappa_i(p)$ by
\begin{equation}\label{s3:1-k}
  \kappa_i(\sigma(t))=\frac{\kappa_i(p)}{1+(t-\varepsilon)\kappa_i(p)}.
\end{equation}
If $\kappa_i(p)\neq 1/{\varepsilon}$, the equation is equivalent to
\begin{equation}\label{s3:1-k2}
  \frac{\kappa_i(\sigma(t))}{1-t\kappa_i(\sigma(t))}=\frac{\kappa_i(p)}{1-\varepsilon\kappa_i(p)}.
\end{equation}
Since both sides of \eqref{s3:1-k2} do not depend on $t$, we can define their value as $\kappa_i(x,v)$ for $(x,v)\in \tilde{\mathcal{N}}_K$, which we call the generalized anisotropic principal curvatures of $K$.
\begin{defn}\label{s3:1-def1}
Let $K$ be a convex body in $\mathbb{R}^{n+1}$. We define the generalized anisotropic principal curvatures of $K$ at $(x,v)\in \tilde{\mathcal{N}}_K\subset \mathcal{N}_K$ by
\begin{equation}\label{s3:kap-v}
  \kappa_i(x,v):=\left\{\begin{array}{cc}
                          \dfrac{\kappa_i(p)}{1-\varepsilon\kappa_i(p)}, &  \mathrm{if}~\kappa_i(p)\neq 1/\varepsilon, \\
                          +\infty, & \mathrm{if}~\kappa_i(p)=1/\varepsilon ,
                        \end{array}\right.
\end{equation}
where $p=x+\varepsilon v\in \partial K^{\varepsilon}$. Note that the right-hand side of \eqref{s3:kap-v} does not depend on $\varepsilon$.
\end{defn}
\begin{rem}
Alternatively, we have
\begin{equation*}
\kappa_i(x,v):=\lim_{t\rightarrow 0+}\kappa_i(\sigma(t))\in \R\cup \{+\infty\},
\end{equation*}
where $\sigma(t)=x+tv\in \partial K^t$. If $\partial K$ is $C^2$ then $\kappa_i(x,v)$ coincide with the ordinary anisotropic principal curvatures of $\partial K$. In the following, without loss of generality, we will write
\begin{equation*}
  \kappa_i(v)=\kappa_i(x,v)
\end{equation*}
and write $v\in \mathcal{N}_K$  instead of $(x,v)\in \mathcal{N}_K$ for the simplicity of the notation. If the base point of $v$ needs to be emphasized, we use the notation $x=\Pi(v)\in\partial K$ to denote the base point of $v$.
\end{rem}

Relabel $\kappa_i(v)$ such that $\kappa_1(v)\geq \cdots \geq \kappa_n(v)$. For integer $r=0,1,\dots,n$, define
\begin{equation*}
\tilde{\mathcal{N}}_{K,r}:=\{\;v\in \tilde{\mathcal{N}}_K\;|\; \kappa_1(v)=\cdots=\kappa_r(v)=+\infty\text{ and }\kappa_{r+1}(v),\dots,\kappa_n(v)<+\infty\;\}.
\end{equation*}
%\end{rem}

The following proposition can be derived using the equations \eqref{s3:1-b} and \eqref{s3:1-k} as in Theorem 2.4 of \cite{Koh91}. We omit the proof here.
\begin{prop}\label{s3:prop1}
Let $v\in \tilde{\mathcal{N}}_{K}$ and $x=\Pi(v)\in \partial K$ be the base point of $v$. The following three statements are equivalent:
\begin{enumerate}
\item $v\in \tilde{\mathcal{N}}_{K,r}$.
\item For any $\varepsilon>0$, $1/\varepsilon$ is an anisotropic principal curvature of the hypersurface $\pt K^\varepsilon$ at $p=x+\varepsilon v$ with multiplicity $r$.
\item For any $\varepsilon>0$ and $p=x+\varepsilon v$, we have $\mathrm{rank}((f_K)_*|_p)=n-r$.
\end{enumerate}
\end{prop}
Using Proposition \ref{s3:prop1}, we have the following characterization for the set $\tilde{\mathcal{N}}_{K,0}$.
\begin{prop}\label{prop-regular-points}
Define the following two sets
\begin{equation*}
\widetilde{\pt K}:=\{x\in \pt K: \pt K \text{ is twice differentiable at }x\},
\end{equation*}
and
\begin{equation*}
\widehat{\pt K}:=\{x\in \pt K: \text{the unit anisotropic normal is unique at }x\}.
\end{equation*}
Then $\widetilde{\pt K}\subset \Pi (\tilde{\mathcal{N}}_{K,0})\subset \widehat{\pt K}$ and $\Pi^{-1}(\widetilde{\pt K})\cap \tilde{\mathcal{N}}_K\subset \tilde{\mathcal{N}}_{K,0}$. It follows that $\widetilde{\pt K} \cap \Pi(\bigcup_{r=1}^n \tilde{\mathcal{N}}_{K,r})=\emptyset$.
\end{prop}
\begin{proof}
First, for any $x\in \widetilde{\pt K}$ the unit anisotropic normal $v(x)$ at $x$ is unique and the anisotropic principal curvatures $\kappa_i(x)$ of $\pt K$ at $x$ are finite. So $\kappa_i(v(x))$ are all finite. Therefore $\widetilde{\pt K}\subset \Pi (\tilde{\mathcal{N}}_{K,0})$.

Second, take any $v\in \tilde{\mathcal{N}}_{K,0}$. By Proposition \ref{s3:prop1}, for any $\varepsilon>0$ and $p=\Pi(v)+\varepsilon v$, we have $\mathrm{rank}((f_K)_*|_p)=n$. So there exist $n$ linearly independent tangent vectors for $\pt K$ at $f_K(p)=\Pi(v)$. It follows that at $\Pi(v)$ the unit anisotropic normal is unique. So $\Pi (\tilde{\mathcal{N}}_{K,0})\subset \widehat{\pt K}$. The other statements follow immediately. %So we finish the proof of the proposition.
\end{proof}

\subsection{Volumes of anisotropic parallel sets}\label{sec:3-2}$\ $

Let $\beta\subset \R^{n+1}$ be a bounded Borel set. Then the preimage of $\beta$ under the projection $f_K$ is
\begin{equation*}
f_K^{-1}(\beta)=(\beta\cap K)\cup f_K^{-1}(\beta\cap \pt K).
\end{equation*}
For $\varepsilon>0$, define
\begin{align*}
\beta_\varepsilon&:=f_K^{-1}(\beta)\cap \pt K^\varepsilon,\quad \mathrm{and} \quad \tilde{\beta}_\varepsilon:=F_K(\beta_\varepsilon)\in \mathcal{N}_K.
\end{align*}
Since $\pt K^\varepsilon$ is regular, i.e., the support hyperplane is unique at any point of $\partial K^{\varepsilon}$, there exists a unique isotropic normal vector $\nu$ at any point of $\pt K^\varepsilon$. Then the anisotropic normal vector on $\pt K^\varepsilon$ is given by $v=D\gamma (\nu)$. We can equivalently express the isotropic normal vector $\nu$ in terms of the anisotropic normal vector $v$ by
 \begin{equation}\label{s3:nu}
  \nu~=~\frac{D{\gamma}^0(v)}{|D{\gamma}^0(v)|_{\delta_{ij}}}.
\end{equation}
On $\pt K^\varepsilon$, we use the induced metric $\hat{g}=G(v)(\cdot,\cdot)$ from $G$ and let $d\mu_\gamma=\gamma(\nu)d\mu$ be the anisotropic area form on $\partial K^{\varepsilon}$, where $d\mu$ the isotropic area form on $\pt K^\varepsilon$ with respect to the metric induced from the Euclidean metric.
%In other words, $\mu$ is the $n$-dimensional Hausdorff measure on $\beta_\varepsilon$ with respect to the Euclidean metric; while $\mu_\gamma$ the $n$-dimensional Hausdorff measure with respect to the new metric $G$.

Let $v(p)$ denote the unit outward anisotropic normal of $\pt K^\varepsilon$ at $p$. We have
\begin{equation*}
F_K(p)=(f_K(p),v_K(p))=(p-\varepsilon v(p),v(p))~\in~\mathcal{N}_K.
\end{equation*}
For any $(x,v)\in \tilde{\mathcal{N}}_K$, we calculate the differential of the map $F_K$ at the point $p=x+\varepsilon v\in \partial K^{\varepsilon}$. Let $\{e_i\}_{i=1}^n$ be an orthonormal eigenbasis of the anisotropic Weingarten map on $T_p(\pt K^\varepsilon)$ with respect to the new metric $G(v)(\cdot,\cdot)$. We get
\begin{align*}
(F_K)_*(e_i)=&(e_i-\varepsilon v_*(e_i), v_*(e_i))\\
=& \left((1-\varepsilon \kappa_i(p))e_i, \kappa_i(p)e_i\right),
\end{align*}
where $\kappa_i(p)$ are the corresponding anisotropic principal curvatures of $\partial K^{\varepsilon}$ at $p$ which exist as we assumed that $(x,v)\in\tilde{\mathcal{N}}_K$. By definition \eqref{s3:kap-v} of $\kappa_i(v)$, we have
\begin{equation}\label{s3:2-kap}
\kappa_i(p)=\frac{\kappa_i(v)}{1+\varepsilon \kappa_i(v)}.
\end{equation}
The equation \eqref{s3:2-kap} is true algebraically even when $\kappa_i(v)=+\infty$ for some $i$. Then
\begin{align}\label{s3:2-F*}
(F_K)_*(e_i)=& \left(\left(1-\frac{\varepsilon\kappa_i(v)}{1+\varepsilon \kappa_i(v)}\right)e_i, \frac{\kappa_i(v)}{1+\varepsilon \kappa_i(v)}e_i\right).
\end{align}
Therefore, the Jacobi determinant of the map $F_K|_{\pt K^\varepsilon}$ is given by
\begin{equation}\label{s3:2-Jn2}
J_n(F_K|_{\pt K^\varepsilon})=\prod_{i=1}^n \frac{(1+\kappa_i^2(v))^{1/2}}{1+\varepsilon \kappa_i(v)}
\end{equation}
for $v\in \tilde{\mathcal{N}}_K$. On $\mathcal{N}_K$ we use the $n$-dimensional Hausdorff measure $\tilde{\mathcal{H}}^n$ induced from the metric  $\tilde{G}$ on $T\R^{n+1}$. This is comparable to the Euclidean Hausdorff measure $\mathcal{H}^n$, as the metric $G(v)(\cdot,\cdot)$ is comparable to the Euclidean metric. Therefore, we also have $\tilde{\mathcal{H}}^n(\mathcal{N}_K\setminus\tilde{\mathcal{N}}_K)=0$.

As introduced in Section \ref{sec:1}, we define the local $\varepsilon$-parallel set $A_{\varepsilon}(K,\beta)$ of a convex body $K$  by
\begin{equation}\label{s3:A}
  A_{\varepsilon}(K,\beta)=\{x\in \mathbb{R}^{n+1}|~0<d_W(K,x)\leq \varepsilon,~f_K(x)\in\beta\}.
\end{equation}
This is equal to the set $f_K^{-1}(\beta)\cap (K^\varepsilon \setminus K)$. In the following, we compute the $(n+1)$-dimensional Lebesgue measure $\mathcal{L}^{n+1}( A_{\varepsilon}(K,\beta))$. Since $\partial K^{\varepsilon}$ is the set of the points in $\mathbb{R}^{n+1}\setminus K$ having distance $\varepsilon$ to the boundary $\partial K$,  we can write
\begin{equation*}
  \partial K^{\varepsilon}=\{x\in \mathbb{R}^{n+1}\setminus K|~u(x):=d_W(\partial K,x)=\varepsilon\}.
\end{equation*}
We have shown in Lemma \ref{s3:1-lem1} that the projection map $f_K$ from $\mathbb{R}^{n+1}\setminus K$ to $\partial K$ is Lipschitz, so the function $u$ is a Lipschitz function. By Federer's coarea formula \cite{Fed59} for Lipschitz functions, we have
\begin{align*}
  \mathcal{L}^{n+1}( A_{\varepsilon}(K,\beta))=&\mathcal{L}^{n+1}(f_K^{-1}(\beta)\cap (K^\varepsilon \setminus K))\nonumber\\
  =& \int_0^{\varepsilon}\int_{u^{-1}(t)\cap f^{-1}_K(\beta)} \frac 1{|Du|}d\mu(p)dt.
\end{align*}
At any point $p\in u^{-1}(t)=\partial K^t$ for $0<t< \varepsilon$, we have for $|\tau|$ small,
\begin{equation*}
  u(p+\tau v(p))=t+\tau.
\end{equation*}
Differentiating this equation with respect to $\tau$ gives that
\begin{equation*}%\label{s5:Du}
  1=Du(p)\cdot v(p)= Du(p)\cdot\left(\gamma(\nu)\nu+\nabla^{\mathbb{S}}\gamma |_{\nu}\right)=|Du|\gamma(\nu),\quad \forall~p\in u^{-1}(t),
\end{equation*}
where $\nu$ is the isotropic unit normal vector which exists and is unique as $\partial K^t$ has only regular points. Hence $|Du|=1/{\gamma(\nu)}$. Consequently,
\begin{align}\label{s3:Vol-1}
\mathcal{L}^{n+1}( A_{\varepsilon}(K,\beta))=&\mathcal{L}^{n+1}(f_K^{-1}(\beta)\cap (K^\varepsilon \setminus K))\nonumber\\
=&\int_0^{\varepsilon}\left(\int_{\beta_{t}}\gamma(\nu)d\mu(p)\right)dt\nonumber\\
=&\int_0^{\varepsilon}\left(\int_{\beta_{t}\cap \mathscr{D}_K}\gamma(\nu)d\mu(p)\right)dt\nonumber\\
=&\int_0^\varepsilon \left(\int_{F_K({\beta}_{t})\cap \tilde{\mathcal{N}}_K}(J_n(F_K|_{\pt K^t}))^{-1}d\tilde{\mathcal{H}}^nv\right)dt\nonumber\\
=&\int_0^\varepsilon \left(\int_{\Pi^{-1}(\beta\cap\partial K)\cap \tilde{\mathcal{N}}_K}  \prod_{i=1}^n \frac{1+t \kappa_i(v)}{(1+\kappa_i^2(v))^{1/2}} d\tilde{\mathcal{H}}^nv\right)dt,
\end{align}
where in the third equality we used the fact that $\mathscr{D}_K$ has full measure, and in the fourth equality we applied the generalized formula for the change of variables for Lipschitz maps (see, e.g., \cite[Thm.~3.2.5]{Fed69} and \cite[Sec.~3.2.46]{Fed69}).
\begin{figure}
  \centering
 \begin{tikzpicture}
\begin{scope}[even odd rule]
\clip (-3,-3.4) [rounded corners=10pt] -- (-3.6,-2.6) -- (-3.5,-1.4)
 [rounded corners=12pt]--(-2.8,-0.5)
 [rounded corners=12pt]--(-1.7,-0.6)
 [rounded corners=12pt]--(-0.3,-1.2)
 [rounded corners=9pt]--(-0.1,-1.9)
  [rounded corners=12pt]--(-0.4,-2.7)--(-1.9,-3.6)
  [rounded corners=10pt]--cycle (-3,-3.9) [rounded corners=10pt] -- (-3.9,-2.9) --(-4.1,-1.7)-- (-3.6,-0.8)
 [rounded corners=12pt]--(-2.8,-0.1)
 [rounded corners=12pt]--(-1.7,0)
 [rounded corners=12pt]--(-0.2,-0.5)
 [rounded corners=9pt]--(0.4,-1.9)
  [rounded corners=12pt]--(0.2,-2.7)--(-1,-3.8)
  [rounded corners=10pt]--cycle;
\path[fill=green!20,dashed,draw] (-1.75,0.2) -- (-1.9,-1) -- (-0.6,-1.2) -- (-0.3,-0.3) --cycle;
\end{scope}
\path[fill=green!5,draw]
(-3,-3.4) [rounded corners=10pt] -- (-3.6,-2.6) -- (-3.5,-1.4)
 [rounded corners=12pt]--(-2.8,-0.5)
 [rounded corners=12pt]--(-1.7,-0.6)
 [rounded corners=12pt]--(-0.3,-1.2)
 [rounded corners=9pt]--(-0.1,-1.9)
  [rounded corners=12pt]--(-0.4,-2.7)--(-1.9,-3.6)
  [rounded corners=10pt]--cycle;
\path[draw,dashed,thick,color=blue]
(-3,-3.9) [rounded corners=10pt] -- (-3.9,-2.9) --(-4.1,-1.7)-- (-3.6,-0.8)
 [rounded corners=12pt]--(-2.8,-0.1)
 [rounded corners=12pt]--(-1.7,0)
 [rounded corners=12pt]--(-0.2,-0.5)
 [rounded corners=9pt]--(0.4,-1.9)
  [rounded corners=12pt]--(0.2,-2.7)--(-1,-3.8)
  [rounded corners=10pt]--cycle;

\node at (-1.2,-2.9) {$K$};
\node at (0.8,-1.9) {$\partial K^{\varepsilon}$};
\draw[-stealth,dashed] (-1,-0.6) to [out=80,in=170] (0.5,0.4) node[right]{$A_{\varepsilon}(K,\beta)$};
\draw[-stealth,dashed] (-1.5,-0.7) to [out=-170,in=80] (-2,-1.4) node[below]{$\beta\cap\partial K$};
\end{tikzpicture}
 \caption{Local $\varepsilon$-parallel set $A_{\varepsilon}(K,\beta)$}%\label{}
\end{figure}
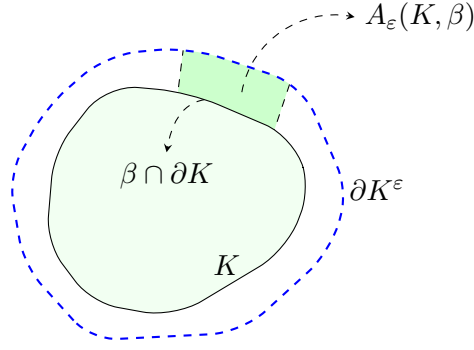

\subsection{Anisotropic curvature measures}$\ $

The formula \eqref{s3:Vol-1} allows us to define the anisotropic curvature measures of a convex body $K$ in $\mathbb{R}^{n+1}$. Let
\begin{align*}
\psi(v):=\prod_{i=1}^n (1+\kappa_i^2(v))^{-1/2}
\end{align*}
and $E_{r}(v)$ be the normalized $r$th elementary symmetric function of $\kappa_i(v)$.  The product $\prod_{i=1}^n (1+t \kappa_i(v))$ can be expanded as
\begin{equation*}
  \prod_{i=1}^n (1+t \kappa_i(v))=\sum_{r=0}^n \binom{n}{r}E_{n-r}(v)t^{n-r}.
\end{equation*}
Therefore by \eqref{s3:Vol-1},
\begin{align*}%\label{s3:Vol-1}
\mathcal{L}^{n+1}(A_{\varepsilon}(K,\beta))=&\sum_{r=0}^n \binom{n}{r} \int_0^\varepsilon \left(\int_{\Pi^{-1}(\beta\cap\partial K)\cap \tilde{\mathcal{N}}_K}  \psi(v)E_{n-r}(v)d\tilde{\mathcal{H}}^nv\right)t^{n-r}dt\\
=& \sum_{r=0}^n \binom{n}{r} \frac {\varepsilon^{n+1-r}}{n+1-r}\int_{\Pi^{-1}(\beta\cap\partial K)\cap \tilde{\mathcal{N}}_K}  \psi(v)E_{n-r}(v)d\tilde{\mathcal{H}}^nv\\
=& \frac 1{n+1}\sum_{r=0}^n \varepsilon^{n+1-r}\binom{n+1}{r}\int_{\Pi^{-1}(\beta\cap\partial K)\cap \tilde{\mathcal{N}}_K}  \psi(v)E_{n-r}(v)d\tilde{\mathcal{H}}^nv.
\end{align*}
\begin{rem}
As in \cite{Koh91}, the product $ \psi(v)E_{n-r}(v)$ needs to be understood algebraically. In other words, if $\kappa_i(v)=+\infty$ for some $1\leq i\leq n$,  we replace all such $\kappa_i(v)$ in $ \psi(v)E_{n-r}(v)$ by a finite $c>0$ and then take $c\rightarrow +\infty$.
\end{rem}
If we denote
\begin{align}\label{s3:psi-r}
\Phi_r(K;\beta)=& \int_{\Pi^{-1}(\beta\cap\partial K)\cap \tilde{\mathcal{N}}_K} \psi(v)E_{n-r}(v) d\tilde{\mathcal{H}}^nv,\quad 0\leq r\leq n,
\end{align}
we arrive at the following local Steiner type formula
\begin{equation}\label{s3:Ste}
\mathcal{L}^{n+1}(A_{\varepsilon}(K,\beta))=\frac 1{n+1}\sum_{r=0}^{n}\varepsilon^{n+1-r}\binom{n+1}{r}\Phi_r(K;\beta).
\end{equation}
By setting
\begin{equation*}
  \Phi_{n+1}(K;\beta)=(n+1)\mathcal{L}^{n+1}(K\cap \beta),
\end{equation*}
the formula \eqref{s3:Ste} is equivalent to
\begin{equation}\label{s3:Ste2}
\mathcal{L}^{n+1}(f_K^{-1}(\beta)\cap K^{\varepsilon})=\frac 1{n+1}\sum_{r=0}^{n+1}\varepsilon^{n+1-r}\binom{n+1}{r}\Phi_r(K;\beta).
\end{equation}
Since the $(n+1)$-dimensional Lebesgue measure coincides with the $(n+1)$-dimensional Hausdorff measure in the Euclidean space $\mathbb{R}^{n+1}$, the formula \eqref{s1:Sten} is equivalent to \eqref{s3:Ste}.
\begin{defn}
Let $W$ be a Wulff shape in $\mathbb{R}^{n+1}$ with  the support function $\gamma\in C^{\infty}(\mathbb{S}^n)$. We call $\Phi_r(K;\cdot)$ ($0\leq r\leq n+1$) the \emph{anisotropic curvature measures} of the convex body $K$ (with respect to the Wulff shape $W$). %The formula \eqref{s3:Ste} is called the local Steiner type formula.
\end{defn}

\begin{rem}
\begin{itemize}
  \item[(i)] Comparing \eqref{s3:Ste2} with \eqref{s2:Ste} we can  see that
  \begin{equation*}
    \Phi_r(K;\mathbb{R}^{n+1})=V_r(K,W).
  \end{equation*}
  That is, the anisotropic curvature measures of a convex body is the localized version of the mixed volumes introduced in  \S \ref{sec:2-3}.
  \item[(ii)]When $\pt K$ is $C^2$, then $\mathcal{N}_K=\{(p,v(p)):p\in  \pt K\}$, $\kappa_i(v)=\kappa_i(p)$ and the Jacobi determinant of the map (with the definition naturally extended to $\pt K$)
  \begin{equation*}
    F_K:\pt K\rightarrow \mathcal{N}_K,\qquad F_K(p)=(p,v(p))
  \end{equation*}
  is $\psi(v)^{-1}$. Thus
\begin{equation*}
\psi(v)d\tilde{\mathcal{H}}^n v=d\mu_\gamma(p)=\gamma(\nu)d\mu(p)
\end{equation*}
and
\begin{equation}\label{s3:Phi-s}
  \Phi_r(K;\beta) = \int_{\partial K\cap \beta} E_{n-r}(p)\gamma(\nu)\,d \mu(p).
\end{equation}
This   implies the same formula as in (i),
\begin{equation*}
  \Phi_r(K;\mathbb{R}^{n+1})= \int_{\partial K}E_{n-r}\gamma(\nu)d\mu=V_r(K,W)
\end{equation*}
by \eqref{def-quermass} and means that the anisotropic curvature measures provide the local information of the mixed volumes of a convex body.
\end{itemize}
\end{rem}

In particular, for a general convex body $K$ in $\mathbb{R}^{n+1}$ we have the following geometric interpretation of the anisotropic curvature measures in the extreme cases $r=0$ and $r=n$.
\begin{prop}\label{s3:3-prop}
Let $W$ be a Wulff shape in $\mathbb{R}^{n+1}$ with the support function $\gamma\in C^{\infty}(\mathbb{S}^n)$ and $K$ be a convex body in $\mathbb{R}^{n+1}$. Then for any Borel set $\beta$ in $\mathbb{R}^{n+1}$,
\begin{align}
\Phi_0(K;\beta) =& \int_{\pi(\Pi^{-1}(\beta\cap\partial K)\cap \mathcal{N}_K)}d\mu_{\gamma},\label{s3:Phi0-0}\\
\Phi_n(K;\beta)=&\int_{\widetilde{\pt K}\cap \beta} {\gamma}(\nu) d\mu(p),\label{s3:Phin}
\end{align}
where $\pi: \mathcal{N}_K\to \Sigma=\partial W$ is the projection of $\mathcal{N}_K$ to the vector part given by $\pi(x,v)=v$ and $\widetilde{\pt K}$ is given in Proposition~\ref{prop-regular-points}.
\end{prop}
\begin{proof}
For $r=0$, we consider the projection map $\pi: \mathcal{N}_K\to \Sigma=\partial W$. The equations \eqref{s3:2-F*} and \eqref{s3:2-Jn2} imply that the Jacobi determinant of the map $\pi$ is given by
\begin{equation*}
  J_n(\pi)= \prod_{i=1}^n\frac{\kappa_i(v)}{(1+\kappa_i^2(v))^{1/2}}.
\end{equation*}
By \eqref{s3:psi-r},
\begin{align*}
  \Phi_0(K;\beta)=&  \int_{\Pi^{-1}(\beta\cap\partial K)\cap \tilde{\mathcal{N}}_K}\psi(v)E_n(v)d\tilde{\mathcal{H}}^nv\\
  =& \int_{\Pi^{-1}(\beta\cap\partial K)\cap \tilde{\mathcal{N}}_K}\prod_{i=1}^n\frac{\kappa_i(v)}{(1+\kappa_i^2(v))^{1/2}}d\tilde{\mathcal{H}}^nv \\
  = &\int_{\pi(\Pi^{-1}(\beta\cap\partial K)\cap \tilde{\mathcal{N}}_K)}d\mu_{\gamma},
\end{align*}
where in the last equality we used the transformation for change of variables. Since $\tilde{\mathcal{H}}^n(\mathcal{N}_K\setminus \tilde{\mathcal{N}}_K)=0$, we obtain the formula \eqref{s3:Phi0-0}.

For $r=n$, by the definition of $\psi(v)$,
\begin{align*}
\Phi_n(K;\beta)&= \int_{\Pi^{-1}(\beta\cap\partial K)\cap \tilde{\mathcal{N}}_K}\psi(v) d\tilde{\mathcal{H}}^nv\\
&=\int_{\Pi^{-1}(\beta\cap\partial K)\cap \tilde{\mathcal{N}}_{K,0}} \psi(v) d\tilde{\mathcal{H}}^nv\\
&=\int_{\beta\cap \Pi(\tilde{\mathcal{N}}_{K,0}) } {\gamma}(\nu) d\mu(p),
\end{align*}
where the last equality is due to the transformation from $\tilde{\mathcal{N}}_{K,0}$ to $\Pi(\tilde{\mathcal{N}}_{K,0})$. Since $\Pi(\tilde{\mathcal{N}}_{K,0}) \subset \widehat{\pt K}$ (Proposition \ref{prop-regular-points}), the transformation is admissible. Alternatively, we may first transform the integral over $\tilde{\mathcal{N}}_{K,0}$ onto $\pt K^\varepsilon$ and then let $\varepsilon\rightarrow 0+$. Finally, by Proposition \ref{prop-regular-points}, we have $\widetilde{\pt K}\subset \Pi(\tilde{\mathcal{N}}_{K,0})$. Since the $n$-dimensional Hausdorff measure of $\pt K\setminus \widetilde{\pt K}$ is equal to zero (see \cite[\S 2]{Schn}), we obtain the formula \eqref{s3:Phin}.
\end{proof}
\begin{rem}
As a consequence of Proposition \ref{s3:3-prop}, we also call $\Phi_n(K;\cdot)$ the \emph{boundary anisotropic curvature measure} of a convex body. Moreover, for any convex body $K$ in $\mathbb{R}^{n+1}$ we have
\begin{align}
  \Phi_0(K;\mathbb{R}^{n+1})=&\int_{\partial W}d\mu_\gamma=|\partial W|_{\gamma},\label{s3:Phi0}\\
  \Phi_n(K;\mathbb{R}^{n+1})=&\int_{\widetilde{\partial K}}d\mu_\gamma=|\partial K|_{\gamma}.
\end{align}
\end{rem}

\subsection{Weak continuity of the anisotropic curvature measures}$\ $

As in the classical setting of curvature measures for convex bodies in the Euclidean space, the anisotropic curvature measures introduced here are also weakly continuous in the Hausdorff topology.

For any two convex bodies $K, L$ in $\mathbb{R}^{n+1}$, the anisotropic Hausdorff distance between them is defined as
\begin{equation*}
d^\mathcal{H}_W(K,L)=\inf_{r>0} \{r:K+rW\subset L \text{ and } L+rW\subset K\}.
\end{equation*}
Note that the anisotropic and isotropic distance functions in $\mathbb{R}^{n+1}$ have the relation
\begin{equation*}
\frac{1}{C}d(x,y)\leq d_W(x,y)\leq Cd(x,y)
\end{equation*}
for some positive constant $C$ depending only on the Wulff shape $W$. So the anisotropic Hausdorff distance $d^\mathcal{H}_W$ and the isotropic Hausdorff distance $d^\mathcal{H}$ for sets are equivalent.
\begin{thm}\label{s3:thm-W}
Let $W$ be a Wulff shape in $\mathbb{R}^{n+1}$ with the support function $\gamma\in C^{\infty}(\mathbb{S}^n)$, and $\{K_j\}_{j=1}^{\infty}$ be a sequence of convex bodies in $\R^{n+1}$ such that $K_j\to K$ as $j\to \infty$ in the Hausdorff topology. Then for every $r=0,\dots,n$, the anisotropic curvature measure $\Phi_r(K_j;\cdot)$ converges weakly in the sense of measure to the anisotropic curvature measure $\Phi_r(K;\cdot)$.
\end{thm}

According to the Steiner type formula \eqref{s3:Ste}, Theorem \ref{s3:thm-W} is a direct consequence of the following lemma.
\begin{lem}\label{s3:lem-WM}
Let $W$ be a Wulff shape in $\mathbb{R}^{n+1}$ with the support function $\gamma\in C^{\infty}(\mathbb{S}^n)$, and let $K$ and $K_i$, $i\geq 1$ be convex bodies in $\R^{n+1}$. Suppose that
\begin{equation*}
  \lim_{i\to\infty}d_W^{\mathcal{H}}(K_i,K)=0.
\end{equation*}
Define $\mu_K^\varepsilon(\beta):=\mathcal{L}^{n+1}(f_K^{-1}(\beta)\cap K^{\varepsilon})$ for Borel set $\beta$, and similarly $\mu_{K_i}^\varepsilon(\beta):=\mathcal{L}^{n+1}(f_{K_i}^{-1}(\beta)\cap K_i^{\varepsilon})$. Then
\begin{equation*}
  \mu_{K_i}^\varepsilon \stackrel{w}\longrightarrow \mu_K^\varepsilon
\end{equation*}
as $i\to\infty$.
\end{lem}
\begin{proof}%[Proof of Lemma \ref{s3:lem-WM}]
The sets $f_K^{-1}(\beta)\cap K^{\varepsilon}$ and $f_{K_i}^{-1}(\beta)\cap K_i^{\varepsilon}$ are measurable since $f_K$ and $f_{K_i}$ are continuous. Then it suffices to prove:
\begin{itemize}
  \item[(1)] for an open set $\beta$ in $\mathbb{R}^{n+1}$, $\mu_{K}^\varepsilon(\beta)\leq \liminf_{i\rightarrow \infty} \mu_{K_i}^\varepsilon(\beta)$, and
  \item[(2)] $\lim_{i\rightarrow \infty}\mu_{K_i}^\varepsilon(\R^{n+1})=\mu_K^\varepsilon(\R^{n+1})$.
\end{itemize}

To prove (1), let $p\in \mathring{K}^\varepsilon$, the interior of $K^\varepsilon$, with $f_K(p)\in \beta$. First we claim $\lim_{i\rightarrow \infty}f_{K_i}(p)=f_K(p)$. Otherwise assume that a subsequence (still denoted by $f_{K_i}(p)$) converges to a point $p^*$ which is not $f_K(p)$. Note that
\begin{equation*}
d_W(f_{K_i}(p),p)=d_W(K_i,p)\leq d_W(K_i,K)+d_W(K,p),
\end{equation*}
which implies
\begin{equation*}
p\in (d_W(K_i,K)+d_W(K,p))W+f_{K_i}(p).
\end{equation*}
Letting $i\rightarrow \infty$ yields $p\in d_W(K,p)W+p^*$. On the other hand, since $f_{K_i}(p)\in K_i$ and $d_W(K_i,K)\rightarrow 0$, we know $p^*\in K$. So $d_W(K,p)$ is attained at two distinct points $f_K(p)$ and $p^*$, which is a contradiction.  Since $\lim_{i\rightarrow \infty}f_{K_i}(p)=f_K(p)$ and $\beta$ is open, we have $f_{K_i}(p)\in\beta$ for any fixed $q\in B\cap \pt K$ and any fixed $p\in f_K^{-1}(q)\cap \mathring{K}^\varepsilon$ and any large $i$. So we have proved the claim. It follows that
\begin{equation*}
f_K^{-1}(\beta)\cap \mathring{K}^{\varepsilon}\subset \lim_{i\rightarrow \infty}\bigcap_{m\geq i}f_{K_m}^{-1}(\beta)\cap K_m^{\varepsilon}.
\end{equation*}
Therefore,
\begin{equation*}
\mu_{K }^\varepsilon(\beta)\leq \liminf_{i\rightarrow \infty} \mu_{K_i}^\varepsilon(\beta),
\end{equation*}
which completes the proof of (1).

To prove (2), note that
\begin{equation*}
\mu_{K}^{\varepsilon-d_W^\mathcal{H}(K_i,K)}(\R^{n+1})\leq \mu_{K_i}^{\varepsilon}(\R^{n+1}) \leq \mu_{K}^{\varepsilon+d_W^\mathcal{H}(K_i,K)}(\R^{n+1})
\end{equation*}
for $i$ large enough. So (2) follows immediately.
\end{proof}

\section{Minkowski integral formulas}\label{sec:Mink}

The Minkowski integral formulas are important tools in the study of the differential geometry. For smooth closed hypersurfaces, the following anisotropic Minkowski integral formulas are proved in \cite{HL08} by He and Li.
\begin{thm}[\cite{HL08}]\label{s4:thm-M}
Let $W$ be a Wulff shape with the support function ${\gamma}\in C^\infty(\SS^n)$.  For any smooth closed orientable hypersurface $M$ in $\R^{n+1}$ with a unit isotropic normal $\nu$ and any integer $1\leq r\leq n$, we have
\begin{equation}\label{s4:MF-1}
\int_{M} (E_{r-1}(p)-S(p)E_r(p)){\gamma}(\nu)d\mu(p)=0,
\end{equation}
where $E_r$ denotes the anisotropic $r$th mean curvature of $M$ with respect to the Wulff shape $W$, and $S(p)=G(v)(v,p)$ is the anisotropic support function on $M$ with $v=D\gamma(\nu)$ denoting the unit anisotropic normal of $M$.
\end{thm}

In this section, we generalize Theorem \ref{s4:thm-M} to nonsmooth setting by adapting the idea in Kohlmann's work \cite{Koh94} for the isotropic case.
\begin{thm}\label{s4:thm-GM}
Let $W$ be a Wulff shape with the support function ${\gamma}\in C^\infty(\SS^n)$. For any convex body $K$ in $\R^{n+1}$ and $1\leq r\leq n$, we have
\begin{equation}\label{s4:eq-GM}
\int_{\mathcal{N}_K} \psi(v)(E_{r-1}(v)-S(v)E_r(v)) d\tilde{\mathcal{H}}^nv=0,
\end{equation}
where $S(v)$ is the anisotropic support function defined on $\mathcal{N}_K$ by
\begin{equation*}
S(v):=G(v)(v,\Pi(v))
\end{equation*}
for any $v\in \mathcal{N}_K$, and $E_r(v)$ is the generalized $r$th mean curvature of the convex body $K$.
\end{thm}
Note that our formulas live on the unit anisotropic normal bundle $\mathcal{N}_K$ rather than on the boundary $\partial K$ itself. Though the generalized principal curvatures $\kappa_i(v)$ are defined only at those points in $\tilde{\mathcal{N}}_K$, we can still write our formulas in the form \eqref{s4:eq-GM} as we have $\tilde{\mathcal{H}}^n(\mathcal{N}_K\setminus \tilde{\mathcal{N}}_K)=0$ and the integrand on this part can be assigned any finite value. Restricted to the smooth case, $S(v)=G(v)(v(p),p)=S(p)$ for $p\in \partial K$ is just the anisotropic support function defined on the boundary $\partial K$ that is used in Theorem \ref{s4:thm-M}.

\begin{proof}[Proof of Theorem~\ref{s4:thm-GM}]
For the proof we use the fact that any convex body $K$ in $\mathbb{R}^{n+1}$ can be approximated by a sequence of smooth convex bodies $\{K_i\}_{i=1}^\infty$ with respect to the Hausdorff distance (cf. \cite{BF87,Schn}). Then by \eqref{s4:MF-1} we have
\begin{equation}\label{s4:eq3}
\int_{\pt K_i^\varepsilon} (E_{r-1}(p)-S(p)E_r(p)){\gamma}(\nu)d\mu(p)=0
\end{equation}
for small $\varepsilon>0$, where $K_i^{\varepsilon}=\{x:~d_W(K_i,x)\leq \varepsilon\}$ is the anisotropic parallel body of $K_i$ with the anisotropic distance $\varepsilon$. Choose $K_i$ such that $d_W^\mathcal{H}(K_i,K)<{\varepsilon}/{i}$. Then
\begin{lem} \label{s4:lem3}
We have
\begin{equation*}
d^\mathcal{H}_W({K^\varepsilon_i},{K^\varepsilon})< \frac{\varepsilon}{i}.
\end{equation*}
\end{lem}
\begin{proof}
Since $K$ and $K_i$ are symmetric in the statement, it suffices to prove that for any $p\in {K^\varepsilon}$ we can find $q\in {K^\varepsilon_i}$ such that $d_W(q,p)<  \varepsilon/i$. To this end, first we can find $p'\in K$ with $d_W(p',p)\leq \varepsilon$. Then we can choose $q'\in K_i$ such that $d_W(q',p')<\varepsilon/i$. Therefore by the triangle inequality we get $d_W(q',p)<\varepsilon(1+i^{-1})$.

Now consider two cases: (1) If $d_W(q',p)\leq d_W(q',p')$, then we choose $q=q'$ ($\in K_i\subset {K^\varepsilon_i}$). So we have
\begin{equation*}
d_W(q,p)=d_W(q',p)\leq d_W(q',p')<\frac{\varepsilon}{i}.
\end{equation*}
(2) If $d_W(q',p)> d_W(q',p')$, we can choose $q$ on the segment $\overline{q'p}$ such that $d_W(q',q)=d_W(q',p)-d_W(q',p')$ ($\leq d_W(p',p)\leq \varepsilon $). Then we have
\begin{equation*}
d_W(q,p)=d_W(q',p)-d_W(q',q)=d_W(q',p')<\frac{\varepsilon}{i}.
\end{equation*}
Since $d_W(q',q)\leq \varepsilon$, we have $q\in K_i^{\varepsilon}$. This completes the proof of the lemma.
\end{proof}

It follows from Lemma \ref{s4:lem3} that
\begin{equation*}
K^{\varepsilon (1-\frac{1}{i})}\subset K^\varepsilon_i\subset K^{\varepsilon (1+\frac{1}{i})}.
\end{equation*}

For any $p\in \pt K^\varepsilon$, we consider the half-line $f_K(p)p$ starting from $f_K(p)$. Since $K_i^\varepsilon$ is convex, there exists a unique point $p_i(p)\in \pt K_i^\varepsilon$ on the half-line $f_K(p)p$. Denote by $v_{K_i^\varepsilon}(p_i)$ the unit outward anisotropic normal of $\pt K_i^\varepsilon$ at $p_i$. Then we have the following result.
\begin{lem}\label{s4:lem-conv}
In $T\R^{n+1}$, we have for fixed $p\in \pt K^\varepsilon$,
\begin{equation*}
(p_i(p),v_{K_i^\varepsilon}(p_i(p)))\longrightarrow (p,v_{K^\varepsilon}(p)),\text{ as } i\rightarrow \infty.
\end{equation*}
\end{lem}
\begin{proof}
First, since $K^{\varepsilon (1-\frac{1}{i})}\subset K^\varepsilon_i\subset K^{\varepsilon (1+\frac{1}{i})}$, we have
\begin{equation*}
\varepsilon (1-\frac{1}{i})\leq d_W(f_K(p),p_i(p))\leq \varepsilon (1+\frac{1}{i}).
\end{equation*}
Letting $i\rightarrow \infty$ yields $d_W(f_K(p),p_i(p))\rightarrow \varepsilon$. So $p_i(p)\rightarrow p$ as $i\rightarrow \infty$.

Second, to show the convergence of the unit anisotropic normals, we consider the translated scaled Wulff shape
\begin{equation*}
  W_{\varepsilon(1-\frac{1}{i})}(f_K(p)):=\varepsilon(1-\frac{1}{i}) W+f_K(p),
\end{equation*}
and the cone $C_i$ generated by it with vertex $p_i(p)$
\begin{equation*}
C_i=\left\{t\:\overrightarrow{p_i(p)q}:t\geq 0, \quad q\in W_{\varepsilon(1-\frac{1}{i})}(f_K(p))\right\}.
\end{equation*}
As $i\rightarrow \infty$, the Wulff shape $W_{\varepsilon(1-\frac{1}{i})}(f_K(p))$ converges to $W_{\varepsilon}(f_K(p))$ which is contained in $K^\varepsilon$ and is tangential to $\pt K^\varepsilon$. Noting that $p_i(p)\rightarrow p$, we conclude that $C_i$ converges to $T$, one of the half-spaces bounded by $T_pK_\varepsilon$.

On the other hand, since $W_{\varepsilon(1-\frac{1}{i})}(f_K(p))$ is contained in the convex body $K_i^\varepsilon$, the cone $C_i$ is contained in $T_i$, one of the half-spaces determined by $T_{p_i(p)}\pt K_i^\varepsilon$. Taking $i\rightarrow \infty$, we can see that $T_i$ converges to a half-space and the limit half-space is $T$. Then we conclude that $v_{K_i^\varepsilon}(p_i(p))\rightarrow v_{K^\varepsilon}(p)$, as $i\rightarrow \infty$. So we finish the proof.

%On the other hand, using the metric $G(v_{K^\varepsilon_i}(p_i(p)))(\cdot,\cdot)$ on $T_{p_i(p)}\R^{n+1}$,  define the orthogonal cone $C_i^\perp$ of $C_i$ with the same vertex $p_i(p)$ by
%\begin{equation*}
%C_i^\perp=\left\{\;\overrightarrow{p_i(p)q'}:G(v_{K^\varepsilon_i}(p_i(p)))(\overrightarrow{p_i(p)q'},\overrightarrow{p_i(p)q})\leq 0,\; \quad \forall\; \overrightarrow{p_i(p)q}\in C_i \;\right\}.
%\end{equation*}
%By the argument above, we know that $C_i^\perp$ converges to the half-line $\{p+tv_{K^\varepsilon}(p), \: t\geq 0\}$. Note that $v_{K^\varepsilon_i}(p_i(p))\in C_i^\perp$, which follows from the fact that $W_{\varepsilon(1-\frac{1}{i})}(f_K(p))$ is contained in $K_i^\varepsilon$. So we have $v_{K_i^\varepsilon}(p_i(p))\rightarrow v_{K^\varepsilon}(p)$, as $i\rightarrow \infty$. The proof is complete.
\end{proof}

Letting $i\rightarrow \infty$ in \eqref{s4:eq3} and using the convergence result in Lemma \ref{s4:lem-conv}, we can prove the following lemma.
\begin{lem}
For sufficiently small $\varepsilon>0$, we have
\begin{equation}\label{s4:Ke}
\int_{\Pi(\tilde{\mathcal{N}}_{K^{\varepsilon}})}\biggl(S(v_{K^\varepsilon}(p))E_r(p)-E_{r-1}(p)\biggr)\gamma(\nu_{K^{\varepsilon}})d\mu(p)=0,
\end{equation}
where $v_{K^{\varepsilon}}$ denotes the unit anisotropic normal of $\partial K^{\varepsilon}$ which exists uniquely for any point of $\partial K^{\varepsilon}$, and $\nu_{K^{\varepsilon}}$ is the corresponding isotropic normal determined by the identity \eqref{s3:nu}.
\end{lem}
\proof
For any point $p\in \partial K^{\varepsilon}$, let $p_i=p_i(p)\in \partial K_i^{\varepsilon}$ as before. Since $\partial K_i^{\varepsilon}$ is smooth, the anisotropic support function satisfies $S(p_i)=S(v_{K_i^{\varepsilon}}(p_i))$. Let
\begin{equation}\label{s4:lem5-pf1}
\lambda_i:=\int_{\pt K^\varepsilon_i} S(v_{K^\varepsilon_i}(p_i))E_r(p_i){\gamma}(\nu_{K_i^{\varepsilon}}(p_i))d\mu(p_i),
\end{equation}
where $\nu_{K_i^{\varepsilon}}(p_i)$ is the unit isotropic normal of $\partial K_i^{\varepsilon}$ at the point $p_i$ which is determined by the anisotropic normal $v_{K_i^{\varepsilon}}(p_i)$ via the identity \eqref{s3:nu}. Then by the equation \eqref{s4:eq3} and the fact $S(p_i)=S(v_{K_i^{\varepsilon}}(p_i))$,  we have
\begin{equation}\label{s4:lem5-pf2}
\lambda_i=\int_{\pt K^\varepsilon_i} E_{r-1}(p_i){\gamma}(\nu_{K_i^{\varepsilon}}(p_i))d\mu(p_i).
\end{equation}
We compute the limit $\lim_{i\to\infty}\lambda_i$ using both expressions of $\lambda_i$. First, by \eqref{s4:lem5-pf2} and using the anisotropic curvature measure and its weak continuity with respect to the anisotropic Hausdorff distance, we have
\begin{align}\label{s4:lem5-pf3}
  \lim_{i\to\infty}\lambda_i =& \lim_{i\to\infty}\int_{\pt K_i^\varepsilon} d\Phi_{n-r+1}(K_i^\varepsilon;p_i) \nonumber\\%\xrightarrow{i\to\infty}
  =&\int_{\pt K^\varepsilon} d\Phi_{n-r+1}(K^\varepsilon;p) \nonumber\\
  = &\int_{\tilde{\mathcal{N}}_{K^\varepsilon}}\psi(v_{K^{\varepsilon}}) E_{r-1}(v_{K^{\varepsilon}})d\tilde{\mathcal{H}}^nv_{K^{\varepsilon}}.
\end{align}
Since $\partial K^{\varepsilon}$ has only regular points and the anisotropic principal curvatures $\kappa_i(p)$ are well defined on $\Pi(\tilde{\mathcal{N}}_{K^{\varepsilon}})$, we can rewrite \eqref{s4:lem5-pf3} as
\begin{align}\label{s4:lem5-pf4}
  \lim_{i\to\infty}\lambda_i %%%= &\int_{\tilde{\mathcal{N}}_{K^\varepsilon}}\psi(v_{K^{\varepsilon}}) E_{r-1}(v_{K^{\varepsilon}})d\tilde{\mathcal{H}}^nv_{K^{\varepsilon}}\nonumber\\
  =&\int_{\Pi(\tilde{\mathcal{N}}_{K^{\varepsilon}})}E_{r-1}(p)\gamma(\nu_{K^{\varepsilon}})d\mu(p).
\end{align}
On the other hand, since $S(v_{K^\varepsilon_i}(p_i))$ is a bounded function, by \eqref{s4:lem5-pf1} and the weak continuity of anisotropic curvature measures, we obtain
\begin{align*}
\lim_{i\rightarrow \infty} \lambda_i &=\lim_{i\rightarrow \infty}\int_{\pt K_i^\varepsilon} S(v_{K^\varepsilon_i}(p_i))d\Phi_{n-r}(K_i^\varepsilon;p_i)\\
&=\int_{\tilde{\mathcal{N}}_{K^\varepsilon}} \psi(v_{K^{\varepsilon}})S(v_{K^\varepsilon}(p))E_r(v_{K^{\varepsilon}})d\tilde{\mathcal{H}}^nv_{K^{\varepsilon}}\\
&=\int_{\Pi(\tilde{\mathcal{N}}_{K^{\varepsilon}})}S(v_{K^\varepsilon}(p))E_r(p)\gamma(\nu_{K^{\varepsilon}})d\mu(p),
\end{align*}
where we used Lemma~\ref{s4:lem-conv} to show the convergence of $S(v_{K^\varepsilon_i}(p_i))$ to $S(v_{K^\varepsilon}(p))$ as $i\to\infty$. This completes the proof of the lemma.
\endproof

Put $v=v_{K^{\varepsilon}}(p)$. Then the map $F_K: \partial K^{\varepsilon}\to \mathcal{N}_K$ maps any point $p\in\partial K^{\varepsilon}$ to $(p-\varepsilon v, v)\in \mathcal{N}_K$. For any point $(p,v)\in \tilde{\mathcal{N}}_{K^{\varepsilon}}$, we also have $(p-\varepsilon v,v)\in \tilde{\mathcal{N}}_K$. Now we can transform the integral \eqref{s4:Ke} on $\Pi(\tilde{N}_{K^{\varepsilon}})\subset \partial K^{\varepsilon}$ to the integral on $\tilde{\mathcal{N}}_K$ or $\mathcal{N}_K$ as follows:
\begin{align*}
  0= &  \int_{\Pi(\tilde{N}_{K^{\varepsilon}})}\biggl(S(v_{K^\varepsilon}(p))E_r(p)-E_{r-1}(p)\biggr)\gamma(\nu_{K^{\varepsilon}})d\mu(p)\\
   =&\int_{\mathcal{N}_K} \biggl(S(v)E_r\left(\frac{\kappa_i(v)}{1+\varepsilon\kappa_i(v)}\right)-E_{r-1}\left(\frac{\kappa_i(v)}{1+\varepsilon\kappa_i(v)}\right)\biggr)\prod_{i=1}^n\frac{1+\varepsilon\kappa_i(v)}{(1+\kappa_i(v)^2)^{1/2}}d\tilde{\mathcal{H}}^nv.
\end{align*}
Comparing the coefficients of $\varepsilon^m$  and letting $\varepsilon\to 0$, we arrive at the equation \eqref{s4:eq-GM} and finish the proof of Theorem~\ref{s4:thm-GM}.
\end{proof}

\section{Volume representation involving anisotropic interior reach}\label{sec:Vol}

In this section, we derive a representation for the volume of a convex body involving the anisotropic interior reach. Let $W$ be a Wulff shape with the support function ${\gamma}\in C^\infty(\SS^n)$.  Given a convex body $K$ in $\R^{n+1}$, the anisotropic inner radius $r(K)$ of $K$ is the radius of the largest Wulff shape contained in $K$, i.e.,
\begin{equation*}
  r(K)=\sup \{\lambda>0: \lambda W+x\subset K~\mathrm{for}~\mathrm{some}~x\in \mathbb{R}^{n+1}\}.
\end{equation*}
The anisotropic interior reach of $K$ at the boundary point $p\in \pt K$ is defined by
\begin{equation*}
r(K,p)=\sup \{\lambda\geq 0: p\in \lambda W+x\subset K \text{ for some }x\in \R^{n+1}\}.
\end{equation*}
The main result of this section is stated as follows.
\begin{prop}\label{s5:prop-main}
Let $W$ be a Wulff shape with the support function ${\gamma}\in C^\infty(\SS^n)$.  For any convex body $K$ in $\mathbb{R}^{n+1}$, we have the following representation for the volume of $K$:
\begin{align}\label{s5:VK}
\mathrm{Vol}(K)&=\frac 1{n+1}\sum_{i=0}^n (-1)^{n-i}\binom{n+1}{i}\int_{\pt K} r(K,p)^{n+1-i} d\Phi_i(K;p),
\end{align}
where $r(K,p)$ is the anisotropic interior reach of $K$, and $d\Phi_i(K;p)$ is the anisotropic curvature measure element.
\end{prop}

The proof of \eqref{s5:VK} uses inner parallel bodies and is inspired by the argument in Sangwine-Yager's \cite{San94} for the isotropic case. For $0\leq \lambda \leq r(K)$, the inner parallel body with respect to $W$ at a distance $\lambda$ from $\pt K$ is defined by
\begin{equation*}
K_{-\lambda}=\{x:\lambda W+x\subset K\}.
\end{equation*}
Then $K_{-\lambda}$ is also a convex body. The outer parallel body at the distance $\lambda\geq 0$ is defined as $K_\lambda=K+\lambda W$. Since the set of boundary points with positive anisotropic interior reach is a subset of the set of boundary points with unique anisotropic normals, we have the following lemma which can be  verified straightforwardly.
\begin{lem}\label{lem-interior-reach}
Let $K$ be a convex body with a positive anisotropic inner radius $r(K)$ in $\R^{n+1}$. If $p\in \pt K_{-\lambda}$, $0\leq \lambda <r(K)$, with $r(K_{-\lambda},p)=\varepsilon> 0$, then for all $t>-\varepsilon$, we have $p+tv(p)\in \pt K_{-\lambda+t}$ and $$r(K_{-\lambda+t},p+tv(p))=\varepsilon+t,$$
where $v(p)$ denotes the unique unit anisotropic normal at the point $p$.
\end{lem}
In fact, for $p\in \partial K_{-\lambda}$ with $r(K_{-\lambda},p)=\varepsilon>0$, the unique outward unit anisotropic normal $v(p)$ is well-defined. So $\varepsilon W+(p-\varepsilon v(p))\subset K_{-\lambda}$, and for all $t>-\varepsilon$, $(t+\varepsilon)W+(p-\varepsilon v(p))$ is the largest interior Wulff shape contained in $K_{-\lambda+t}$ and touching at the boundary point $p+tv(p)$.

For any convex body $K$ with a positive anisotropic inner radius, let $\mathscr{Z}(K)=\{p\in \pt K:r(K,p)=0\}$. We now define the interior skeleton of $K$ by
\begin{equation*}
\mathscr{S}(K)=\bigcup_{0<\lambda\leq r(K)}\mathscr{Z}(K_{-\lambda}).
\end{equation*}
The following lemma shows that the interior skeleton is of zero measure.
\begin{lem}\label{s5:lem2}
If $K$ is a convex body with a positive anisotropic inner radius $r(K)$ in $\R^{n+1}$, then
\begin{equation*}
\mathcal{L}^{n+1}(\mathscr{S}(K))=0.
\end{equation*}
\end{lem}
\begin{proof}
We first observe that $\Phi_n(K;\mathscr{Z}(K))=0$. In fact, recall that $\widetilde{\pt K}$ consists of the boundary points which have a unique anisotropic normal. From the definition of $\mathscr{Z}(K)$ and the fact that the set of boundary points with positive anisotropic interior reach is a subset of $\widetilde{\pt K}$, we have $\widetilde{\pt K}\cap \mathscr{Z}(K)=\emptyset$. This together with \eqref{s3:Phin} implies that
\begin{equation*}
\Phi_n(K;\mathscr{Z}(K))=\int_{\widetilde{\pt K}\cap \mathscr{Z}(K)}{\gamma}(\nu)d\mu(p)=0.
\end{equation*}

Given a positive integer $m$, let $0=r_0<r_1<\cdots<r_{m-1}<r_m=r(K)$ be a partition of the interval $[0,r(K)]$. For $j=1,\dots,m-1$, let
\begin{equation*}
\beta_j=\pt K_{-r_j}\setminus \mathscr{Z}(K_{-r_j}) \text{ and }A_j=A_{r_j-r_{j-1}}(K_{-r_j},\beta_j),
\end{equation*}
where the notation $A_\varepsilon(K,\beta)$ denotes the local $\varepsilon$-parallel set of a convex body $K$ defined as in \eqref{s3:A}.
%\begin{equation}\label{s5:As}
%A_t(K,\beta)=\{x:0<d_W(f_K(x),x)\leq t, \text{ and }f_K(x)\in \beta\}
%\end{equation}
%for any Borel set $\beta$ in $\R^{n+1}$.
Then any point $p$ in $A_j$ lies on the boundary of some inner parallel body, which implies $K\setminus \mathscr{S}(K)\supset \bigcup_{j=1}^{m-1} A_j$. Therefore, using the local Steiner's formula \eqref{s3:Ste}, we have
\begin{align*}
\mathcal{L}^{n+1}(K\setminus \mathscr{S}(K))&\geq \sum_{j=1}^{m-1}\mathcal{L}^{n+1}(A_j)=\sum_{j=1}^{m-1}\mathcal{L}^{n+1}(A_{r_j-r_{j-1}}(K_{-r_j},\beta_j))\\
&=\frac 1{n+1}\sum_{j=1}^{m-1}\sum_{i=0}^n(r_j-r_{j-1})^{n+1-i}\binom{n+1}{i}\Phi_i(K_{-r_j};\beta_j)\\
&\geq \sum_{j=1}^{m-1}(r_j-r_{j-1}) \Phi_n(K_{-r_j};\beta_j)\\
&=\sum_{j=1}^{m-1}(r_j-r_{j-1}) \Phi_n(K_{-r_j};\pt K_{-r_j}).
\end{align*}
Taking the supremum of the right-hand side over all partitions of the interval $[0,r]$ yields
\begin{align}\label{s5:Integration}
\mathcal{L}^{n+1}(K\setminus \mathscr{S}(K))&\geq \int_0^{r(K)} \Phi_n(K_{-\lambda};\pt K_{-\lambda}) d\lambda.
\end{align}

On the other hand, we can apply Federer's coarea formula \cite{Fed59} for Lipschitz functions and the formula \eqref{s3:Phin} to show that the right-hand side of \eqref{s5:Integration} is equal to $\mathcal{L}^{n+1}(K)$. In fact, since $\partial K_{-\lambda}$ is the set of points in $K$ having the anisotropic distance $\lambda$ to the boundary $\partial K$, we can write
\begin{equation*}
  \partial K_{-\lambda}=\{x\in K| ~u(x):=d_W(x,\partial K)=\lambda\},
\end{equation*}
i.e., $\partial K_{-\lambda}$ is the level set $u^{-1}(\lambda)$ of $u$. We claim that $u:K\rightarrow \R$ is a Lipschitz function. Note that the projection map from $x\in K$ to the nearest points on the boundary $\partial K$ is different from the map $f_K$ in \S\ref{sec:3-1} (e.g., it may be multi-valued). So we cannot use the property for $f_K$. To prove the claim, let $x,y\in K$ satisfy $0\leq u(x)<u(y)\leq r(K)$. Choose a point $y_0\in \{y+u(y)W\}\cap \pt K$. Then there exists a point $y_1$ on the segment $yy_0$ such that $u(y_1)=u(x)$ and we can derive
\begin{align*}
|u(y)-u(x)|&=d_W(y,y_0)-d_W(y_1,y_0)=d_W(y,y_1)=d_W(\pt K_{-u(y)},y_1)\\
&\leq d_W(\pt K_{-u(y)},x)\leq d_W(y,x)\leq C|y-x|.
\end{align*}
So we obtain the claim.

%Since the projection map from $x\in K$ to the nearest point on the boundary $\partial K$ is Lipschitz, the function $u$ is a  Lipschitz function on $K$.

In view of the claim, by Federer's coarea formula \cite{Fed59} for Lipschitz functions, we have
\begin{align}\label{s5:coarea}
  \mathcal{L}^{n+1}(K) =& \int_0^{r(K)} \int_{u^{-1}(\lambda)}\frac 1{|Du|}d\mu d\lambda.
\end{align}
Arguing as in \S \ref{sec:3-2}, we see that
\begin{equation*}
  |Du|\gamma(\nu)~=~1
\end{equation*}
at any point $x\in u^{-1}(\lambda)$, where $\nu$ is the isotropic unit normal vector. Then applying \eqref{s3:Phin} to \eqref{s5:coarea} we conclude that
\begin{align}\label{s5:coarea2}
  \mathcal{L}^{n+1}(K) =& \int_0^{r(K)} \int_{u^{-1}(\lambda)}\gamma(\nu)d\mu d\lambda\nonumber\\
  =&\int_0^{r(K)}\Phi_n(K_{-\lambda};\pt K_{-\lambda}) d\lambda.
\end{align}
This combined with \eqref{s5:Integration} completes the proof of the lemma.
%Now for $\lambda \in (0,r)$, consider the set $K_{-\lambda+d\lambda}\setminus K_{-\lambda}$. By the Steiner formula \eqref{s3:Ste}, we have
%\begin{equation*}
%\mathcal{L}^{n+1}(K_{-\lambda+d\lambda}\setminus K_{-\lambda})=\frac 1{n+1}\sum_{r=0}^{n}(d\lambda)^{n+1-r}\binom{n+1}{r}\Phi_r(K_{-\lambda};\pt K_{-\lambda}).
%\end{equation*}
%After integrating the above equality with respect to $\lambda$ and omitting all higher order terms involving $(d\lambda)^{k}$ with $k\geq 2$, we get
%\begin{align*}
%\mathcal{L}^{n+1}(K)&=\int_0^r \mathcal{L}^{n+1}(K_{-\lambda+d\lambda}\setminus K_{-\lambda})=\int_0^r \Phi_n(K_{-\lambda};\pt K_{-\lambda}) d\lambda,
%\end{align*}
%which together with \eqref{s5:Integration} yields $\mathcal{L}^{n+1}(\mathscr{S}(K))=0$.
\end{proof}

\begin{lem}\label{s5:lem3}
Let $K$ be a convex body with a positive anisotropic inner radius $r(K)$ in $\R^{n+1}$. If $0<\lambda<\lambda+\varepsilon\leq r(K)$ and
\begin{equation*}
\beta=\{p\in \pt K_{-\lambda}:0<r(K_{-\lambda},p)\leq \varepsilon\},
\end{equation*}
then
\begin{equation*}
\Phi_i(K_{-\lambda}+\lambda W;\beta+\lambda W)=\Phi_i(K;\beta+\lambda W),\qquad i=0,\dots,n.
\end{equation*}
\end{lem}
\proof The idea is to show that the set $A_t(K_{-\lambda}+\lambda W,\beta+\lambda W)$ coincides with $A_t(K,\beta+\lambda W)$ for all $t>0$ and to use the local Steiner's formula \eqref{s3:Ste}. This follows from the following observation
\begin{equation}\label{s5:lem3-pf1}
  (\beta+\lambda W)\cap \partial K=(\beta+\lambda W)\cap \partial (K_{-\lambda}+\lambda W)=\{p+\lambda v(p):~p\in\beta\},
\end{equation}
where $v(p)$ denotes the unit outward anisotropic normal. The proof of \eqref{s5:lem3-pf1} uses the result in Lemma \ref{lem-interior-reach} repeatedly and is similar to that for Lemma 3 of \cite{San94}. We omit it here.
\endproof

\begin{lem}\label{s5:lem4}
Let $K$ be a convex body and $\beta\subset \pt K\setminus \mathscr{Z}(K)$. Then for $t\geq 0$ and $0\leq i\leq n$,
\begin{equation}\label{s5:Phi-i}
\Phi_i(K;\beta)=\sum_{j=0}^i \binom{i}{j}(-t)^{i-j}\Phi_j(K+tW;\beta+tW).
\end{equation}
\end{lem}
\begin{proof}
First we note that
\begin{equation}\label{s5:lem5-pf1}
\mathcal{L}^{n+1}(A_s(K+tW,\beta+tW))=\frac 1{n+1}\sum_{r=0}^{n}s^{n+1-r}\binom{n+1}r\Phi_r(K+tW;\beta+tW),
\end{equation}
where $A_s(K+tW,\beta+tW)$ is defined by \eqref{s3:A}. On the other hand, since
\begin{equation*}
  A_t(K,\beta)\cup A_s(K+tW,\beta+tW)=A_{s+t}(K,\beta)
\end{equation*}
for $s,t\geq0$, we have
\begin{align}\label{s5:lem5-pf2}
\mathcal{L}^{n+1}&(A_s(K+tW,\beta+tW))=\mathcal{L}^{n+1}(A_{s+t}(K,\beta))-\mathcal{L}^{n+1}(A_{t}(K,\beta))\nonumber\\
&=\frac 1{n+1}\sum_{i=0}^{n}(s+t)^{n+1-i}\binom{n+1}{i}\Phi_i(K;\beta)-\mathcal{L}^{n+1}(A_{t}(K,\beta))\nonumber\\
&=\frac 1{n+1}\sum_{i=0}^{n}\binom{n+1}{i}\sum_{j=0}^{n+1-i}\binom{n+1-i}{j}t^{n+1-i-j}s^j\Phi_i(K;\beta)  -\mathcal{L}^{n+1}(A_{t}(K,\beta)).
\end{align}
Comparing the coefficients of $s^{n+1-r}$ on the right-hand side of the equations \eqref{s5:lem5-pf1} and \eqref{s5:lem5-pf2}, we have
\begin{align*}
 \binom{n+1}r\Phi_r(K+tW;\beta+tW)= &  \sum_{i=0}^{r}\binom{n+1}{i}\binom{n+1-i}{n+1-r}t^{r-i}\Phi_i(K;\beta)\\
  =& \sum_{i=0}^{r}\binom{n+1}r\binom{r}{i}t^{r-i}\Phi_i(K;\beta).
\end{align*}
It follows that
\begin{equation*}
\Phi_r(K+tW;\beta+tW)=\sum_{i=0}^r \binom{r}{i}t^{r-i}\Phi_i(K;\beta).
\end{equation*}
This is a linear system of equations, and can be solved to get $\Phi_i(K;\beta)$ as in \eqref{s5:Phi-i}. The verification that \eqref{s5:Phi-i} is the solution to this system relies on the fact
\begin{align}\label{s5:binom-1}
 \sum_{i=0}^r\binom{r}{i}t^{r-i} \sum_{j=0}^i\binom{i}{j}(-t)^{i-j}a_j=&\sum_{j=0}^rt^{r-j}a_j\sum_{i=j}^r(-1)^{i-j}\binom{r}{i}\binom{i}{j}\nonumber \\
  = & \sum_{j=0}^rt^{r-j}\binom{r}{j}a_j\sum_{i=j}^r(-1)^{i-j}\binom{r-j}{i-j}=a_r,
\end{align}
where $\{a_j\}_{j=0}^r$ is any sequence.
\end{proof}

Let $K$ and $\beta$ in Lemma \ref{s5:lem4} be the $K_{-\lambda}$ and $\beta$ in Lemma \ref{s5:lem3}. As an immediate consequence of Lemma \ref{s5:lem4} we can derive the following formula
\begin{align}\label{s5:binom-2}
\sum_{i=0}^n\lambda^{n+1-i}\binom{n+1}{i}\Phi_i(K_{-\lambda};\beta)&=\sum_{i=0}^n \lambda^{n+1-i}(-1)^{n-i}\binom{n+1}{i}\Phi_i(K_{-\lambda}+\lambda W;\beta+\lambda W)\nonumber\\
&=\sum_{i=0}^n \lambda^{n+1-i}(-1)^{n-i}\binom{n+1}{i}\Phi_i(K;\beta+\lambda W),
\end{align}
where in the first equality we used the exchange of the summations as in \eqref{s5:binom-1}.

Now we prove the main result of this section.
\begin{proof}[Proof of Proposition \ref{s5:prop-main}]
Given any positive integer $m$, let $0<r_1<\dots< r_m=r(K)$ be a partition of the interval $[0,r(K)]$. For $j=1,\dots,m-1$, let
\begin{align*}
\beta_j=\{p:p\in &\pt K_{-r_j}, \: 0<r(K_{-r_j},p)\leq r_{j+1}-r_j\},\\
&\quad A_j=A_{r_j}(K_{-r_j},\beta_j).
\end{align*}
Note that $A_j\cap \pt K$ consists of all points $x$ with $r_j<r(K,x)\leq r_{j+1}$.

On the other hand, any $y\in K\setminus (\mathscr{S}(K)\cup \partial K)$ must lie on the boundary of some parallel body $K_{-\lambda}$, and the set $(\lambda W+y)\cap \pt K$ consists of only one point $x$. Thus for sufficiently large $m$, we can choose some partition of the interval $[0,r(K)]$, such that for some $1\leq j\leq m-1$,
\begin{equation*}
 \lambda < r_j <r(K,x)\leq r_{j+1}.
\end{equation*}
It follows that $y\in A_j$. Thus $\mathcal{L}^{n+1}(K\setminus (\mathscr{S}(K)\cup \partial K))$ is the supremum of the sum of $\mathcal{L}^{n+1}(A_j)$ over all partitions of the interval $[0,r(K)]$. Since the interior skeleton is of zero measure by Lemma \ref{s5:lem2}, this also gives the volume of $K$. As a consequence, we obtain
\begin{align*}
\mathrm{Vol}(K)&=\mathcal{L}^{n+1}(K\setminus (\mathscr{S}(K)\cup \partial K))=\sup_{0<r_1<\dots< r_m=r(K)}\mathcal{L}^{n+1}(\bigcup_j A_j)\\
&=\sup_{0<r_1<\dots< r_m=r(K)}\sum_{j}\mathcal{L}^{n+1}(A_j)\\
&=\frac 1{n+1}\sup_{0<r_1<\dots< r_m=r(K)}\sum_{j}\sum_{r=0}^{n}(r_j)^{n+1-r}\binom{n+1}{r}\Phi_r(K_{-r_j};\beta_j)\\
&=\frac 1{n+1}\sup_{0<r_1<\dots< r_m=r(K)}\sum_{j}\sum_{i=0}^n (r_j)^{n+1-i}(-1)^{n-i}\binom{n+1}{i}\Phi_i(K;\beta_j+r_j W)\\
&=\frac 1{n+1}\sum_{i=0}^n (-1)^{n-i}\binom{n+1}{i}\left(\sup_{0<r_1<\dots< r_m=r(K)}\sum_{j}(r_j)^{n+1-i}\Phi_i(K;\beta_j+r_j W)\right)\\
&=\frac 1{n+1}\sum_{i=0}^n (-1)^{n-i}\binom{n+1}{i}\int_{\pt K} r(K,p)^{n+1-i} d\Phi_i(K;p),
\end{align*}
where we used \eqref{s5:binom-2} in the fourth equality. This completes the proof of Proposition \ref{s5:prop-main}.
\end{proof}

\section{Characterization of the Wulff shape}\label{sec:proof_thm1}
In this section, we will characterize the Wulff shape using an equation involving linear combinations of curvature measures with positive coefficients. We first prove an inequality of Heintze--Karcher type.

\subsection{An inequality of Heintze--Karcher type}
\begin{prop}
Let $W$ be a Wulff shape with the support function ${\gamma}\in C^\infty(\SS^n)$.   For any convex body $K$ in $\mathbb{R}^{n+1}$, we have
\begin{equation}\label{s6:HK}
(n+1)\mathrm{Vol}(K)\leq \int_{\pt K} \frac{{\gamma}(\nu)}{E_1}d\mu(p),
\end{equation}
where $\mathrm{Vol}(K)=\mathcal{L}^{n+1}(K)$ denotes the volume of $K$ and $E_1=E_1(\kappa)$ is the anisotropic mean curvature of $\partial K$ defined for almost every point of $\partial K$. Moreover, the equality holds in \eqref{s6:HK} if and only if the anisotropic principal curvatures satisfy $\kappa_1(p)=\cdots=\kappa_n(p)=r(K,p)^{-1}>0$ almost everywhere.
\end{prop}
\begin{rem}
On the right-hand side of the inequality \eqref{s6:HK}, the integrand $\gamma(\nu)/E_1$ is well defined on $\widetilde{\pt K}$, the set of twice differentiable points; while on $\pt K\setminus \widetilde{\pt K}$ which has zero $n$-dimensional Hausdorff measure, $\gamma(\nu)/E_1$ can be assigned any finite value. The inequality \eqref{s6:HK} for a smooth $K$ was first proved in \cite{HLMG09} (see also \cite{Del18,MaX11,XiaZ17} for alternative proofs).
\end{rem}
\begin{proof}
Recall that the set $\mathscr{Z}(K)=\{p\in \pt K:r(K,p)=0\}$ consists of the points on $\partial K$ with zero anisotropic interior reach. So by the representation \eqref{s5:VK} for the volume of $K$, we have
\begin{align*}
\mathrm{Vol}(K)&=\frac 1{n+1}\sum_{i=0}^n (-1)^{n-i}\binom{n+1}{i}\int_{\pt K} r(K,p)^{n+1-i} d\Phi_i(K;p)\\
&=\frac 1{n+1}\sum_{i=0}^n (-1)^{n-i}\binom{n+1}{i}\int_{\pt K \setminus \mathscr{Z}(K)} r(K,p)^{n+1-i} d\Phi_i(K;p).
\end{align*}

On the other hand, since on $\pt K \setminus \mathscr{Z}(K)$ the anisotropic unit normal is unique, we obtain that
\begin{equation*}
d\Phi_i(K;p)=\psi(v)E_{n-i}(v)d\tilde{\mathcal{H}}^nv= E_{n-i}(v(p)){\gamma}(\nu)d\mu(p)
\end{equation*}
holds on $(\pt K \setminus \mathscr{Z}(K))\cap \Pi(\tilde{\mathcal{N}}_K)$. In addition, by Lemma \ref{lem-interior-reach}, for any $\varepsilon>0$ and $p\in (\pt K \setminus \mathscr{Z}(K))\cap \Pi(\tilde{\mathcal{N}}_K)$, the parallel body $K^\varepsilon$ has an interior touching Wulff shape with radius $r(K,p)+\varepsilon$ at $p+\varepsilon v(p)$, which implies
\begin{equation*}
r(K,p)+\varepsilon \leq \frac{1}{\kappa_i(p+\varepsilon v(p))}, \: i=1,\dots,n.
\end{equation*}
Sending $\varepsilon\rightarrow 0+$, by the definition of the anisotropic principal curvatures on $\tilde{\mathcal{N}}_K$ we get
\begin{equation}\label{s6:prop1-pf1}
r(K,p) \leq \frac{1}{\kappa_i(v(p))}, \: i=1,\dots,n.
\end{equation}
Hence $r(K,p)\leq 1/{E_1(v(p))}$. From the inequality above we also know that all $\kappa_i(v(p))$ are finite for $p\in (\pt K \setminus \mathscr{Z}(K))\cap \Pi(\tilde{\mathcal{N}}_K)$. In addition, by the arithmetic--geometric means inequality,
\begin{equation}\label{s6:prop1-pf2}
  \prod_{i=1}^n (1-\kappa_i(v)t) \leq \left(\frac{1}{n}\sum_{i=1}^n (1-\kappa_i(v)t)\right)^n = (1-E_1(v)t)^n,\quad 0\leq t\leq r(K,p).
\end{equation}
Combining these facts together and noting that $\mathcal{H}^n(\mathscr{Z}(K))=0$, we get
\begin{align*}
\mathrm{Vol}(K)&=\int_{ \pt K \setminus \mathscr{Z}(K)} \int_0^{r(K,p)} \sum_{i=0}^n(-t)^{n-i}\binom{n}{i} E_{n-i}(v(p)){\gamma}(\nu)dt d\mu(p)\\
&=\int_{ \pt K \setminus \mathscr{Z}(K)} \int_0^{r(K,p)} \prod_{i=0}^n(1-\kappa_i(v(p)) t){\gamma}(\nu)dt d\mu(p)\\
&\leq \int_{ \pt K \setminus \mathscr{Z}(K)} \int_0^{1/E_{1}(v(p))} (1-E_{1}(v(p)) t)^n{\gamma}(\nu)dt d\mu(p)\\
&=\frac{1}{n+1}\int_{ \pt K \setminus \mathscr{Z}(K)}\frac{{\gamma}(\nu)}{E_{1}(v(p))}d\mu(p)\\
&=\frac{1}{n+1}\int_{\pt K}\frac{{\gamma}(\nu)}{E_{1}(v(p))}d\mu(p).
\end{align*}
This proves the inequality \eqref{s6:HK}. The equality case follows from the two inequalities \eqref{s6:prop1-pf1} and \eqref{s6:prop1-pf2} becoming equalities.
\end{proof}

%This section follows the argument in Kohlmann's \cite{Koh98}.
\subsection{Proof of Theorem \ref{thm-main}}$\ $

Now we  prove Theorem \ref{thm-main}, which is recalled here.
\begin{thm}\label{s6:thm}
Let $W$ be a Wulff shape with the support function ${\gamma}\in C^\infty(\SS^n)$.  For a convex body $K$ in $\mathbb{R}^{n+1}$, assume that the anisotropic curvature measures satisfy
\begin{equation}\label{S6:condition}
\Phi_n(K;\cdot)=\sum_{r=0}^{n-1} \lambda_r \Phi_r(K;\cdot)
\end{equation}
for some constants $\lambda_0,\dots,\lambda_{n-1}\geq 0$.  Then $K$ is a scaled Wulff shape.
\end{thm}

\begin{proof}
First we apply the scaling property of the anisotropic curvature measures to get
\begin{equation*}
\Phi_n(\alpha K;\cdot )=\alpha^n \Phi_n( K;\cdot{} )=\sum_{r=0}^{n-1}\lambda_r \alpha^{n-r}\Phi_r(\alpha K;\cdot )
\end{equation*}
for $\alpha>0$. Since $\sum_{r=0}^{n-1}\lambda_r\neq 0$,  we can assume
\begin{equation}\label{S6:normalization}
\sum_{r=0}^{n-1}\lambda_r= 1.
\end{equation}
Recall the definition \eqref{s3:psi-r} of the anisotropic curvature measure
\begin{align}\label{s6:psi-r}
\Phi_r(K;\beta)= \int_{\Pi^{-1}(\beta\cap\partial K)\cap \tilde{\mathcal{N}}_K} \psi(v)E_{n-r}(v) d\tilde{\mathcal{H}}^nv
\end{align}
for $0\leq r\leq n$ and a Borel set $\beta$ in $\mathbb{R}^{n+1}$. At any point $v\in \tilde{\mathcal{N}}_{K,0}$, we can write the anisotropic principal curvature $\kappa_i(v)$ as $\kappa_i(p)$ with $p=\Pi(v)\in \partial K$, and $v$ is the unique anisotropic unit normal of $\partial K$ at the point $p$. Since the $n$-dimensional Hausdorff measure of $\partial K\setminus \widetilde{\partial K}$ is zero and by Proposition \ref{prop-regular-points} the set $\widetilde{\partial K}$ is contained in $\Pi(\tilde{\mathcal{N}}_{K,0})$, we have the following decomposition
\begin{align}\label{s6:pf-1}
\Phi_r(K;\beta)= &\int_{\Pi^{-1}(\beta\cap\partial K)\cap \tilde{\mathcal{N}}_{K,0}} \psi(v)E_{n-r}(v) d\tilde{\mathcal{H}}^nv+\int_{\Pi^{-1}(\beta\cap\partial K)\cap (\tilde{\mathcal{N}}_K\setminus \tilde{\mathcal{N}}_{K,0})} \psi(v)E_{n-r}(v) d\tilde{\mathcal{H}}^nv\nonumber\\
=&\int_{\beta\cap \Pi(\tilde{\mathcal{N}}_{K,0})}E_{n-r}(p){\gamma}(\nu)d\mu(p)+ \int_{\Pi^{-1}(\beta\cap\partial K)\cap (\tilde{\mathcal{N}}_K\setminus \tilde{\mathcal{N}}_{K,0})} \psi(v)E_{n-r}(v) d\tilde{\mathcal{H}}^nv\nonumber\\
=&\int_{\beta\cap \pt K}E_{n-r}(p){\gamma}(\nu)d\mu(p)+ \int_{\Pi^{-1}(\beta\cap\partial K)\cap (\tilde{\mathcal{N}}_K\setminus \tilde{\mathcal{N}}_{K,0})} \psi(v)E_{n-r}(v) d\tilde{\mathcal{H}}^nv
\end{align}
for $0\leq r\leq n$. Let us define the singular part
\begin{equation*}
a_r(K;\beta):=\int_{\Pi^{-1}(\beta\cap\partial K)\cap (\tilde{\mathcal{N}}_K\setminus \tilde{\mathcal{N}}_{K,0})} \psi(v)E_{n-r}(v) d\tilde{\mathcal{H}}^nv.
\end{equation*}
Similarly, we define
\begin{equation*}
b_r(K;\beta):= \int_{\Pi^{-1}(\beta\cap\partial K)\cap (\tilde{\mathcal{N}}_K\setminus \tilde{\mathcal{N}}_{K,0})} \psi(v)S(v)E_{n-r}(v) d\tilde{\mathcal{H}}^nv.
\end{equation*}
Then the Minkowski formula \eqref{s4:eq-GM} can be written as
\begin{align}\label{s6:GM}
  \int_{\partial K}E_{r-1}(p)\gamma(\nu)d\mu(p)+a_{n-r+1}(K;\mathbb{R}^{n+1}) =& \int_{\partial K}S(p)E_r(p)\gamma(\nu)d\mu(p)+b_{n-r}(K;\mathbb{R}^{n+1})
\end{align}
for any integer $1\leq r\leq n$, where $a_{n-r+1}$ and $b_{n-r}$ are singular parts of the formula. We write the regular parts as integrations over $\partial K$, but we need to keep in mind that the integrands are only defined in the smooth part of $\partial K$ and that we have set the integrands to be any finite value on $\partial K\setminus \widetilde{\partial K}$.

\begin{lem}\label{s6:lem3}
For a convex body $K$, an integer $r\in \{0,\dots,n-1\}$ and a constant $\lambda>0$, assume that
\begin{equation*}
\Phi_{r}(K;\cdot)\leq \lambda \Phi_n(K;\cdot).
\end{equation*}
Then $a_r(K;\R^{n+1})=0$ and we can choose the origin such that $b_r(K;\R^{n+1})=0$.
\end{lem}
\begin{proof}
For the first assertion, to get a contradiction, let us assume
\begin{align*}
a_r(K;\R^{n+1})&=\int_{\tilde{\mathcal{N}}_K\setminus \tilde{\mathcal{N}}_{K,0}} \psi(v)E_{n-r}(v) d\tilde{\mathcal{H}}^nv> 0.
\end{align*}
Then there exists a subset $B\subset \tilde{\mathcal{N}}_K\setminus \tilde{\mathcal{N}}_{K,0}$ with $\tilde{\mathcal{H}}^n(B)> 0$ such that $\psi(v)E_{n-r}(v)>0$ on $B$. Then we get
\begin{align*}
0&<\int_{B} \psi(v)E_{n-r}(v) d\tilde{\mathcal{H}}^nv\\
&\leq  \int_{\Pi^{-1}(\Pi(B))\cap \tilde{\mathcal{N}}_K} \psi(v)E_{n-r}(v) d\tilde{\mathcal{H}}^nv\\
&=\Phi_r(K;\Pi(B))\leq \lambda \Phi_n(K;\Pi(B)).
\end{align*}
By \eqref{s3:Phin}, the last term is given by
\begin{align*}
\Phi_n(K;\Pi(B))=\int_{\Pi(B)\cap \widetilde{\pt K}}{\gamma}(\nu)d\mu(p).
\end{align*}
However, the set $\Pi(B)\cap \widetilde{\pt K}=\emptyset$ by Proposition \ref{prop-regular-points}. So $\Phi_n(K;\Pi(B))=0$, which is a contradiction. Therefore we have $a_r(K;\R^{n+1})=0$.

For the second assertion, by translating the origin, we may assume that there exist $0<r_1<r_2$ such that
\begin{equation*}
r_1W\subset K \subset r_2W.
\end{equation*}
It follows that $r_1\leq S(v)\leq r_2$ for $\forall v\in \mathcal{N}_K$. So by the definitions of $a_r$ and $b_r$ we have
\begin{equation*}
0=r_1 \:a_r(K;\R^{n+1})\leq b_r(K;\R^{n+1})\leq r_2\: a_r(K;\R^{n+1})=0,
\end{equation*}
which implies $b_r(K;\R^{n+1})=0$.
\end{proof}

For $\lambda_{r}>0$, by Lemma \ref{s6:lem3} we have $a_r(K;\R^{n+1})=0$. It follows from \eqref{S6:condition} that
\begin{equation}\label{s6:pf-2}
1=\sum_{r=1}^n \lambda_{n-r}  E_r(p),
\end{equation}
almost everywhere on $\pt K$. Then
\begin{equation*}
1\leq \sum_{r=1}^n \lambda_{n-r}  E_1^r(p),
\end{equation*}
which implies $E_1(p)\geq 1>0$ a.e. on $\partial K$ in view of \eqref{S6:normalization}. %So we can apply the Heintze--Karcher type inequality later.

Similarly, by Lemma~\ref{s6:lem3} the singular part $b_{n-r}(K;\R^{n+1})$ in the Minkowski formula \eqref{s6:GM} also vanishes for $\lambda_{n-r}>0$ after choosing a suitable origin. Meanwhile, using the generalized divergence formula \cite[Thm.~4.5.6]{Fed69} and the relationship between the anisotropic support function $S(p)$ and the isotropic support function $\langle p,\nu\rangle$, we get
\begin{align*}
(n+1)\mathrm{Vol}(K)&=\int_{\pt K} \langle p,\nu\rangle d\mu(p)=\int_{\pt K} \frac{\langle p,\nu\rangle}{\gamma(\nu)} \gamma(\nu)d\mu(p)=\int_{\pt K} S(p){\gamma}(\nu)d\mu(p).
\end{align*}
Then by \eqref{s6:pf-2} and the Minkowski formula \eqref{s6:GM} we derive
\begin{align*}
(n+1)\mathrm{Vol}(K)& =\sum_{r=1}^n \lambda_{n-r} \int_{\pt K} S(p)E_r(p){\gamma}(\nu)d\mu(p)\\
&=\sum_{r=1}^n \lambda_{n-r}  \left(\int_{\pt K} E_{r-1}(p){\gamma}(\nu)d\mu(p)+a_{n-r+1}(K;\mathbb{R}^{n+1})\right)\\
&\geq \sum_{r=1}^n \lambda_{n-r}  \int_{\pt K} E_{r-1}(p){\gamma}(\nu)d\mu(p).
\end{align*}
On the other hand, by the Heintze--Karcher type inequality \eqref{s6:HK},
\begin{align*}
(n+1)\mathrm{Vol}(K)&\leq \int_{\pt K} \frac{{\gamma}(\nu)}{E_{1}(p)}d\mu(p)\\
&=\int_{\pt K} \frac{{\gamma}(\nu)}{E_{1}(p)}\left(\sum_{r=1}^n \lambda_{n-r}  E_r(p)\right)d\mu(p)\\
&\leq\sum_{r=1}^n \lambda_{n-r} \int_{\pt K} E_{r-1}(p){\gamma}(\nu)d\mu(p),
\end{align*}
where we used the inequality $E_r(p)\leq E_{r-1}(p)E_1(p)$. So the equality must hold, and $\kappa_1(p)=\cdots=\kappa_n(p)=r(K,p)^{-1}=1$ almost everywhere on $\pt K$. This implies that the anisotropic area of the boundary $\partial K$ satisfies
\begin{align*}
|\pt K|_{\gamma}&=\int_{\pt K}{\gamma}(\nu)d\mu(p)=\int_{\pt K}E_n(p){\gamma}(\nu)d\mu(p)\\
&=\Phi_0(K;\R^{n+1})-a_0(K;\R^{n+1})\\
&\leq \Phi_0(K;\R^{n+1})= |\pt W|_{\gamma},
\end{align*}
where in the third equality we used the decomposition \eqref{s6:pf-1} and in the last equality we used the fact \eqref{s3:Phi0}. Since the anisotropic reach of $K$ is 1 almost everywhere, $K$ contains a Wulff shape of radius $1$ and we have
\begin{align*}
\mathrm{Vol}(K)&\geq  \mathrm{Vol}(W).
\end{align*}
Combining the above two inequalities gives rise to
\begin{equation*}
\left(\frac{|\pt K|_{\gamma}}{|\pt W|_{\gamma}}\right)^{n+1}\leq \left(\frac{\mathrm{Vol}(K)}{\mathrm{Vol}(W)}\right)^n,
\end{equation*}
which together with the equality case of the anisotropic isoperimetric inequality \eqref{s2:Iso} (with $\ell=n$) implies that $K$ is a scaled Wulff shape. This completes the proof of Theorem \ref{s6:thm}.
\end{proof}

\section{Volume preserving anisotropic curvature flow}\label{sec:VPACF}

In the second part of this paper, we will study the volume preserving flow by powers of the $k$th anisotropic mean curvature.

\subsection{Variation equations}$\ $

We first review some variation equations for geometric quantities on the hypersurface $M$ in $\mathbb{R}^{n+1}$. Let $W$ be a Wulff shape with the support function $\gamma\in C^\infty(\mathbb{S}^n)$, and let $X(\cdot,t):M \rightarrow \mathbb{R}^{n+1}$, $t\in [0,T)$, be a smooth family of embeddings satisfying
\begin{equation}\label{eq2.1}
\frac{\pt }{\pt t}X =\eta \nu_\gamma,
\end{equation}
for some time-dependent smooth function $\eta$, where $\nu_\gamma$ denotes the anisotropic unit normal of $M_t=X(M,t)$. The closure of the enclosed domain of $M_t$ is denoted by $K_t$. The following equations were calculated in \cite{And01}.
\begin{lem}[{\cite{And01}}]
Under the variation \eqref{eq2.1}, we have the following evolution equations for the anisotropic area element $d\mu_\gamma$, anisotropic unit normal $\nu_\gamma$, the induced anisotropic metric $\hat{g}_{ij}$ and the anisotropic Weingarten tensor $\hat{h}_k^\ell$ of the evolving hypersurface $M_t=X(M,t)$:
\begin{align}
\frac{\pt}{\pt t}d\mu_\gamma&= \eta H_\gamma d\mu_\gamma,\label{s2:dmu}\\
\frac{\pt}{\pt t}\nu_\gamma&= -\hat{\nabla} \eta,\\%\label{s2:nu_F}\\
\frac{\pt}{\pt t} \hat{g}_{ij}&=2\eta \hat{h}_{ij}-Q_{ij}{}^k\pt_k \eta,\\
\frac{\pt}{\pt t}\hat{h}^{\ell}_k&= -\hat{\nabla}^\ell\hat{\nabla}_k \eta-{A}_k{}^{p\ell}\pt_p \eta-\eta \hat{h}^p_k\hat{h}_p^\ell,\label{eq2.h}
%\frac{\pt}{\pt t}\sigma_F &= \eta -V^i\pt_i\eta,\label{eq2.sigma}
\end{align}
where the upper indices are lifted using the metric $\hat{g}$, the function $H_{\gamma}=nE_1$ is the anisotropic mean curvature of $M_t$, and the tensors $A$ and $Q$ are defined in \S \ref{sec:2-2}.  In addition, the volume of the enclosed $K_t$ satisfies
\begin{align}\label{s2:Vol-1}
\frac{d}{dt} \mathrm{Vol}(K_t)&=\int_{M_t}\eta d\mu_\gamma.
\end{align}
\end{lem}

In \cite{Rei76} Reilly derived the following variational formula for the mixed volumes.
\begin{lem}[{\cite{Rei76}}]
Under the variation \eqref{eq2.1},  the mixed volume $V_{n+1-k}(K_t,W)$ relative to the Wulff shape $W$ evolves by \begin{equation}\label{s4:evl-Vk}
  \frac d{dt}V_{n+1-k}(K_t,W)~=~(n+1-k)\int_{M_t}\eta\: {E}_kd\mu_{\gamma},
\end{equation}
where $k=1,\dots,n$.
\end{lem}

Since the global term $\phi(t)$ in the flow equation \eqref{flow-VMCF} is defined as in \eqref{s1:phi-1}, it is easy to check using \eqref{s2:Vol-1} that the volume of the evolving domain remains constant along the flow \eqref{flow-VMCF}:
\begin{align*}%\label{s5:Vn1}
\frac d{dt}\mathrm{Vol}(K_t)=\int_{M_t}(\phi(t)-{E}_k^{{\alpha}/k})d\mu_{\gamma}=0.
\end{align*}

\subsection{Monotonicity of the isoperimetric ratio}$\ $

We next show that the flow \eqref{flow-VMCF} decreases the $(n+1-k)$th mixed volume $V_{n+1-k}(K_t,W)$. This will imply that the isoperimetric ratio
\begin{equation*}
{\mathcal I}_{n+1-k}(K_t,W) = \frac{V_{n+1-k}(K_t,W)^{n+1}}{V_{n+1}(K_t,W)^{n+1-k} V_0(K_t,W)^{k}}
\end{equation*}
is monotone non-increasing in time.

%We note that all of these isoperimetric ratios are comparable, in the sense that a bound on any of these implies bounds on all of the others:  The fact that $r_m(\Omega,W)$ is non-increasing in $m$ implies that
%\begin{equation}\label{s5:Il-3}
%{\mathcal I}_m(\Omega,W)^{\frac1{m(n+1)}}=\frac{r_m(\Omega,W)}{r_{n+1}(\Omega,W)}
%\end{equation}
%is non-increasing in $m$.  In the other direction, the Alexandrov-Fenchel inequality $V_i^{n+1-j}\geq V_{n+1}^{i-j}V_j^{n+1-i}$ for $i>j$ implies that ${\mathcal I}_m(\Omega,W)^{\frac{1}{n+1-m}}$ is non-decreasing in $m$.

\begin{prop}\label{s5:prop-monot}
Let $M_t=X(M,t)=\partial K_t$ be a smooth strictly convex solution of the flow \eqref{flow-VMCF} on $[0,T)$ with the global term $\phi(t)$ given by \eqref{s1:phi-1}. Then $V_{n+1-k}(K_t,W)$ is non-increasing in $t$, and thus the isoperimetric ratio ${\mathcal I}_{n+1-k}(K_t,W)$ is monotone non-increasing in $t$.
\end{prop}
\proof
By the definition \eqref{s1:phi-1} of $\phi(t)$, the volume of the enclosed $K_t$ remains to be a constant. Since $K_t$ is smooth and convex, we can write the mixed volume $V_{n+1-k}(K_t,W)$ as in \eqref{def-quermass}.  By \eqref{s4:evl-Vk}, we have
\begin{align}\label{s5:Vn1k}
  \frac 1{(n-k+1)}&\frac d{dt}V_{n+1-k}(K_t,W)=~\int_{M_t}{E}_{k}(\phi(t)-{E}_k^{{\alpha}/k})d\mu_{\gamma}\nonumber\\
  =&\frac{1}{|M_t|_{\gamma}}\int_{M_t}{E}_kd\mu_{\gamma}\int_{M_t}{E}_k^{\alpha/k}d\mu_{\gamma} -\int_{M_t}{E}_k^{\frac{\alpha}k+1}d\mu_{\gamma}\nonumber\\
  =&-\int_{M_t}\left({E}_k-\bar{E}_k\right){E}_k^{\alpha/k}d\mu_{\gamma}\nonumber\\
  =&-\int_{M_t}\left({E}_k-\bar{E}_k\right)\left({E}_k^{\alpha/k}-\bar{E}_k^{\alpha/k}\right)d\mu_{\gamma}\nonumber\\
  \leq&~0,
\end{align}
where $\bar{E}_k$ denotes the average integral of $E_k$. This says that $V_{n+1-k}(K_t,W)$ is non-increasing in time. Since $V_{n+1}(K_t,W)=(n+1)\mathrm{Vol}(K_t)$ and $V_{0}(K_t,W)=(n+1)\mathrm{Vol}(W)$ are constants, we conclude that the isoperimetric ratio ${\mathcal I}_{n+1-k}(K_t,W)$ is non-increasing in $t$.
\endproof

The monotonicity of the isoperimetric ratio together with the isoperimetric inequality \eqref{s2:Iso} implies that
\begin{equation*}
  1\leq {\mathcal I}_{n+1-k}(K_t,W)\leq {\mathcal I}_{n+1-k}(K_0,W)
\end{equation*}
along the flow \eqref{flow-VMCF}. This can be used to control the anisotropic area and shape of the evolving hypersurfaces $M_t$. Let $K$  be a convex body in $\mathbb{R}^{n+1}$ with boundary $M=\partial K$. The anisotropic outer radius of $K$ (or $M$) relative to $W$ is defined as
\begin{equation*}
  R(K):=\inf\{\rho>0:~K\subset\rho W+p~\mathrm{for}~\mathrm{some}~ \mathrm{point}~ p\in \mathbb{R}^{n+1}\}.
\end{equation*}
The anisotropic inner radius of $K$ relative to $W$ is defined as
\begin{equation*}
  r(K):=\sup\{\rho>0:~\rho W+p\subset K~\mathrm{for}~\mathrm{some}~ \mathrm{point}~ p\in \mathbb{R}^{n+1}\}.
\end{equation*}
We denote by $r(t)$ and $R(t)$ the anisotropic inner radius and outer radius of the evolving set $K_t$ under the flow \eqref{flow-VMCF}, respectively. As a consequence of Proposition \ref{s5:prop-monot}, we have the following estimate on $r(t)$ and $R(t)$.

\begin{prop}\label{s5:prop-2}
Let $M_t$ be a smooth convex solution of the flow \eqref{flow-VMCF} on $[0,T)$ with the global term $\phi(t)$ given by \eqref{s1:phi-1}. Then there exist constants $c_1,c_2,R_1,R_2$ depending only on $n,k,{\gamma}, \mathrm{Vol}(K_0)$ and $V_{n+1-k}(K_0,W)$ such that
\begin{equation}\label{s5:AreaVol}%\label{prop2-2-ineq}
  0<c_1\leq |M_t|_{\gamma}\leq c_2
\end{equation}
and
\begin{equation}\label{s5:radius}
  0<R_1\leq r(t)\leq R(t)\leq R_2
\end{equation}
for all $t\in [0,T)$.
\end{prop}
\proof
First, the isoperimetric inequality \eqref{s2:Iso} and the preserving of the enclosed volume imply that
\begin{align*}
  |M_t|_{\gamma}^{n+1}=&~V_n(K_t,W)^{n+1}\\
  \geq &~V_{n+1}(K_t,W)^nV_0(K_t,W)\\
  =&~(n+1)^{n+1}\mathrm{Vol}(K_0)^n\mathrm{Vol}(W),
\end{align*}
which gives the lower bound of $|M_t|_{\gamma}$. For the upper bound of $|M_t|_{\gamma}$: if $k=1$ in the flow equation \eqref{flow-VMCF}, the monotonicity of $V_{n+1-k}(K_t,W)$ gives the upper bound $|M_t|_{\gamma}\leq |M_0|_{\gamma}$; if $k=2,\dots,n$, the Alexandrov--Fenchel inequality \eqref{AF-k=n+1} implies that
\begin{align*}
  V_{n+1-k}^n(K_0,W)\geq &~V_{n+1-k}^{n}(K_t,W)\\
  \geq &~V_n^{n+1-k}(K_t,W)V_0^{k-1}(K_t,W)\\
  =&~|M_t|_{\gamma}^{n+1-k}(n+1)^{k-1}\mathrm{Vol}(W)^{k-1},
\end{align*}
which gives the upper bound of $|M_t|_{\gamma}$. This proves the inequality \eqref{s5:AreaVol}.

In Proposition 5.1 of \cite{And01}, the first author proved that for any convex body $K$ with boundary $M=\partial K$, there holds
\begin{equation}\label{s5:And-ineq}
  \frac{\mathrm{Vol}(K)}{|M|_{\gamma}}~\leq~r(K)\leq R(K)\leq ~\frac{|M|_{\gamma}^n}{(n+1)^{n-1}\mathrm{Vol}(W)\mathrm{Vol}(K)^{n-1}}.
\end{equation}
Then the estimate \eqref{s5:radius} follows from the estimate \eqref{s5:AreaVol} and the inequality \eqref{s5:And-ineq}.
\endproof

\section{Anisotropic Gauss map parametrization}\label{sec:gauss map}
In this section, we rewrite the flow \eqref{flow-VMCF} as a scalar parabolic equation on the Wulff shape $\Sigma=\partial W$ of the anisotropic support function of the evolving hypersurface, and derive some evolution equations in this parametrization.

\subsection{Anisotropic Gauss map parametrization}\label{sec:8-1}$\ $

We first review the anisotropic Gauss map parametrization of a smooth strictly convex hypersurface. Given a smooth strictly convex hypersurface $M$ in $\mathbb{R}^{n+1}$, we assume that the origin lies inside $M$. The Gauss map is defined as $\nu: M\to \mathbb{S}^n$ which maps the point $x\in M$ to the outward unit normal $\nu\in \mathbb{S}^n$ at this point. The Gauss map of a strictly convex hypersurface is a nondegenerate diffeomorphism between $M$ and $\mathbb{S}^n$. Since the Wulff shape $W$ is also strictly convex, we can define a map from $\mathbb{S}^n$ to $\Sigma=\partial W$ which maps $\nu$ to $\nu_{\gamma}=D{\gamma}(\nu)$. Then the anisotropic Gauss map is defined as $\nu_{\gamma}:M\to \Sigma$ satisfying $\nu_{\gamma}(x)=D{\gamma}(\nu(x))$ for any point $x\in M$.

The anisotropic Gauss map is a nondegenerate diffeomorphism and can be used to reparametrize a strictly convex hypersurface $M$ (see \cite{Xia13}):
\begin{equation*}%\label{}
  X:\Sigma\to M\subset \mathbb{R}^{n+1},\quad X(z)=X(\nu_{\gamma}^{-1}(z)),\quad z\in \Sigma.
\end{equation*}
The anisotropic support function of $M$ is then defined as a function on $\Sigma$ by
\begin{equation}\label{s8:s-def}
  s(z)=\sup_{y\in M}G(z)(z,y)=G(z)(z,X(z))
\end{equation}
for $z\in \Sigma$, where $G$ is the metric on $\mathbb{R}^{n+1}$ defined in \eqref{s2:G}. Let $\bar{g}$ and $\bar{\nabla}$ denote the induced metric and its Levi-Civita connection on $\Sigma$ from $(\mathbb{R}^{n+1},G)$ respectively. Taking the first covariant derivative of $s$ on $\Sigma$, we have
\begin{equation}\label{s7:s-d1}
  \bar{\nabla}_is(z)=G(z)(\partial_iz,X(z))
\end{equation}
in a local orthonormal basis $\{e_i\}, i=1,\dots,n$ of $T_z\Sigma=T_{X(z)}M$.  Taking the second covariant derivatives of $s$, Xia \cite{Xia13} obtained that the anisotropic principal radii of $M$ are given by the eigenvalues $\tau=(\tau_1,\dots,\tau_n)$ of
\begin{equation}\label{s4:tau-def}
  \tau_{ij}=\tau_{ij}[s]=\bar{\nabla}_i\bar{\nabla}_js+\bar{g}_{ij}s-\frac 12Q_{ijk}\bar{\nabla}_ks.
\end{equation}
The anisotropic principal curvatures $\kappa=(\kappa_1,\dots,\kappa_n)$ are related to the anisotropic principal radii $\tau=(\tau_1,\dots,\tau_n)$ by
\begin{equation*}
  \kappa_i=\frac 1{\tau_i},\qquad i=1,\dots,n.
\end{equation*}
Moreover, any smooth function $s$ on $\Sigma$ satisfying $(\tau_{ij})>0$ is an anisotropic support function of a smooth, closed and strictly convex hypersurface $M$ in $\mathbb{R}^{n+1}$.

%\subsection{Codazzi and Simons type equations}
\begin{lem}
We have the following Codazzi and Simons type equations for $\tau_{ij}$:
\begin{equation}\label{s4:Codaz}
  \bar{\nabla}_j\tau_{k\ell}+\frac 12Q_{k\ell p}\tau_{jp}~=~\bar{\nabla}_k\tau_{j\ell}+\frac 12Q_{j\ell p}\tau_{kp}
\end{equation}
and
\begin{align}\label{s4:simon}
  \bar{\nabla}_i \bar{\nabla}_j\tau_{k\ell} =&\bar{\nabla}_k\bar{\nabla}_\ell\tau_{ij}+\frac 12Q_{ijp}\bar{\nabla}_p\tau_{k\ell}-\frac 12Q_{k\ell p}\bar{\nabla}_p\tau_{ij}+\bar{g}_{ij}\tau_{k\ell}-\bar{g}_{kj}\tau_{i\ell}\nonumber\\
 &+\bar{g}_{i\ell}\tau_{kj}-\bar{g}_{k\ell}\tau_{ij} +\frac 14\left(Q_{jkq}Q_{ipq}-Q_{kpq}Q_{qij}\right)\tau_{p\ell}\nonumber\\
&+\frac 14\left(Q_{j\ell q}Q_{ipq}-Q_{\ell pq}Q_{qij}\right)\tau_{kp}+\frac 14\left(Q_{k\ell q}Q_{ipq}-Q_{kpq}Q_{qi\ell}\right)\tau_{pj}\nonumber\\
&+\frac 14\left(Q_{k\ell q}Q_{jpq}-Q_{kpq}Q_{qj\ell}\right)\tau_{ip}+\frac 12\bar{\nabla}_pQ_{ijk}\tau_{\ell p}+\frac 12\bar{\nabla}_pQ_{ij\ell}\tau_{kp}\nonumber\\
&-\frac 12\bar{\nabla}_pQ_{jk\ell}\tau_{ip}-\frac 12\bar{\nabla}_pQ_{ik\ell}\tau_{jp}.
\end{align}
\end{lem}
The proof of \eqref{s4:Codaz} is by taking covariant derivatives of \eqref{s4:tau-def}, and using the Gauss equation \eqref{s2:gauss-2} and the Ricci identity. See the proof of Lemma 5.2 in \cite{Xia13}. Taking the covariant derivative of \eqref{s4:Codaz}  and using the Gauss equation \eqref{s2:gauss-2} and the Ricci identity, we can obtain \eqref{s4:simon} after rearranging the terms. See the calculation in Lemma A.1 of \cite{WX21}.

\subsection{Properties of symmetric functions}$\ $

To derive the evolution equation of the anisotropic support function along the flow \eqref{flow-VMCF}, we require some properties on smooth symmetric functions which we recall here. The readers can refer to \cite{And07} for more details.

For a smooth symmetric function $F(A)=f(\kappa(A))$, where $A=(A_{ij})\in \mathrm{Sym}(n)$ is a symmetric matrix and $\kappa(A)=(\kappa_1,\dots,\kappa_n)$ give the eigenvalues of $A$, we denote by $\dot{F}^{ij}(A)$ and $\ddot{F}^{ij,kl}(A)$ the first and second derivatives of $F$ with respect to the components of its argument.  We also denote by $ \dot{f}^i(\kappa)$ and $\ddot{f}^{ij}(\kappa)$ the first and second derivatives of $f$ with respect to $\kappa$, respectively. At any diagonal $A$ with distinct eigenvalues $\kappa=\kappa(A)$, the first derivative of $F$ satisfies
\begin{equation*}
  \dot{F}^{ij}(A)~=~\dot{f}^i(\kappa)\delta_{ij}
\end{equation*}
and the second derivative of $F$ in the direction $B\in \mathrm{Sym}(n)$ is given in terms of $\dot{f}$ and $\ddot{f}$ by  (see, e.g., \cite{And07}):
\begin{equation}\label{s2:F-ddt}
  \ddot{F}^{ij,k\ell}(A)B_{ij}B_{k\ell}=\sum_{i,j}\ddot{f}^{ij}(\kappa)B_{ii}B_{jj}+2\sum_{i>j}\frac{\dot{f}^i(\kappa)-\dot{f}^j(\kappa)}{\kappa_i-\kappa_j}B_{ij}^2.
\end{equation}
This formula makes sense as a limit in the case of any repeated values of $\kappa_i$.

We say that $f$ is concave if its Hessian is non-positive definite. From the equation \eqref{s2:F-ddt}, we know that $F$ is concave at $A$ if and only if $f$ is concave at $\kappa$ and
\begin{equation*}%\label{s2:f-conc}
  \left(\dot{f}^k-\dot{f}^\ell\right)(\kappa_k-\kappa_\ell)~\leq ~0,\quad \forall~k\neq \ell.
\end{equation*}
Moreover, if $f$ is concave, homogeneous of degree one and is normalized such that $f(1,\dots,1)=1$, then $\sum_{i=1}^n\dot{f}^i\geq 1$.
%\begin{equation}\label{s2:f-conc-2}
%
%\end{equation}

We say that $f$ is \emph{inverse-concave}, if  the dual function
    \begin{equation*}%\label{s5:f-dual}
      f_*(x_1,\dots,x_n)=f(x_1^{-1},\dots,x_n^{-1})^{-1}
    \end{equation*}
is concave. Note that the dual function $f_*$ is only defined on the positive definite cone $\Gamma_+$
     \begin{equation*}
      \Gamma_+=\{x=(x_1,\dots,x_n)\in\mathbb{R}^n:\ ~x_i>0,~i=1,\dots,n\}.
    \end{equation*}
The functions
\begin{equation*}
  E_k^{1/k},\quad \mathrm{and} \quad (E_k/{E_\ell})^{1/{(k-\ell)}},~ k>\ell
\end{equation*}
are important examples of concave and inverse-concave symmetric functions.

\subsection{Evolution equations}$\ $

The anisotropic support function introduced in \S \ref{sec:8-1} allows the degenerate parabolic system of Equation \eqref{flow-VMCF} to be written as a parabolic scalar equation (see, e.g., \cite{Urb91,xia-2}). Let $F=E_k^{1/k}$ and define its dual function $F_*$ by
\begin{align*}
  F_*(x_1,\dots,x_n)=&F(x_1^{-1},\dots,x_n^{-1})^{-1}  =~\left(\frac{E_n(x)}{E_{n-k}(x)}\right)^{1/k}
\end{align*}
for $x\in \Gamma_+$. Then $F_*$ is concave in its arguments.
\begin{lem}
The solution of the flow \eqref{flow-VMCF} is given, up to a time-dependent diffeomorphism,  by solving the scalar parabolic equation on the Wulff shape $\Sigma$
 \begin{equation}\label{flow-gauss}
   \frac{\partial }{\partial t}s=\Psi(\tau_{ij})+\phi(t)=-F_*(\tau_{ij})^{-\alpha}+\phi(t)
 \end{equation}
for the anisotropic support function $s$, where $F_*(\tau_{ij})$ can be viewed as the function $F_*=F_*(\tau)$ evaluated at the eigenvalues $\tau=(\tau_1,\cdots,\tau_n)$ of $\tau_{ij}$.
\end{lem}
\noindent In fact, suppose that $X(\cdot,t):M\to \mathbb{R}^{n+1}$ is a family of smooth embeddings satisfying \eqref{flow-VMCF} and each $M_t$ is strictly convex for $t\in [0,T)$. We can reparametrize $M_t$ using the anisotropic Gauss map. Define $\bar{X}(\cdot,t):\Sigma\to \mathbb{R}^{n+1}$ by
\begin{equation*}
  \bar{X}(z,t)=X(\nu_{\gamma}^{-1}(z,t),t)
\end{equation*}
for each $t>0$. Then
\begin{align*}
  \frac{\partial}{\partial t}\bar{X}(z,t) =& \frac{\partial X}{\partial x^i}\frac{\partial (\nu_{\gamma}^{-1})^i}{\partial t}+\frac{\partial}{\partial t}{X}(\nu_{\gamma}^{-1}(z,t),t),
\end{align*}
where $x^i$, $i=1,\dots,n$, denote the local coordinates on $M$. By the definition of the anisotropic support function, it follows from the above equation that
\begin{align}
  \frac{\partial}{\partial t}s(z,t) =& G(z)(z,\frac{\partial}{\partial t}\bar{X}(z,t))\nonumber\\
  =&G(z)(z,\frac{\partial}{\partial t}{X}(\nu_{\gamma}^{-1}(z,t),t))\nonumber\\
  =&\phi(t)-E_k^{\alpha/k}(\kappa),\label{s8:st-2}
\end{align}
where we used the fact that $ \partial X/\partial x^i$ is tangential to $M_t$. Since $\kappa=(\kappa_1,\dots,\kappa_n)$ are the reciprocal of the eigenvalues of $\tau_{ij}$, the last term of \eqref{s8:st-2} can be expressed as $F_*^{-\alpha}(\tau_{ij})$ and we obtain the equation \eqref{flow-gauss}. The converse is also true: any smooth solution $s(z,t)$ to \eqref{flow-gauss} defines a smooth solution $M_t$ to \eqref{flow-VMCF} whose anisotropic support functions are given by $s(z,t)$.

Applying the implicit function theorem to \eqref{flow-gauss}, we have the existence of a smooth solution of \eqref{flow-VMCF} for a short time for any smooth, strictly convex initial hypersurface. In the rest of this section, we derive the evolution equations of $s, |\bar{\nabla}s|^2$ and $\tau_{ij}[s]$ along the flow \eqref{flow-gauss}.
\begin{lem}
The evolution equation \eqref{flow-gauss} for the anisotropic support function $s$ has the following equivalent form:
\begin{align}\label{s4:s-1}
 \frac{\partial }{\partial t}s- &\dot{\Psi}^{k\ell}\left(\bar{\nabla}_k\bar{\nabla}_\ell s-\frac 12Q_{k\ell p}\bar{\nabla}_ps\right) ~=~\phi(t)+(1+\alpha)\Psi+s\dot{\Psi}^{k\ell}\bar{g}_{k\ell}.
\end{align}
\end{lem}
\proof
Since $\Psi$ is homogeneous of degree $-\alpha$ with respect to $\tau_{ij}$, we have
\begin{equation*}
 -\alpha \Psi= \dot{\Psi}^{k\ell}\tau_{k\ell}= \dot{\Psi}^{k\ell}\left(\bar{\nabla}_k\bar{\nabla}_\ell s+\bar{g}_{k\ell}s-\frac 12Q_{k\ell p}\bar{\nabla}_ps\right).
\end{equation*}
This implies that
\begin{align*}
 \frac{\partial }{\partial t}s- &\dot{\Psi}^{k\ell}\left(\bar{\nabla}_k\bar{\nabla}_\ell s-\frac 12Q_{k\ell p}\bar{\nabla}_ps\right)  =  \Psi+\phi(t)+\alpha\Psi+s\dot{\Psi}^{k\ell}\bar{g}_{k\ell}.
\end{align*}
\endproof

\begin{lem}
Under the flow \eqref{flow-gauss}, the squared norm $|\bar{\nabla}s|^2$ of the gradient of the anisotropic support function $s$ evolves according to
\begin{align}\label{s8:ds-0}
 & \frac{\partial }{\partial t} |\bar{\nabla}s|^2 -\dot{\Psi}^{k\ell}\left(\bar{\nabla}_k\bar{\nabla}_\ell|\bar{\nabla}s|^2 -\frac 12Q_{k\ell p}\bar{\nabla}_p|\bar{\nabla}s|^2\right)\nonumber\\
 = & ~ \dot{\Psi}^{k\ell}\biggl(-2\tau_{ik}\tau_{i\ell}-2s^2\bar{g}_{k\ell}+4s\tau_{k\ell}-2\tau_{ik}Q_{i\ell p}s_p+2sQ_{k\ell p}s_p\biggr)\nonumber\\
 &\quad +\dot{\Psi}^{k\ell}\biggl(2s_ks_\ell -\frac 12(2Q_{i\ell m}Q_{mkp}-Q_{ipm}Q_{mk\ell})s_is_p-\bar{\nabla}_iQ_{k\ell p}s_ps_i\biggr).
\end{align}
\end{lem}
\proof
The evolution of $|\bar{\nabla}s|^2$ can be computed as follows:
\begin{align*}
  \frac{\partial }{\partial t} |\bar{\nabla}s|^2=&~ 2s_i\bar{\nabla}_i(\partial_ts) =  2s_i\bar{\nabla}_i(\Psi+\phi(t))=~2s_i\dot{\Psi}^{k\ell}\bar{\nabla}_i\tau_{k\ell}\\
  =&~2s_i\dot{\Psi}^{k\ell}\biggl(\bar{\nabla}_i\bar{\nabla}_k\bar{\nabla}_\ell s+\bar{\nabla}_is\bar{g}_{k\ell}-\frac 12Q_{k\ell p}\bar{\nabla}_i\bar{\nabla}_ps-\frac 12\bar{\nabla}_iQ_{k\ell p}\bar{\nabla}_ps\biggr).
\end{align*}
By the Ricci identity and \eqref{s2:gauss-2}, we get
\begin{align*}
  \bar{\nabla}_k\bar{\nabla}_\ell|\bar{\nabla}s|^2 =& 2s_{ik}s_{i\ell}+2s_i\bar{\nabla}_k\bar{\nabla}_\ell\bar{\nabla}_is \\
  =&2s_{ik}s_{i\ell}+2s_i(\bar{\nabla}_i\bar{\nabla}_k\bar{\nabla}_\ell s+\bar{R}_{ki\ell p}\bar{\nabla}_ps) \\
  =& 2s_{ik}s_{i\ell}+2s_i\bar{\nabla}_i\bar{\nabla}_k\bar{\nabla}_\ell s+2|\bar{\nabla}s|^2\bar{g}_{k\ell}-2s_ks_\ell +\frac 12(Q_{i\ell m}Q_{mkp}-Q_{ipm}Q_{mk\ell})s_is_p.
\end{align*}
Combining the above two equations yields
\begin{align*}%\label{s8:ds}
 & \frac{\partial }{\partial t} |\bar{\nabla}s|^2 -\dot{\Psi}^{k\ell}\left(\bar{\nabla}_k\bar{\nabla}_\ell|\bar{\nabla}s|^2 -\frac 12Q_{k\ell p}\bar{\nabla}_p|\bar{\nabla}s|^2\right)\\
 =&~\dot{\Psi}^{k\ell}\biggl(-2s_{ik}s_{i\ell}+2s_ks_\ell -\frac 12(Q_{i\ell m}Q_{mkp}-Q_{ipm}Q_{mk\ell})s_is_p-\bar{\nabla}_iQ_{k\ell p}s_ps_i\biggr).
\end{align*}
Rewriting $s_{ik}, s_{i\ell}$ using \eqref{s4:tau-def} gives the equation \eqref{s8:ds-0}.
\endproof

\begin{lem}
Under the flow \eqref{flow-gauss}, the speed function $\Psi$ evolves according to
\begin{equation}\label{s4:flow-Psi}
  \frac{\partial}{\partial t}\Psi=\dot{\Psi}^{k\ell}\left(\bar{\nabla}_k\bar{\nabla}_\ell\Psi-\frac 12Q_{k\ell p}\bar{\nabla}_p\Psi\right)+(\Psi+\phi(t))\dot{\Psi}^{k\ell}\bar{g}_{k\ell}.
\end{equation}
The tensor $\tau_{ij}=\tau_{ij}[s]$ evolves by
\begin{align}\label{s4:flow-tau2}
  \frac{\partial}{\partial t}\tau_{ij}=&\dot{\Psi}^{k\ell}\left(\bar{\nabla}_k\bar{\nabla}_\ell\tau_{ij}-\frac 12Q_{k\ell p}\bar{\nabla}_p\tau_{ij}\right)+\ddot{\Psi}^{k\ell,pq}\bar{\nabla}_i\tau_{k\ell}\bar{\nabla}_j\tau_{pq}\nonumber\\
  &+((1-\alpha)\Psi+\phi(t))\bar{g}_{ij}-\dot{\Psi}^{k\ell}\bar{g}_{k\ell}\tau_{ij}\nonumber\\
  & +\frac 12\dot{\Psi}^{k\ell} \left(Q_{jkq}Q_{ipq}-Q_{kpq}Q_{qij}\right)\tau_{p\ell}\nonumber\\
  &+\frac 12\dot{\Psi}^{k\ell} \left(Q_{k\ell q}Q_{ipq}-Q_{kpq}Q_{qi\ell}\right)\tau_{pj}\nonumber\\
&+\dot{\Psi}^{k\ell} \bar{\nabla}_pQ_{ijk}\tau_{p\ell}-\dot{\Psi}^{k\ell} \bar{\nabla}_pQ_{ik\ell}\tau_{pj}.
\end{align}
\end{lem}
\proof
(i) The evolution of $\Psi$ follows from the equation \eqref{flow-gauss} and the definition \eqref{s4:tau-def} of $\tau$:
\begin{align*}
  \frac{\partial}{\partial t}\Psi =& \dot{\Psi}^{k\ell}\frac{\partial}{\partial t}\tau_{k\ell} \\
  = & \dot{\Psi}^{k\ell}\left(\bar{\nabla}_k\bar{\nabla}_\ell\frac{\partial}{\partial t}s+\bar{g}_{k\ell}\frac{\partial}{\partial t}s-\frac 12Q_{k\ell p}\bar{\nabla}_p\frac{\partial}{\partial t}s\right)\\
  =&\dot{\Psi}^{k\ell}\left(\bar{\nabla}_k\bar{\nabla}_\ell\Psi-\frac 12Q_{k\ell p}\bar{\nabla}_p\Psi\right)+(\Psi+\phi(t))\dot{\Psi}^{k\ell}\bar{g}_{k\ell}.
\end{align*}
(ii) Again by \eqref{s4:tau-def} and \eqref{flow-gauss}, we get
\begin{align}\label{s4:flow-tau1}
  \frac{\partial}{\partial t}\tau_{ij}=&\bar{\nabla}_i\bar{\nabla}_j \frac{\partial}{\partial t}s+\bar{g}_{ij} \frac{\partial}{\partial t}s-\frac 12Q_{ijp}\bar{\nabla}_p \frac{\partial}{\partial t}s\nonumber\\
  =&\bar{\nabla}_i\bar{\nabla}_j\Psi+\bar{g}_{ij}(\Psi+\phi(t))-\frac 12Q_{ijp}\bar{\nabla}_p\Psi\nonumber\\
=& \dot{\Psi}^{k\ell} \bar{\nabla}_i\bar{\nabla}_j\tau_{k\ell}+\ddot{\Psi}^{k\ell,pq}\bar{\nabla}_i\tau_{k\ell}\bar{\nabla}_j\tau_{pq} +\bar{g}_{ij}(\Psi+\phi(t))-\frac 12Q_{ijp}\bar{\nabla}_p\Psi.
\end{align}
Applying the Simons' type equation \eqref{s4:simon}, we can rewrite the first term on the right-hand side of \eqref{s4:flow-tau1} as
\begin{align}\label{s4:tau-3}
 \dot{\Psi}^{k\ell} \bar{\nabla}_i\bar{\nabla}_j\tau_{k\ell} = & \dot{\Psi}^{k\ell} \bar{\nabla}_k\bar{\nabla}_\ell\tau_{ij}+\frac 12Q_{ijp}\bar{\nabla}_p\Psi-\frac 12\dot{\Psi}^{k\ell}Q_{k\ell p}\bar{\nabla}_p\tau_{ij}-\alpha\Psi\bar{g}_{ij}-\dot{\Psi}^{k\ell}\bar{g}_{k\ell}\tau_{ij} \nonumber\\
& +\frac 12\dot{\Psi}^{k\ell} \left(Q_{jkq}Q_{ipq}-Q_{kpq}Q_{qij}\right)\tau_{p\ell}+\frac 12\dot{\Psi}^{k\ell} \left(Q_{k\ell q}Q_{ipq}-Q_{kpq}Q_{qi\ell}\right)\tau_{pj}\nonumber\\
&+\dot{\Psi}^{k\ell} \bar{\nabla}_pQ_{ijk}\tau_{p\ell}-\dot{\Psi}^{k\ell} \bar{\nabla}_pQ_{ik\ell}\tau_{pj}.
\end{align}
The equation \eqref{s4:flow-tau2} follows by substituting \eqref{s4:tau-3} into \eqref{s4:flow-tau1}.
\endproof

\section{Long-time existence}\label{sec:LTE}

In this section, we prove that the flow \eqref{flow-VMCF} exists for all positive time $t\in [0,\infty)$. Let $M_t$ be a strictly convex solution to the flow \eqref{flow-VMCF} on some time interval $[0,T)$. We first observe that in light of the estimate \eqref{s5:radius} on the anisotropic inner radius and outer radius of $M_t$, for each time we can find an origin $p_t$ such that the anisotropic support function $s(z,t)$ of $M_t$ with respect to $p_t$ satisfies
\begin{equation}\label{s9:s}
R_1\leq s(z,t)\leq R_2,
\end{equation}
where $R_1,R_2$ are positive constants depending only on $n,k,{\gamma}, \mathrm{Vol}(K_0)$ and $V_{n+1-k}(K_0,W)$.  The $C^0$ estimate \eqref{s9:s} of $s$ together with the strict convexity of $M_t$ also implies a bound on the gradient of the anisotropic support function $s$
\begin{equation}\label{s9:ds-1}
  |\bar{\nabla}s|^2(z,t)\leq C(R_2-R_1)
\end{equation}
for all $z\in\Sigma$.   The $C^1$ estimate \eqref{s9:ds-1} in terms of the $C^0$ estimate does not depend on the flow and holds for any strictly convex hypersurface.
\begin{lem}\label{s9.lemC1}
Assume that $M$ is a strictly convex hypersurface in $\mathbb{R}^{n+1}$ with the anisotropic support function $s(z)$ with respect to some point $o$ satisfying
\begin{equation}\label{s9.lemC1-s}
  R_1\leq s(z)\leq R_2,\quad \forall~z\in \Sigma
\end{equation}
for some positive constants $R_1,R_2$. Then the gradient of $s$ satisfies
\begin{equation}\label{s9.lemC1-ds}
  |\bar{\nabla}s|^2(z)\leq C(R_2-R_1)
\end{equation}
for all $z\in \Sigma$.
\end{lem}
\proof
This estimate \eqref{s9.lemC1-ds} follows from the estimate \eqref{s9.lemC1-s} and the equations \eqref{s8:s-def} and \eqref{s7:s-d1} for strictly convex hypersurfaces directly. Alternatively,  we consider the function (cf. \cite[Proposition 4.3]{xia-2})
\begin{equation*}
  \omega(z)=\frac{|\bar{\nabla}s|^2}{R_2+\delta-s}
\end{equation*}
on $\Sigma$, where $\delta>0$ is an arbitrary constant. Suppose that the maximum of $\omega(z)$ is achieved at $z_0\in \Sigma$. Then at the point $z_0$, there holds
\begin{align}\label{s9:ds-0}
  0= & \sum_i\bar{\nabla}_is\bar{\nabla}_i\omega =\frac{2s_is_ks_{ki}}{R_2+\delta-s}+\frac{|\bar{\nabla}s|^4}{(R_2+\delta-s)^2}\nonumber\\
   =& \frac{2s_is_k\tau_{ki}}{R_2+\delta-s}-\frac{2s|\bar{\nabla}s|^2}{R_2+\delta-s}+\frac{Q_{kip}s_is_ks_p}{R_2+\delta-s}+\frac{|\bar{\nabla}s|^4}{(R_2+\delta-s)^2}\nonumber\\
   \geq &\frac{ |\bar{\nabla}s|^2}{(R_2+\delta-s)^2}\left(|\bar{\nabla}s|^2-(2s+C|\bar{\nabla}s|)(R_2+\delta-s)\right),
\end{align}
where $C$ is a constant depending on $Q$, the notation $s_i, s_{ki}$ denotes the covariant derivatives of $s$ on $\Sigma$ with respect to the metric $\bar{g}$, and we used \eqref{s4:tau-def} in the third equality. It follows from \eqref{s9:ds-0} that
\begin{equation*}
  \omega(z_0)=\frac{|\bar{\nabla}s|^2}{R_2+\delta-s}\leq 2s+C|\bar{\nabla}s|.
\end{equation*}
If $2s\geq C|\bar{\nabla}s|$ at $z_0$, then
\begin{equation*}
  \omega(z_0)\leq 4s\leq 4R_2.
\end{equation*}
If $2s\leq C|\bar{\nabla}s|$ at $z_0$, then
\begin{equation*}
\frac{|\bar{\nabla}s|}{R_2+\delta-s}\leq 2C,
\end{equation*}
and
\begin{equation*}
  \omega(z_0)\leq (2C)^2(R_2+\delta-s(z_0))\leq 4C^2(R_2+\delta-R_1).
\end{equation*}
Hence, there exists a constant $C=C(Q,R_1,R_2)$ such that $\omega(z)\leq C$ for all $z\in\Sigma$. This implies that
\begin{equation*}
  |\bar{\nabla}s|^2=(R_2+\delta-s(z))\omega(z)\leq C(R_2+\delta-R_1)
\end{equation*}
for all $z\in\Sigma$. Since $\delta>0$ is arbitrary, the estimate \eqref{s9.lemC1-ds} follows.
\endproof

The estimate \eqref{s9:s} depends on the origin $p_t$ we choose. The following lemma shows the existence of a ball with a fixed center enclosed by our flow hypersurfaces on a suitable fixed time interval.
\begin{lem}\label{s8:lem1}
For any time $t_0\in [0,T)$, let $p_0$ be the center such that the estimate \eqref{s9:s} holds on $M_{t_0}$. Then there exists some positive $\tau$ depending only on $n,k,{\gamma},\alpha, \mathrm{Vol}(K_0)$ and $V_{n+1-k}(K_0,W)$ such that the anisotropic support function $s(z,t)$ of $M_t$ with respect to the center $p_0$ satisfies
\begin{equation}\label{s9:s-1}
  s(z,t)\geq R_1/2,\qquad \forall~t\in [t_0,\min\{T,t_0+\tau\}).
\end{equation}
\end{lem}
\proof
Let $\underline{s}(t)=\min_{\Sigma} s(\cdot,t)=s(z_t,t)$. At the point $z_t$, we have
\begin{equation*}
  \tau_{ij}[s(z_t,t)]\geq \underline{s}(t)\bar{g}_{ij}.
\end{equation*}
Then the evolution equation \eqref{flow-gauss} of $s(z,t)$ implies that
\begin{equation*}
  \frac d{dt}\underline{s}(t)\geq -\underline{s}(t)^{-\alpha}+\phi(t)\geq -\underline{s}(t)^{-\alpha}.
\end{equation*}
By the initial condition $\underline{s}(t_0)\geq R_1$, we solve the above inequality as follows
\begin{equation*}
  \underline{s}(t)\geq \left(R_1^{1+\alpha}-(\alpha+1)(t-t_0)\right)^{\frac 1{1+\alpha}}\geq R_1/2,
\end{equation*}
provided that $t-t_0\leq (1+\alpha)^{-1}(1-2^{-\alpha-1})R_1^{\alpha+1}=:\tau$. The conclusion \eqref{s9:s-1} follows immediately.
\endproof

We now estimate the upper bound on the speed of the flow.
\begin{prop}\label{s5:thm-Ek-ub}
Let $M_t$ be a smooth convex solution of the flow \eqref{flow-VMCF} on $[0,T)$ with the global term $\phi(t)$ given by \eqref{s1:phi-1}. Then
\begin{equation}\label{s8:Ek-bd}
  \max_{M_t}{E}_k\leq  ~C\left(1+t^{-\frac{\alpha}{1+\alpha}}\right)
\end{equation}
for any $t\in [0,T)$, where $C$ depends on $n,k,{\gamma},\alpha, \mathrm{Vol}(K_0)$ and $V_{n+1-k}(K_0,W)$.
\end{prop}
\proof
The proof is by applying the technique of Tso \cite{Tso85}.  For any given time $t_0\in [0,T)$, the estimate \eqref{s9:s-1} implies that the anisotropic support function $s(z,t)$ with respect to the center $p_0$ satisfies $s(z,t)\geq R_1/2$ in the time interval $t\in [t_0,\min\{T,t_0+\tau\})$. Then the following function $Z$
\begin{equation*}
  Z(z,t)=\frac{F_*^{-\alpha}}{s(z,t)-\frac 14R_1},\quad z\in \Sigma,
\end{equation*}
is well defined for all time $t\in [t_0,\min\{T,t_0+\tau\})$. The estimate \eqref{s9:s-1} together with \eqref{s5:radius} also implies that the anisotropic support function $s(z,t)$ with respect to the center $p_0$ is bounded from above by $2R_2$. Since ${E}_k^{\alpha/k}=F_*^{-\alpha}$, to prove \eqref{s8:Ek-bd} it suffices to estimate the upper bound of $Z$. Combining \eqref{s4:s-1} and \eqref{s4:flow-Psi}, we obtain
\begin{align}\label{s5:Z}
\frac{\partial }{\partial t}Z- &\dot{\Psi}^{k\ell}\left(\bar{\nabla}_k\bar{\nabla}_\ell Z-\frac 12Q_{k\ell p}\bar{\nabla}_pZ\right)\nonumber\\
=&\frac 1{s-\frac 14R_1}\left(\frac{\partial }{\partial t}F_*^{-\alpha}- \dot{\Psi}^{k\ell}\left(\bar{\nabla}_k\bar{\nabla}_\ell F_*^{-\alpha}-\frac 12Q_{k\ell p}\bar{\nabla}_pF_*^{-\alpha}\right)\right)\nonumber\\
& -\frac Z{s-\frac 14R_1}\left(\frac{\partial }{\partial t}s- \dot{\Psi}^{k\ell}\left(\bar{\nabla}_k\bar{\nabla}_\ell s-\frac 12Q_{k\ell p}\bar{\nabla}_ps\right)\right)+2 \dot{\Psi}^{k\ell}\frac{\bar{\nabla}_kZ\bar{\nabla}_\ell s}{s-\frac 14R_1}\nonumber\\
=& 2 \dot{\Psi}^{k\ell}\frac{\bar{\nabla}_kZ\bar{\nabla}_\ell s}{s-\frac 14R_1}+\frac 1{s-\frac 14R_1}(F_*^{-\alpha}-\phi(t))\dot{\Psi}^{k\ell}\bar{g}_{k\ell}\nonumber\\
& -\frac Z{s-\frac 14R_1}\left(\phi(t)+(1+\alpha)\Psi+s\dot{\Psi}^{k\ell}\bar{g}_{k\ell}\right)\nonumber\\
  = &2 \dot{\Psi}^{k\ell}\frac{\bar{\nabla}_kZ\bar{\nabla}_\ell s}{s-\frac 14R_1}-\frac{R_1Z}{4(s-\frac 14R_1)}\dot{\Psi}^{k\ell}\bar{g}_{k\ell}+(1+\alpha)Z^2-\phi(t)\frac{\dot{\Psi}^{k\ell}\bar{g}_{k\ell}+Z}{s-\frac 12R_1}.
\end{align}
The second term on the right-hand side of \eqref{s5:Z} can be estimated as follows:
\begin{align*}
  \frac{R_1Z}{4(s-\frac 14R_1)}\dot{\Psi}^{k\ell}\bar{g}_{k\ell} =& \frac{\alpha R_1Z}{4(s-\frac 14R_1)}F_*^{-\alpha-1}\dot{F}_*^{k\ell}\bar{g}_{k\ell} \\
  \geq  & \frac{\alpha R_1Z}{4(s-\frac 14R_1)}F_*^{-\alpha-1}\\
  =&\frac{\alpha R_1}4Z^{2+\frac 1{\alpha}}(s-\frac 14R_1)^{1/{\alpha}}\\
  \geq &\alpha\left(\frac{R_1}4\right)^{1+1/{\alpha}}Z^{2+\frac 1{\alpha}},
\end{align*}
where we used the fact $\dot{F}_*^{k\ell}\bar{g}_{k\ell}=\sum_i\dot{f}^i_*\geq 1$ due to the concavity of $F_*$. Let $\tilde{Z}(t)=\sup Z(\cdot,t)$. Since $\phi(t)\geq 0$, Equation \eqref{s5:Z} implies that (in the form given in \cite[Lemma 3.5]{Ha86})
\begin{align*}
\frac{d }{d t}\tilde{Z}\leq &-\alpha\left(\frac{R_1}4\right)^{1+1/{\alpha}}\tilde{Z}^{2+\frac 1{\alpha}}+(1+\alpha)\tilde{Z}^2.
\end{align*}
Let
\begin{equation*}
  \tilde{Z}_0(\alpha,R_1)=\left(\frac 4{R_1}\right)^{\alpha+1}\left(\frac{2(1+\alpha)}{\alpha}\right)^{\alpha},
\end{equation*}
which is a positive constant depending only on $\alpha$ and $R_1$. Whenever $\tilde{Z}(t)\geq \tilde{Z}_0$, we have
\begin{align*}
\frac{d }{d t}\tilde{Z}\leq &-(1+\alpha)\tilde{Z}^2.
\end{align*}
It follows that
\begin{align*}
  \tilde{Z}(t)\leq &\max\left\{\left(\frac 4{R_1}\right)^{\alpha+1}\left(\frac{2(1+\alpha)}{\alpha}\right)^{\alpha},\left(\tilde{Z}(t_0)^{-\frac{1+\alpha}{\alpha}}+\frac{1+\alpha}{2}\left(\frac{R_1}{4}\right)^{\frac{1+\alpha}{\alpha}}(t-t_0)\right)^{-\frac{\alpha}{1+\alpha}} \right\}\\
  \leq& ~C\left(1+(t-t_0)^{-\frac{\alpha}{1+\alpha}}\right)
\end{align*}
for all $t\in (t_0,\min\{T,t_0+\tau\})$, where $C$ depends only on $n,k,{\gamma},\alpha, \mathrm{Vol}(K_0)$ and $V_{n+1-k}(K_0,W)$. The upper bound of ${E}_k^{\alpha/k}=F_*^{-\alpha}$ follows from the bound on $Z$ and the fact $s(z,t)\leq 2R_2$. Proposition \ref{s5:thm-Ek-ub} follows since $t_0$ is arbitrary.
\endproof

Proposition \ref{s5:thm-Ek-ub} implies the following estimate on the global term $\phi(t)$.
\begin{cor}\label{cor-phi-bd}
Let $M_t$ be a smooth convex solution of the flow \eqref{flow-VMCF} on $[0,T)$ with the global term $\phi(t)$ given by \eqref{s1:phi-1}. Then for any $p>0$ we have
\begin{equation}\label{s5:cor-eqn1}
  0<C_1\leq \frac 1{|M_t|_{\gamma}}\int_{M_t}{E}_k^pd\mu_{\gamma}\leq C_2
\end{equation}
on $[0,T)$, where the constants $C_1,C_2$ depend only on $n,k,{\gamma},\alpha,p, \mathrm{Vol}(K_0)$ and $V_{n+1-k}(K_0,W)$. In particular,
\begin{equation}\label{s5:cor-eqn2}
  C_1\leq~\phi(t)\leq C_2
\end{equation}
on $[0,T)$.
\end{cor}
\proof
The upper bound in \eqref{s5:cor-eqn1} follows from the upper bound on $E_k$. For the lower bound in \eqref{s5:cor-eqn1}, if $0<p\leq 1$, by the Alexandrov--Fenchel inequality \eqref{AF-k=n+1} and the upper bound on $E_k$,
 \begin{align*}
   (n+1)^{n+1}\mathrm{Vol}(K_t)^{\frac{n-k}{n+1}}\mathrm{Vol}(W)^{\frac{k+1}{n+1}}\leq& \int_{M_t}E_kd\mu_\gamma\leq \sup_{M_t}E_k^{1-p}\int_{M_t}E_k^pd\mu_\gamma.
 \end{align*}
Since the enclosed $K_t$ has fixed volume, the lower bound in \eqref{s5:cor-eqn1} follows from the above inequality and the upper bound on $|M_t|_{\gamma}$ in \eqref{s5:AreaVol}. If $p>1$, the lower bound follows similarly by using the following inequality
  \begin{align*}
   (n+1)^{n+1}\mathrm{Vol}(K_t)^{\frac{n-k}{n+1}}\mathrm{Vol}(W)^{\frac{k+1}{n+1}}\leq \int_{M_t}E_kd\mu_\gamma\leq& ~ \left(\int_{M_t}E_k^pd\mu_\gamma\right)^{1/p}|M_t|_{\gamma}^{1-\frac 1p}.
 \end{align*}
\endproof

%\subsection{Long time existence}
Let $[0,T)$ be the maximal interval such that the smooth solution $M_t$ of the flow \eqref{flow-VMCF} exists.
\begin{prop}\label{s5:prop-conv}
If $M_0$ is strictly convex, then the solution $M_t$ of the flow \eqref{flow-VMCF} is strictly convex for all $t\in [0,T)$. If $T<\infty$, there exists a constant $C$ depending on $M_0,{\gamma}$ and $T$ such that the anisotropic principal curvatures of $M_t$ satisfy $\kappa_i\geq C>0$ on $t\in [0,T)$.
\end{prop}
\proof
To prove this estimate, we use the anisotropic Gauss map parametrization of the flow \eqref{flow-VMCF} as described in \S \ref{sec:gauss map}. The flow \eqref{flow-VMCF} is equivalent to the scalar parabolic equation \eqref{flow-gauss} on $\Sigma$ of the anisotropic support function $s(z,t)$. Since the eigenvalues of $(\tau_{ij})$ are equal to the reciprocal of the anisotropic principal curvatures $\kappa$, in order to estimate the lower bound of $\kappa$, it suffices to estimate the upper bound of the eigenvalues of $\tau_{ij}$.

For any given time $t_0\in [0,T)$, we choose the origin as $p_0$ such that the estimate \eqref{s9:s} holds. Then Lemma \ref{s8:lem1} implies that the anisotropic support function $s(z,t)$ of $M_t$ with respect to the origin $p_0$ satisfies
\begin{equation}\label{s9.s-bd}
  R_1/2\leq s(z,t)\leq 2R_2,\qquad \forall~t\in [t_0,\min\{T,t_0+\tau\})
\end{equation}
for some fixed constant $\tau$ depending only on $n,k,\alpha,\gamma,M_0$. Suppose $e_1$ is the direction where the maximum eigenvalue of $(\tau_{ij})$ occurs. By the evolution equation \eqref{s4:flow-tau2} of $\tau_{ij}$, we have
\begin{align}\label{s5:evl-tau}
  \frac{\partial}{\partial t}\tau_{11}\leq &~\dot{\Psi}^{k\ell}\left(\bar{\nabla}_k\bar{\nabla}_\ell\tau_{11}-\frac 12Q_{k\ell p}\bar{\nabla}_p\tau_{11}\right)+\ddot{\Psi}^{k\ell,pq}\bar{\nabla}_1\tau_{k\ell}\bar{\nabla}_1\tau_{pq}\nonumber\\
  &\quad +(1-\alpha)\Psi+\phi(t)+C_1\dot{\Psi}^{k\ell}\bar{g}_{k\ell}\tau_{11},
\end{align}
where $C_1=C_1(Q,\bar{\nabla}Q)$ is a constant depending only on $Q, \bar{\nabla}Q$. Applying Theorem \ref{s5:thm-Ek-ub} and Corollary \ref{cor-phi-bd}, we have $(1-\alpha)\Psi+\phi(t)\leq C$ for some uniform constant $C>0$. However, $\dot{\Psi}^{k\ell}\bar{g}_{k\ell}=\alpha F_*^{-\alpha-1}\dot{F}_*^{k\ell}\bar{g}_{k\ell}$, which has no control from above. Therefore, the parabolic maximum principle for tensors can not be applied directly to \eqref{s5:evl-tau} to deduce the upper bound of $\tau_{11}$.

To overcome this problem, we define
\begin{equation*}
  \zeta= \sup\{\tau_{ij}\xi^i\xi^j:~|\xi|=1\}
\end{equation*}
and consider
\begin{equation}\label{s9.omege-def}
  \omega(z,t) =\frac{\zeta(z,t)}{(s(z,t)-\frac{R_1}4)^a},
\end{equation}
where the power $a>0$ will be determined later. The function $\omega$ is well defined on the time interval $[t_0,\min\{T,t_0+\tau\})$. We consider the spatial maximum point $(z_1,t_1)$ where $\omega(z_1,t_1)=\max_\Sigma\omega(\cdot,t_1)$ for each $t_1\in [t_0,\min\{T,t_0+\tau\})$. By rotation of the local orthonormal frame, we assume that $\xi=e_1$ and $(\tau_{ij})=\mathrm{diag}(\tau_1,\dots,\tau_n)$ is diagonal at $(z_1,t_1)$. Then we have $\zeta=\tau_{11}$ at $(z_1,t_1)$. Combining \eqref{s4:s-1} and \eqref{s5:evl-tau} gives
\begin{align}\label{s9.omeg-1}
   &\frac{\partial }{\partial t}\omega -\dot{\Psi}^{k\ell}\left(\bar{\nabla}_k\bar{\nabla}_\ell\omega-\frac 12Q_{k\ell p}\bar{\nabla}_p\omega\right)-\frac{2}{(s-\frac{R_1}4)^a}\dot{\Psi}^{k\ell}\bar{\nabla}_k\omega\bar{\nabla}_\ell(s-\frac{R_1}4)^a\nonumber\\
   =&\frac 1{(s-\frac{R_1}4)^a}\left(\frac{\partial }{\partial t}\tau_{11} -\dot{\Psi}^{k\ell}\left(\bar{\nabla}_k\bar{\nabla}_\ell\tau_{11}-\frac 12Q_{k\ell p}\bar{\nabla}_p\tau_{11}\right)\right)\nonumber\\
   & -\frac{a\omega}{s-\frac{R_1}4} \left(\frac{\partial }{\partial t}s -\dot{\Psi}^{k\ell}\left(\bar{\nabla}_k\bar{\nabla}_\ell s-\frac 12Q_{k\ell p}\bar{\nabla}_ps\right)\right)+\frac{a(a-1)}{(s-\frac{R_1}4)^2}\omega\dot{\Psi}^{k\ell}\bar{\nabla}_ks\bar{\nabla}_\ell s\nonumber\\
   \leq & \frac 1{(s-\frac{R_1}4)^a}\left(\ddot{\Psi}^{k\ell,pq}\bar{\nabla}_1\tau_{k\ell}\bar{\nabla}_1\tau_{pq} +(1-\alpha)\Psi+\phi(t)+C_1\dot{\Psi}^{k\ell}\bar{g}_{k\ell}\tau_{11}\right)\nonumber\\
   &-\frac{a\omega}{s-\frac{R_1}4} \left(\phi(t)+(1+\alpha)\Psi+s\dot{\Psi}^{k\ell}\bar{g}_{k\ell}\right)+\frac{a(a-1)}{(s-\frac{R_1}4)^2}\omega\dot{\Psi}^{k\ell}\bar{\nabla}_ks\bar{\nabla}_\ell s\nonumber\\
  \leq &\frac 1{(s-\frac{R_1}4)^a}\ddot{\Psi}^{k\ell,pq}\bar{\nabla}_1\tau_{k\ell}\bar{\nabla}_1\tau_{pq}+\frac{a(a-1)}{(s-\frac{R_1}4)^2}\omega\dot{\Psi}^{k\ell}\bar{\nabla}_ks\bar{\nabla}_\ell s\nonumber\\
  &+\omega\sum_k\dot{\psi}^k\left(C_1-\frac{as}{s-\frac{R_1}4}\right)+\frac 1{(s-\frac{R_1}4)^a}((1-\alpha)\Psi+\phi(t))\nonumber\\
  &-\frac{a\omega}{s-\frac{R_1}4}\left(\phi(t)+(1+\alpha)\Psi\right).
\end{align}

We estimate gradient terms in \eqref{s9.omeg-1}. Since $\Psi$ is concave, we have $\ddot{\psi}\leq 0$ and $(\dot{\psi}^k-\dot{\psi}^{\ell})(\tau_k-\tau_\ell)\leq 0$. Using the equation \eqref{s2:F-ddt} we estimate the first term on the right-hand side of \eqref{s9.omeg-1}
\begin{align*}%\label{s9.LTE1}
  \ddot{\Psi}^{k\ell,pq}\bar{\nabla}_1\tau_{k\ell}\bar{\nabla}_1\tau_{pq}= &~\ddot{\psi}^{k\ell}\bar{\nabla}_1\tau_{kk}\bar{\nabla}_1\tau_{\ell\ell}+2\sum_{k>\ell}\frac{\dot{\psi}^k-\dot{\psi}^\ell}{\tau_k-\tau_\ell}(\bar{\nabla}_1\tau_{k\ell})^2\nonumber\\
   \leq &~ -2\sum_{k>1}\tau_1^{-1}(\dot{\psi}^k-\dot{\psi}^1)(\bar{\nabla}_1\tau_{k1})^2.
\end{align*}
Codazzi equation \eqref{s4:Codaz} implies that
\begin{align*}%\label{s9.LTE2}
  (\bar{\nabla}_1\tau_{k1})^2= & \left( \bar{\nabla}_k\tau_{11}+\frac 12Q_{11p}\tau_{kp}-\frac 12Q_{k1p}\tau_{1p}\right)^2
  \geq   \frac 12(\bar{\nabla}_k\tau_{11})^2-C(\tau_1)^2.
\end{align*}
Since $\bar{\nabla}_k\omega=0$ at $(z_1,t_1)$, we have
\begin{equation}\label{s9.LTE3}
  \frac{\bar{\nabla}_k\tau_{11}}{\tau_{11}}=\frac{a\bar{\nabla}_ks}{s-\frac{R_1}4}.
\end{equation}
Then the gradient terms on the right-hand side of \eqref{s9.omeg-1} satisfy
\begin{align}\label{s9.LTE4}
& \frac 1{(s-\frac{R_1}4)^a}\ddot{\Psi}^{k\ell,pq}\bar{\nabla}_1\tau_{k\ell}\bar{\nabla}_1\tau_{pq}+\frac{a(a-1)}{(s-\frac{R_1}4)^2}\omega\dot{\Psi}^{k\ell}\bar{\nabla}_ks\bar{\nabla}_\ell s\nonumber\\
  = & \omega\biggl(\frac 1{\tau_{11}}\ddot{\Psi}^{k\ell,pq}\bar{\nabla}_1\tau_{k\ell}\bar{\nabla}_1\tau_{pq}+\frac{a-1}a\frac 1{(\tau_{11})^2}\sum_k\dot{\psi}^k(\bar{\nabla}_k\tau_{11})^2\biggr)\nonumber\\
   \leq &~\omega\biggl(-2\sum_{k>1}\tau_{11}^{-2}(\dot{\psi}^k-\dot{\psi}^1)\left(\frac 12(\bar{\nabla}_k\tau_{11})^2-C(\tau_1)^2\right)+\frac{a-1}a\frac 1{(\tau_{11})^2}\sum_k\dot{\psi}^k(\bar{\nabla}_k\tau_{11})^2\biggr)\nonumber\\
= &~\omega\biggl(\dot{\psi}^1\tau_{11}^{-2}\sum_k(\bar{\nabla}_k\tau_{11})^2+2C\sum_{k>1}(\dot{\psi}^k-\dot{\psi}^1)\biggr)\nonumber\\
   \leq &~\omega\biggl(\dot{\psi}^1a^2\frac{|\bar{\nabla}s|^2}{(s-\frac{R_1}4)^2}+C_2\sum_k\dot{\psi}^k\biggr)
\end{align}
at $(z_1,t_1)$.

Substituting \eqref{s9.LTE4} into \eqref{s9.omeg-1}, we have
\begin{align}\label{s9.omeg-3}
\frac{\partial }{\partial t}\omega\leq ~ &\omega\sum_k\dot{\psi}^k\left(C_1+C_2-\frac{as}{s-\frac{R_1}4}\right)+\omega\dot{\psi}^1a^2\frac{|\bar{\nabla}s|^2}{(s-\frac{R_1}4)^2}\nonumber\\
  &+\frac 1{(s-\frac{R_1}4)^a}((1-\alpha)\Psi+\phi(t))-\frac{a\omega}{s-\frac{R_1}4}\left(\phi(t)+(1+\alpha)\Psi\right)
\end{align}
at $(z_1,t_1)$. Since $s(z,t)$ satisfies the estimate \eqref{s9.s-bd} on the time interval $[t_0,\min\{T,t_0+\tau\})$, by choosing a suitable constant $a>0$ depending on $R_1,R_2$ and $C_1, C_2$ we can ensure that the first term on the right-hand side of \eqref{s9.omeg-3} is negative. This kills the bad term $\sum_k\dot{\psi}^k$. The second term on the right-hand side of \eqref{s9.omeg-3} is bounded, since $|\bar\nabla s|$ is bounded by \eqref{s9:ds-1} and
\begin{equation*}
 \omega\dot{\psi}^1=\frac{\dot{\psi}^1\tau_1}{(s-\frac{R_1}4)^a}\leq \frac{ \sum_{i=1}^n\dot{\psi}^i\tau_i}{(s-\frac{R_1}4)^a}=\frac{\alpha {E}_k^{\alpha/k}}{(s-\frac{R_1}4)^a}\leq C.
\end{equation*}
Then at the spatial maximum point $(z_1,t_1)$, the function $\omega$ satisfies
\begin{equation}\label{s9.omeg-2}
\frac{\partial }{\partial t}\omega\leq C_3+C_4\omega,
\end{equation}
where $C_3$, $C_4$ are constants depending on $a, R_1,R_2$, the bound on $\phi(t)$ in Corollary \ref{cor-phi-bd} and the upper bound on $-\Psi=E_k^{\alpha/k}$ in Proposition \ref{s5:thm-Ek-ub}. Applying the parabolic maximum principle to \eqref{s9.omeg-2}, we see that the function $\omega$ in the time interval $[t_0,\min\{T,t_0+\tau\})$ is bounded by a constant depending on its value at the time $t_0$ and on the parameter $\tau$. Since $0<R_1/2\leq s(z,t)\leq 2R_2$ for all $t\in [t_0,\min\{T,t_0+\tau\})$, this also implies that the largest anisotropic principal radius $\tau_1$ is bounded by a constant depending on its value at the time $t_0$ and on the parameter $\tau$. If the maximal existence time $T$ is finite, after a finite number of  iterations, we conclude that the largest anisotropic principal radius $\tau_1$ is bounded from above by a positive constant $C$ depending on the maximal existence time $T$. This completes the proof of Proposition \ref{s5:prop-conv}.
\endproof

Now we can prove the long time existence of the flow \eqref{flow-VMCF}, i.e., the maximum existence time $T=\infty$.
\begin{thm}
Let $W$ be a Wulff shape in $\mathbb{R}^{n+1}$ with the smooth support function $\gamma\in C^{\infty}(\mathbb{S}^n)$, and $X_0: M^n\to \mathbb{R}^{n+1}$ be a smooth embedding such that $M_0=X_0(M)$ is a closed strictly convex hypersurface in $\mathbb{R}^{n+1}$ enclosing a convex body $K_0$. Then for
any $k\in \{1,\dots,n\}$ and $\alpha>0$, the volume preserving flow \eqref{flow-VMCF} with the global term $\phi(t)$ given by \eqref{s1:phi-1} has a smooth strictly convex solution $M_t$ for all time $t\in [0,\infty)$.
\end{thm}
\proof Suppose that the maximal existence time $T$ is finite. By Proposition \ref{s5:thm-Ek-ub} and Proposition \ref{s5:prop-conv}, there exists a constant $C>0$ depending on $M_0,{\gamma}$ and $T$ such that the anisotropic principal curvatures satisfy
\begin{equation*}
  0<\frac 1C\leq \kappa_i\leq C,\qquad i=1,\cdots,n
\end{equation*}
on $M_t$ for $t\in [0,T)$. This together with $C^0$ and $C^1$ estimates implies the $C^2$ estimate of the evolving hypersurface $M_t$.  To derive the higher regularity estimate, we note that the flow \eqref{flow-VMCF} is equivalent to the equation \eqref{flow-gauss} for the anisotropic support function $s(z,t)$. Since $F_*$ is concave and $\alpha>0$, the right-hand side of \eqref{flow-gauss} is concave with respect to the spatial second derivatives of $s(z,t)$. As the global term $\phi(t)$ is bounded, Theorem 1.1 of  \cite{TW13} can be applied to derive the $C^{2,\beta}$ estimate of $s$ for some $\beta\in (0,1)$. The higher order regularity estimate of $M_t$ follows from the standard parabolic Schauder theory \cite{Lieb96}. Thus $M_t$ converges to a hypersurface $M_{T}$ smoothly as $t\to T$. The short time existence theorem yields that the flow \eqref{flow-VMCF} can be extended past $T$ for a short time interval, contradicting the maximality of $T$.
\endproof

\section{Hausdorff convergence}\label{sec:HC}

We showed in the previous section that the smooth solution $M_t$ of the flow \eqref{flow-VMCF} exists for all time $t\in [0,\infty)$. In this section, using Theorem \ref{thm-main} and Theorem \ref{s3:thm-W} for the anisotropic curvature measures and applying the argument as in \cite[\S 6]{AW17}, we show that the solution $M_t$ converges to a scaled Wulff shape as $t\to\infty$ in the Hausdorff sense. As applications, we further improve the estimate on the global term $\phi(t)$ and obtain the uniform positive lower bound on the $k$th anisotropic mean curvature $E_k$.

\subsection{Hausdorff convergence to the Wulff shape}$\ $

Let
\begin{equation*}
  \bar{E}_k=~\frac 1{|M_t|_{\gamma}}\int_{M_t}{E}_kd\mu_{\gamma}=~\frac{V_{n-k}(K_t,W)}{V_n(K_t,W)}
\end{equation*}
denote the average integral of $E_k$. By Proposition \ref{s5:prop-monot}, the mixed volume $V_{n+1-k}(K_t,W)$ is monotone non-increasing in time. Since $V_{n+1-k}(K_t,W)$ is bounded from below, by the long-time existence of the flow, the equation \eqref{s5:Vn1k} implies that
\begin{equation*}
  \int_0^\infty\int_{M_t} \left(E_k^{\alpha/k}- \bar{E}_k^{\alpha/k}\right)\left(E_k-\bar{E}_k\right)d\mu_\gamma~\leq~V_{n+1-k}(K_0,W)<\infty.
\end{equation*}
Therefore there exists a sequence of times $t_i\to \infty$ such that
\begin{align}\label{s10.1}
\int_{M_{t_i}} \left(E_k^{\alpha/k}- \bar{E}_k^{\alpha/k}\right)\left(E_k-\bar{E}_k\right)d\mu_\gamma~\to 0.
\end{align}
%By the definition of $\phi(t)$,
%\begin{align}\label{s10:Vn1k2}
%  \int_{M_t}{E}_{k}(\phi(t)-{E}_k^{{\alpha}/k})d\mu_{\gamma}=& \frac 1{|M_t|_\gamma}\int_{M_t}E_kd\mu_\gamma\int_{M_t}E_k^{\alpha/k}d\mu\gamma-\int_{M_t}E_k^{1+\frac{\alpha}k}d\mu_\gamma\nonumber \\
% = & \int_{M_t} E_k^{\alpha/k}\left(\bar{E}_k-E_k\right)d\mu_\gamma\nonumber\\
% =& -\int_{M_t} \left(E_k^{\alpha/k}- \bar{E}_k^{\alpha/k}\right)\left(E_k-\bar{E}_k\right)d\mu_\gamma.
%\end{align}
If $0<\alpha<k$, we have
\begin{align*}
  \left(E_k^{\alpha/k}- \bar{E}_k^{\alpha/k}\right)\left(E_k-\bar{E}_k\right) =& \frac{\alpha}k\int_0^1\left((1-s)\bar{E}_k+sE_k\right)^{\alpha/k-1}ds\cdot (E_k-\bar{E}_k)^2 \\
  \geq & \frac{\alpha}k\left(\sup_{M_t}E_k\right)^{\alpha/k-1}(E_k-\bar{E}_k)^2\\
  \geq & C  (E_k-\bar{E}_k)^2
\end{align*}
by the upper bound on $E_k$ that proved in Proposition \ref{s5:thm-Ek-ub}. On the other hand, for $\alpha> k$ we have
\begin{align*}
  \left(E_k^{\alpha/k}- \bar{E}_k^{\alpha/k}\right)\left(E_k-\bar{E}_k\right) =& \frac{\alpha}k\int_0^1\left((1-s)\bar{E}_k+sE_k\right)^{\alpha/k-1}ds\cdot (E_k-\bar{E}_k)^2 \\
  \geq & \frac{\alpha}k\int_0^1(1-s)^{\alpha/k-1}ds \bar{E}_k^{\alpha/k-1}(E_k-\bar{E}_k)^2\\
  \geq & C  (E_k-\bar{E}_k)^2
\end{align*}
by Corollary \ref{cor-phi-bd}. Therefore, \eqref{s10.1} implies that
\begin{align}\label{s10.2}
 \int_{M_{t_i}}\left({E}_k-\bar{E}_k\right)^2d\mu_{\gamma}~\to&~0, \quad\mathrm{ as}~i\to\infty
\end{align}
for a sequence of times $t_i\to \infty$.
\begin{prop}\label{s10.lem1}
Let $M_t$ be the smooth solution to the flow \eqref{flow-VMCF} with the global term $\phi(t)$ given by \eqref{s1:phi-1}. Denote the closure of the enclosed domain of $M_t$ by $K_t$. Then $K_{t}$ converges to a scaled Wulff shape $\bar{r}W$ in the Hausdorff sense as $t\to\infty$,  where $\bar r$ is the radius such that $\mathrm{Vol}(K_0)=\bar{r}^{n+1}\mathrm{Vol}(W)$.
\end{prop}
\proof
The proof is similar to Lemma 6.3 in \cite{AW17}. We include it here for completeness. By the estimate \eqref{s5:radius} on the anisotropic outer radius of $K_t$, the Blaschke selection theorem (see Theorem 1.8.7 of \cite{Schn}) implies that there exists a subsequence of times $t_i$ and a convex body $\hat{K}$ such that $K_{t_i}$ converges to $\hat{K}$ in the Hausdorff sense as $t_i\to\infty$. As each $K_{t_i}$ has the  anisotropic inner radius $r(K_{t_i})\geq R_1$, the limit convex body $\hat{K}$ has a positive anisotropic inner radius. Without loss of generality, we may assume that the sequence $t_i$ is the same sequence such that \eqref{s10.2} holds.

We will show that the limit convex body $\hat{K}$ satisfies $\Phi_{n-k}(\hat K;.) = c\,\Phi_{n}(\hat K;.)$ where $c= V_{n-k}(\hat K,W)/V_n(\hat K,W)$.  The weak continuity of the anisotropic curvature measure $\Phi_m$ proved in Theorem \ref{s3:thm-W} is equivalent to the statement that $\int f d\Phi_m(K_i)$ converges to $\int f d\Phi_m(K)$ whenever $f$ is a bounded continuous function on ${\mathbb R}^{n+1}$ and $K_i$ is a sequence of convex sets converging to $K$ in the Hausdorff distance.  In particular we get that $\int f d\Phi_{n-k}(K_{t_i})$ converges to $\int f d\Phi_{n-k}(\hat K)$, and $\int f d\Phi_{n}(K_{t_i})$ converges to $\int f d\Phi_{n}(\hat K)$, as $i\to\infty$.  Since $K_{t_i}$ is smooth and uniformly convex, by \eqref{s3:Phi-s} we obtain for any bounded continuous $f$,
\begin{align*}
\left|\int f d\Phi_{n-k}(K_{t_i}) -c \int f d\Phi_n(K_{t_i})\right|
&= \left|\int_{M_{t_i}} f E_k \gamma(\nu)d{\mathcal H}^n-\int_{M_{t_i}} f c \gamma(\nu)d{\mathcal H}^n\right|\\
&\leq \sup |f| \int_{M_{t_i}}|E_k-c| \gamma(\nu)d{\mathcal H}^n\\
&\leq \sup|f|\int_{M_{t_i}}|E_k-\bar E_k|\gamma(\nu)d{\mathcal H}^n\\
&\quad\null+\sup|f|V_n(K_{t_i},W)\left|\frac{V_{n-k}(K_{t_i},W)}{V_n(K_{t_i},W)}-\frac{V_{n-k}(\hat K,W)}{V_n(\hat K,W)}\right|.
\end{align*}
The left-hand side converges to $\left|\int f d\Phi_{n-k}(\hat K)-c\int f d\Phi_n(\hat K)\right|$ by the weak continuity of the anisotropic curvature measures, while the first term on the right-hand side converges to zero by \eqref{s10.2}, and the second also does by the continuity of the mixed volumes with respect to the Hausdorff distance.  It follows that
\begin{equation*}
  \int fd\Phi_{n-k}(\hat K) = c\int fd\Phi_n(\hat K)
\end{equation*}
for all bounded continuous functions $f$, and therefore that $\Phi_{n-k}(\hat K;.) = c\,\Phi_n(\hat K;.)$ as claimed. By Theorem \ref{thm-main}, the convex body $\hat K$ is a scaled Wulff shape.

By Proposition \ref{s5:prop-monot}, ${\mathcal I}_{n+1-k}(K_t,W)$ is non-increasing in time. Since ${\mathcal I}_{n+1-k}(K_{t_i},W)\to 1$ as $i\to\infty$ for some subsequence of times $t_i\to\infty$, we conclude that ${\mathcal I}_{n+1-k}(K_t,W)\to 1$ as $t\to\infty$.  It follows from the stability estimate (7.124) in \cite{Schn} that the whole family of $K_{t}$ converges to a scaled Wulff shape $\bar{r}W$ as $t\to\infty$, where $\bar r$ is the radius such that $\mathrm{Vol}(K_0)=\bar{r}^{n+1}\mathrm{Vol}(W)$. This completes the proof of Proposition \ref{s10.lem1}.
\endproof

\subsection{Improved estimate on the global term $\phi(t)$}$\ $

The Hausdorff convergence of the solution has the following consequence on the global term $\phi(t)$ in the flow equation \eqref{flow-VMCF}.
\begin{cor}\label{s10.cor}
If $\alpha\geq k$ in the flow equation \eqref{flow-VMCF}, the global term $\phi(t)$ given by \eqref{s1:phi-1} satisfies
\begin{equation}\label{s10.phi1}
  \liminf_{t\to\infty} \phi(t)\geq \bar{r}^{-\alpha},
\end{equation}
where $\bar r$ is the radius such that $\mathrm{Vol}(K_0)=\bar{r}^{n+1}\mathrm{Vol}(W)$. In general, for any $\alpha>0$, the global term $\phi(t)$ satisfies
 \begin{equation}\label{s10.phi2}
   \int_t^{t+d}\biggl|\phi(s)-\frac 1{\bar{r}^{\alpha}}\biggr|ds ~\to~ 0
 \end{equation}
 as $t\to\infty$ for any fixed constant $d>0$.
\end{cor}
\proof
If $\alpha\geq k$, then by the H\"{o}lder inequality we get
\begin{equation*}
  \phi(t)=\frac 1{|M_t|_{\gamma}}\int_{M_t}{E}_k^{\alpha/k}d\mu_{\gamma}\geq \left(\frac{\int_{M_t}{E}_k d\mu_{\gamma}}{|M_t|_{\gamma}}\right)^{\alpha/k}=\left(\frac{V_{n-k}(K_t,W)}{V_n(K_t,W)}\right)^{\alpha/k}.
\end{equation*}
Since $K_t$ converges to a scaled Wulff shape $\bar{r}W$ in the Hausdorff sense as $t\to\infty$, the continuity of the mixed volumes implies the estimate \eqref{s10.phi1}.

To show \eqref{s10.phi2} for a general $\alpha>0$, as at the beginning of this section we note that
%\begin{align*}
% E_k^{\alpha/k}- \bar{E}_k^{\alpha/k}=& \frac{\alpha}k\int_0^1\left((1-s)\bar{E}_k+sE_k\right)^{\alpha/k-1}ds\cdot (E_k-\bar{E}_k).
%\end{align*}
%This implies that
\begin{align*}
 \left(E_k^{\alpha/k}- \bar{E}_k^{\alpha/k}\right)^2=& \frac{\alpha}k\int_0^1\left((1-s)\bar{E}_k+sE_k\right)^{\alpha/k-1}ds\cdot (E_k-\bar{E}_k) \left(E_k^{\alpha/k}- \bar{E}_k^{\alpha/k}\right).
\end{align*}
If $\alpha>k$,
\begin{align*}
  \frac{\alpha}k\int_0^1\left((1-s)\bar{E}_k+sE_k\right)^{\alpha/k-1}ds \leq  &\frac{\alpha}{k} \left(\sup_{M_t}E_k\right)^{\alpha/k-1}.
\end{align*}
If $0<\alpha<k$,
\begin{align*}
  \frac{\alpha}k\int_0^1\left((1-s)\bar{E}_k+sE_k\right)^{\alpha/k-1}ds \leq  & \bar{E}_k^{\alpha/k-1} \frac{\alpha}k\int_0^1(1-s)^{\alpha/k-1}ds.
\end{align*}
In both cases, we obtain
\begin{align}\label{s10.pf1}
 \left(E_k^{\alpha/k}- \bar{E}_k^{\alpha/k}\right)^2\leq & C (E_k-\bar{E}_k) \left(E_k^{\alpha/k}- \bar{E}_k^{\alpha/k}\right),
\end{align}
where $C=C(n,k,\alpha,\gamma,K_0)$ comes from the bounds on $E_k, \bar{E}_k$ in Proposition~\ref{s5:thm-Ek-ub} and Corollary~\ref{cor-phi-bd}. Since $V_{n+1-k}(K_t,W)$ is non-increasing and converges to the value $V_{n+1-k}(\bar{r}W,W)$, for any small $\varepsilon>0$, there exists a large time $N_0$ such that for all times $t_1>t_2> N_0$,
\begin{equation*}
 V_{n+1-k}(K_{t_2},W)-V_{n+1-k}(K_{t_1},W)\leq \varepsilon.
\end{equation*}
In particular, for any fixed constant $d>0$ the estimates \eqref{s5:Vn1k} and \eqref{s10.pf1} imply that
%\begin{align*}
%&  \int_t^{t+d}\int_{M_s}{E}_{k}({E}_k^{{\alpha}/k}-\phi(t))d\mu_{\gamma}ds\\
%=&\frac 1{n+1-k}\left(V_{n+1-k}(K_{t},W)-V_{n+1-k}(K_{t+d},W)\right)\leq \frac {\varepsilon}{n+1-k}.
%\end{align*}
% The equations \eqref{s10:Vn1k2} and \eqref{s10.cor-2} imply that
\begin{align}\label{s10.pf2}
\int_t^{t+d}\int_{M_s}\left(E_k^{\alpha/k}- \bar{E}_k^{\alpha/k}\right)^2d\mu_\gamma ds\leq  & C\varepsilon
\end{align}
for all $t>N_0$, where $C$ depends on $n,k,\alpha,\gamma,K_0$.  Applying the H\"{o}lder inequality
\begin{align*}%\label{s10.pf2}
  \biggl|\phi(t)-\bar{E}_k^{\alpha/k}\biggr| =& \frac 1{|M_t|_\gamma}\biggl|\int_{M_t}\left(E_k^{\alpha/k} - \bar{E}_k^{\alpha/k}\right)d\mu_\gamma\biggr|\nonumber\\
  \leq & \left(\frac 1{|M_t|_\gamma}\int_{M_t}\left(E_k^{\alpha/k} - \bar{E}_k^{\alpha/k}\right)^2d\mu_\gamma\right)^{1/2}
\end{align*}
and the uniform bounds on $|M_t|_\gamma$, we have
\begin{align}\label{s10.pf3}
\int_t^{t+d}\biggl|\phi(s)-\bar{E}_k^{\alpha/k}\biggr|ds\leq  & C\varepsilon^{1/2}.
\end{align}
By Proposition \ref{s10.lem1}, $K_t$ converges to a scaled Wulff shape $\bar{r}W$ as $t\to\infty$. The continuity of the mixed volume $V_{n-k}(\cdot,W)$ with respect to the Hausdorff distance implies that
\begin{align*}
  \bar{E}_k=&~\frac{V_{n-k}(K_t,W)}{V_n(K_t,W)} \to ~\frac{V_{n-k}(\bar{r}W,W)}{V_n(\bar{r}W,W)}=\bar{r}^{-k},\quad \mathrm{as} ~t\to\infty.
\end{align*}
Then we conclude from \eqref{s10.pf3} that for any fixed constant $d>0$ and any small  $\varepsilon>0$, there exists a large $N_0$ such that
\begin{align}\label{s10.pf4}
\int_t^{t+d}\biggl|\phi(s)-\bar{r}^{-\alpha}\biggr|ds\leq  & C\varepsilon^{1/2}
\end{align}
for all $t>N_0$, where $C$ is a uniform constant depending only on $n,k,\alpha,\gamma$ and $K_0$.
\endproof

\subsection{Lower bound on ${E}_k$}\label{sec:LB}$\ $

In this subsection, we prove a uniform positive lower bound on the $k$th anisotropic mean curvature $E_k$. The key is the Hausdorff convergence of $K_t$ to a scaled Wulff shape $\bar{r}W$ as $t\to\infty$. We also use the uniform two-sided positive bounds on the global term $\phi(t)$.
\begin{prop}\label{s11.prop1}
Let $M_t$, $t\in [0,\infty),$ be a smooth solution of the flow \eqref{flow-VMCF}.  Then there exists a positive constant $C=C(n,k,\alpha,\gamma,K_0)$ such that ${E}_k\geq C$ on $M_t$ for all $t\geq 0$.
\end{prop}
\proof
It suffices to control ${E}_k$ from below for sufficiently large times. By Proposition \ref{s10.lem1}, $K_t$ converges to a scaled Wulff shape $\bar{r}W$ as $t\to\infty$ in the Hausdorff sense. Then for any small $\epsilon>0$, there exists a sufficiently large time $N_0>0$ such that for all $t\geq N_0$, the anisotropic outer radius $R(t)$ and anisotropic inner radius $r(t)$ of $K_t$ relative to $W$ satisfy
\begin{equation}\label{s11.pf1}
  (1-\epsilon)\bar{r}~\leq~ r(t)\leq~R(t)\leq~(1+\epsilon)\bar{r}.
\end{equation}
In the following, we only consider times $t\geq N_0$.

Let $t_1>N_0$ and set
\begin{equation*}
  \varphi_{t_1}(t)=\int_{t_1}^t\phi(s)ds.
\end{equation*}
We use the anisotropic Gauss map parametrization described in \S \ref{sec:gauss map}. Choose the origin to be the center of the inner Wulff shape such that $r(t_1)W\subset K_{t_1}$.
Then
\begin{equation*}
s(z,t_1)\geq~s(z_1,t_1) =r(t_1)\geq~(1-\epsilon)\bar{r},
\end{equation*}
where $z_1$ is the minimum point of $s(z,t_1)$. Define a function $f(z,t)$ on $\Sigma=\partial W$ by
\begin{equation*}%\label{s7:f-def}
  f(z,t)~=~-(1+\alpha)(t-t_1)\Psi(z,t)-(1+\alpha)\varphi_{t_1}(t)+s(z,t)
\end{equation*}
for $t\geq t_1$. Combining \eqref{flow-gauss} and \eqref{s4:flow-Psi} gives the evolution equation of $f(z,t)$
\begin{align}\label{s11.pf2}
   \frac{\partial}{\partial t}f (z,t)= & ~\dot{\Psi}^{k\ell}(z,t)\bar{\nabla}_k\bar{\nabla}_\ell f(z,t)-\frac 12\dot{\Psi}^{k\ell}(z,t)Q_{k\ell p}\bar{\nabla}_pf-\alpha\phi(t) \nonumber\\ &\quad+\biggl(s(z,t)-(1+\alpha)(t-t_1)(\Psi+\phi(t))\biggr)\dot{\Psi}^{k\ell}\bar{g}_{k\ell}.
\end{align}
We now show that the last term on the right-hand side of \eqref{s11.pf2} is nonnegative for a short time after time $t_1$.  By Theorem \ref{s5:thm-Ek-ub}, $\Psi=-{E}_k^{\alpha/k}$ is bounded below by a uniform negative constant $-c_+$. On the other hand, by Corollary \ref{cor-phi-bd} the function $\phi(t)$ is bounded from above and below by positive constants:
\begin{equation*}
  0<\phi_-\leq \phi(t)\leq~\phi_+,
\end{equation*}
where $\phi_-,\phi_+$ depend only on  $n,k,{\gamma},\alpha$ and $K_0$.  Then integrating the evolution equation \eqref{flow-gauss} of $s(z,t)$ gives that
\begin{align*}
s(z,t)\geq & ~s(z,t_1)+(\phi_--c_+)(t-t_1)
   \geq  (1-\epsilon)\bar{r}+(\phi_--c_+)(t-t_1).
\end{align*}
Let $t_2>t_1$ be the time such that
\begin{equation}\label{s11.pf3}
  t_2-t_1=~\frac{(1-\epsilon)\bar{r}}{c_+-\phi_-+(1+\alpha)\phi_+},
\end{equation}
which is positive and independent of the time $t_1$. Then for any $t\in [t_1,t_2]$, we have
\begin{align}\label{s11.pf4}
&s(z,t)-(1+\alpha)(t-t_1)(\Psi(z,t)+\phi(t)) \nonumber\\
 & \geq  (1-\epsilon)\bar{r}+(\phi_--c_+)(t-t_1)-(1+\alpha)(t-t_1)\phi_+\geq~0.
\end{align}

Since $\dot{\Psi}^{k\ell}\bar{g}_{k\ell}\geq 0$, combining \eqref{s11.pf2} with \eqref{s11.pf4} gives that
\begin{align*}%\label{s7:dt-f2}
   \frac{\partial}{\partial t}f (z,t)
   \geq&~\dot{\Psi}^{k\ell}(z,t)\bar{\nabla}_k\bar{\nabla}_\ell f(z,t)-\frac 12\dot{\Psi}^{k\ell}(z,t)Q_{k\ell p}\bar{\nabla}_pf-\alpha\phi(t)
\end{align*}
for all time $t\in [t_1,t_2]$. The maximum principle implies that
\begin{align*}%\label{s7:s2}
  f(z,t) \geq & \min_{M_{t_1}}f(z,t_1)-\alpha\varphi_{t_1}(t)   =~s(z_1,t_1)-\alpha\varphi_{t_1}(t)
\end{align*}
for all time $t\in [t_1,t_2]$, which implies that
 \begin{equation}\label{s11.Psi-lbd1}
   -(1+\alpha)\Psi(z,t)\geq~\frac{\varphi_{t_1}(t)-s(z,t)+s(z_1,t_1)}{t-t_1}
 \end{equation}
for all $t\in (t_1,t_2]$ with $t_2$ defined in \eqref{s11.pf3}.

To estimate the lower bound of $-\Psi$, it remains to estimate the right-hand side of \eqref{s11.Psi-lbd1}. Recall that the origin is chosen such that $r(t_1)W\subset K_{t_1}$. Then by \cite[Lem.~7.2]{And01}, we have
\begin{equation}\label{s11.pf5}
 s(z, t_1)\leq R(t_1)+C(R(t_1)-r(t_1))\leq (1+(1+2C)\epsilon)\bar{r},\quad \forall~z\in\Sigma,
\end{equation}
where $C$ is a constant only depending on $\gamma$.

From the flow equation \eqref{flow-gauss}, we deduce that the spatial maximum $\bar{s}(t)=\max_{\Sigma}s(\cdot,t)$ of the anisotropic support function $s$ satisfies
\begin{equation}\label{s11.pf6}
   \frac d{dt}\bar{s}(t) =~-F_*({\tau}_{ij})^{-\alpha}+\phi(t)\leq ~\phi_+,
 \end{equation}
 which implies that
 \begin{align*}%\label{s7:rt-3}
 \bar{s}(t)\leq&~\bar{s}(t_1)+\phi_+(t-t_1)~\leq~ 2\bar{s}(t_1), \quad \forall~t\in [t_1,t_3],
 \end{align*}
where $t_3$ is given by $t_3-t_1= \phi_+^{-1}{\bar s}(t_1)$. By the definition of $\tau_{ij}$, at the maximum point of $s(\cdot,t)$, we have ${\tau}_{ij}\leq \bar{g}_{ij}s$. Therefore,
\begin{align*}
  \frac d{dt}\big(\varphi_{t_1}(t)- \bar{s}(t)+s(z_1,t_1)\big)=& F_*({\tau}_{ij})^{-\alpha}\\
  \geq &~ \bar{s}(t)^{-\alpha} \\
  \geq &~2^{-\alpha} \bar{s}(t_1)^{-\alpha}, \quad \forall~t\in [t_1,t_3],
\end{align*}
which implies that
\begin{align}\label{s11.pf7}
\varphi_{t_1}(t)-\bar{s}(t)+s(z_1,t_1)\geq &~-\bar{s}(t_1)+s(z_1,t_1) +2^{-\alpha}\bar{s}(t_1)^{-\alpha}(t-t_1)\nonumber\\
 \geq &~ 2^{-1-\alpha}\bar{s}(t_1)^{-\alpha}(t-t_1),
\end{align}
provided that
\begin{equation*}
  t\geq~t_1+2^{1+\alpha}r_+^{\alpha}(\bar{s}(t_1)-s(z_1,t_1))~=:t_4.
\end{equation*}
By \eqref{s11.pf1} and \eqref{s11.pf5}, we know that
\begin{equation*}
\bar{s}(t_1)-s(z_1,t_1)\leq ~2(1+C)\epsilon \bar{r}
\end{equation*}
can be arbitrarily small by assuming $N_0$ is sufficiently large. Then the waiting time $t_4-t_1$ can be made small such that $t_4<\min\{t_2,t_3\}$ by choosing $N_0$ sufficiently large. Combining \eqref{s11.Psi-lbd1} and \eqref{s11.pf7} yields that
 \begin{align*}%\label{s7:Psi-lbd3}
   -(1+\alpha)\Psi(z,t)\geq&~\frac{\varphi_{t_1}(t)-s(z,t)+s(z_1,t_1)}{t-t_1}\nonumber\\
   \geq&~2^{-1-\alpha}\bar{s}(t_1)^{-\alpha}\nonumber\\
   \geq&~2^{-1-\alpha}(1+(1+2C)\epsilon)^{-\alpha}\bar{r}^{-\alpha}
 \end{align*}
 for all time $t\in[t_4,\min\{t_2,t_3\}]$. This gives the uniform positive lower bound on ${E}_k$.
\endproof

\section{Smooth convergence}\label{sec:SC}
In this section, we prove the smooth convergence of the flow \eqref{flow-VMCF} and complete the proof of Theorem \ref{thm-2}. We will consider the three cases (i) $k=1$, $\alpha>0$; (ii) $k=n$, $\alpha>0$ and (iii) $\alpha\geq k=1,\dots,n$  separately.

\subsection{Smooth convergence for $k=1$, $\alpha>0$}$\ $

For $k=1$, from Proposition \ref{s5:thm-Ek-ub} we see that $nE_1=\sum_{i=1}^n\kappa_i\leq C$ for some constant $C>0$. Since $M_t$ is convex for each $t>0$ by Proposition \ref{s5:prop-conv}, we get that the anisotropic principal curvatures are bounded
\begin{equation*}
  0<\kappa_i\leq C,\quad i=1,\cdots,n
\end{equation*}
along the flow \eqref{flow-VMCF}. This is equivalent to the $C^2$ estimate of $M_t$. In the following lemma, we derive the higher regularity of $M_t$.
\begin{lem}
Let $M_t$ be a smooth solution to the flow \eqref{flow-VMCF} with $k=1$, $\alpha>0$. Then for any integer $\ell\geq 2$, there exists a constant $C_\ell$ depending only on $n,\alpha,\ell,\gamma$ such that
\begin{equation*}
  |\hat{\nabla}^\ell \mathcal{W}_{\gamma}|\leq C_\ell
\end{equation*}
for all $t>0$, where $\mathcal{W}_{\gamma}$ is the anisotropic Weingarten map of $M_t$.
\end{lem}
\proof
As in \cite[\S 8]{And01}, we write the evolving hypersurface locally as a graph $x(y^1,\dots,y^n,t)=y^ie_i+u(y^1,\dots,y^n,t)e_0$, where $\{e_i\}_{i=0}^n$ is the standard basis for $\R^{n+1}$. The tangent space of the graph is spanned by $\partial_i=e_i+u_ie_0$, where $u_i=\partial u/\partial y^i$, $i=1,\dots,n$. The isotropic normal vector field is
\begin{equation*}
  \nu=(1+|Du|^2)^{-1/2}(Du,-1).
\end{equation*}
Let $\gamma$ be the $1$-homogenous extension of the support function of the Wulff shape $\Sigma$. Then $\gamma\in C^{\infty}(\mathbb{R}^{n+1}\setminus \{0\})$. Define a function $\hat{\gamma}$ on $\mathbb{R}^n$ by
\begin{equation}\label{s10:F-hat}
  \hat{\gamma}(p_1,\dots,p_n)={\gamma}(p_1,\dots,p_n,-1).
\end{equation}
Then $\hat{\gamma}(Du)=\sqrt{1+|Du|^2}{\gamma}(\nu)$. Let $\omega=(Du,-1)\in \mathbb{R}^{n+1}$. By the $1$-homogeneity of $\gamma$,
\begin{align*}
  \hat{\gamma}(Du)=&\sum_{i=1}^nD^i\gamma\big|_\omega D_iu-D^0\gamma\big|_\omega
  =\sum_{i=1}^nD^i\hat{\gamma}\big|_{Du} D_iu-D^0\gamma\big|_\omega.
\end{align*}
Then the anisotropic unit normal $\nu_{\gamma}=D\gamma(\nu)$ can be expressed as
\begin{align}\label{s10:nu-F}
  \nu_\gamma=&\sum_{i=1}^nD^i\gamma\big|_\omega e_i+D^0\gamma\big|_\omega e_0\nonumber\\
  =&\sum_{i=1}^nD^i\hat{\gamma}|_{Du}e_i+\left(\sum_{i=1}^nD_iuD^i\hat{\gamma}|_{Du}-\hat{\gamma}(Du)\right)e_0.
\end{align}
Differentiate \eqref{s10:nu-F} with respect to $y^\ell$:
\begin{align*}
  \partial_{\ell}\nu_{\gamma}=&  u_{\ell i}D^iD^k\hat{\gamma}\big|_{Du}e_k+\left(u_{\ell i}D^i\hat{\gamma}|_{Du}+u_ku_{\ell i}D^i D^k\hat{\gamma}|_{Du}-u_{\ell i}D^i\hat{\gamma}\big|_{Du}\right)e_0\\
  =& u_{\ell i}D^iD^k\hat{\gamma}\big|_{Du}(e_k+u_ke_0) \\
   =& u_{\ell i}D^iD^k\hat{\gamma}\big|_{Du}\partial_kx.
\end{align*}
By the anisotropic Weingarten formula \eqref{s2:AWeingart}, we obtain
\begin{equation}\label{s10:AW}
  \hat{h}_{\ell}^k=u_{\ell i}D^iD^k\hat{\gamma}\big|_{Du}.
\end{equation}
Moreover, by \eqref{s2:sff} and \eqref{s10:nu-F}, the anisotropic second fundamental form $\hat{h}$ is given by
\begin{align*}
  \hat{h}_{ij} =& -G(\nu_{\gamma})(\nu_{\gamma},\partial_i\partial_jx)\\
  = & -G(\nu_{\gamma})(\nu_{\gamma},u_{ij}e_0)\\
  =&-u_{ij}G(\nu_{\gamma})\left(\nu_{\gamma},-\frac{\nu_{\gamma}}{\hat{\gamma}(Du)}+\frac{D^i\hat{\gamma}\big|_{Du}}{\hat{\gamma}(Du)}\partial_ix\right)\\
  =&\frac{u_{ij}}{\hat{\gamma}(Du)}.
\end{align*}
Then the metric $\hat{g}$ satisfies
\begin{align*}
\hat{g}^{ij}=\hat{\gamma}(Du)D^iD^j\hat{\gamma}|_{Du}.
\end{align*}
The equation \eqref{s10:AW} also implies that the normalized anisotropic mean curvature is given by
\begin{equation*}
E_1(\kappa)=\frac 1n\hat{g}^{ij}\hat{h}_{ij}=\frac 1nu_{ij}D^iD^j\hat{\gamma}|_{Du}.
\end{equation*}

Now we derive the evolution of the function $u$ along the flow \eqref{flow-VMCF}. Write the position vector field of the evolving hypersurface near a point $(x_0,t_0)$ as
\begin{equation*}
  X(p,t)~=~y^i(p,t)e_i+u(y(p,t),t)e_0,
\end{equation*}
where $y(p,t)$ gives a time-dependent diffeomorphism of $\mathbb{R}^n$ such that the graphical representation is preserved along the flow. Then
\begin{equation*}
  \frac{\partial}{\partial t}X(p,t)~=~\frac{\partial}{\partial t}y^i(p,t)e_i+\left(\frac{\partial}{\partial t}u(y(p,t),t)+D_iu \frac{\partial}{\partial t}y^i(p,t)\right)e_0.
\end{equation*}
Taking the expression \eqref{s10:nu-F} of $\nu_\gamma$ into account, from the flow equation \eqref{flow-VMCF} we have
\begin{align*}
  \frac{\partial}{\partial t}y^i(p,t)&=(\phi(t)-E_1^{\alpha})D^i\hat{\gamma}|_{Du}, \\
  \frac{\partial}{\partial t}u(y(p,t),t)+D_iu \frac{\partial}{\partial t}y^i(p,t)&=(\phi(t)-E_1^{\alpha})\left(D_iuD^i\hat{\gamma}|_{Du}-\hat{\gamma}(Du)\right).
\end{align*}
Combining the above two equations, we see that the flow \eqref{flow-VMCF} is equivalent to the following scalar parabolic equation
\begin{align}
\frac{\partial}{\partial t}u=&~\hat{\gamma}(Du)(E_1^{\alpha}-\phi(t))\nonumber\\
=&~\hat{\gamma}(Du)\left(\frac 1nD^iD^j\hat{\gamma}|_{Du}u_{ij}\right)^{\alpha}-\hat{\gamma}(Du)\phi(t)\label{s10:u-evl}
\end{align}
of the function $u$ locally on $\mathbb{R}^n$.

The bound on the anisotropic principal curvatures and the convexity of $M_t$ implies that $|Du|+|D^2u|\leq C$ for some uniform constant on a uniform space-time neighbourhood of any $(y_0,t_0)$ chosen such that $u(y_0,t_0)=0$ and $Du|_{(y_0,t_0)}=0$ (see, e.g., the argument in \cite[\S 9]{And01}). Since $E_1$ is linear in $D^2u$, by the uniform bound
\begin{equation*}%\label{s10:E1}
  0<1/C\leq {E}_1\leq C
\end{equation*}
on $E_1$ and the bound on $|Du|$, the equation \eqref{s10:u-evl} is uniformly parabolic. Since $\alpha>0$ and $E_1$ is linear with respect to $D^2u$, we can apply Theorem 6 of \cite{And04} to derive the uniform $C^{2,\beta}$ estimate of $u$ for some $\beta\in (0,1)$ in a uniform space-time neighbourhood of $(y_0,t_0)$.  The parabolic Schauder theory \cite{Lieb96} then implies the $C^{\ell,\beta}$ estimate of $u$ for all $\ell\geq 2$. This completes the proof of the lemma.
\endproof
By the Hausdorff convergence of $M_t$ to a scaled Wulff shape $W$, we conclude that $M_t$ converges smoothly to a scaled Wulff shape $M_{\infty}=\bar{r}\Sigma$  up to a translation.

\subsection{Smooth convergence for $k=n$, $\alpha>0$}$\ $

As mentioned in \S \ref{sec:2-2}, in the case $k=n$, we have $ E_n(\kappa)=\det(A_{\gamma}(\nu))E_n(\lambda)$, where $\lambda=(\lambda_1,\dots,\lambda_n)$ denote the isotropic principal curvatures of $M_t$. Since the anisotropic normal vector field is given by $\nu_{\gamma}=\gamma(\nu)\nu+\nabla^{\mathbb{S}}\nu$, the flow \eqref{flow-VMCF} is equivalent to
\begin{equation}\label{s10:flow-kn}
  \frac{\partial }{\partial t}X=\left(\phi(t)-\det(A_{\gamma}(\nu))^{\alpha/n}E_n(\lambda)^{\alpha/n}\right)\gamma(\nu)\nu
\end{equation}
up to a tangential diffeomorphism. If we define the isotropic support function $h(z,t):\mathbb{S}^n\to \mathbb{R}$ of $M_t=\partial K_t$ by
\begin{equation*}
  h(z,t):=\sup\{\langle x,z\rangle: x\in K_t\},
\end{equation*}
we can recover $M_t$ via an embedding $\bar{X}(z,t)$: $\mathbb{S}^n\to \mathbb{R}^{n+1}$:
\begin{equation*}
  \bar{X}(z,t)= h(z,t)z+\nabla^{\mathbb{S}}_ih(z,t)\sigma^{ij}\nabla^{\mathbb{S}}_jz,
\end{equation*}
where we set $g_{\mathbb{S}^n}=(\sigma_{ij})$.  Denote by $\mathfrak{r}=(\mathfrak{r}_{ij})$ the matrix
\begin{equation*}
\mathfrak{r}_{ij}=\nabla^{\mathbb{S}}_i\nabla^{\mathbb{S}}_jh+h\sigma_{ij},
\end{equation*}
whose eigenvalues are the principal radii of curvature $\mathfrak{r}_i=1/{\lambda_i}$, $i=1,\dots,n$ of $M_t$. Then it is well known that we can rewrite \eqref{s10:flow-kn} as a scalar parabolic PDE of the isotropic support function $h(z,t)$ on the unit sphere:
\begin{align}\label{s10:flow-kn2}
  \frac{\partial }{\partial t}h(z,t)=&\left(\phi(t)-\det(A_{\gamma}(z))^{\alpha/n}E_n(\mathfrak{r})^{-\alpha/n}\right)\gamma(z)\nonumber\\
  =&\phi(t)\gamma(z)-\rho(z)E_n(\mathfrak{r})^{-\alpha/n},
\end{align}
where we set $\rho(z)=\det(A_{\gamma}(z))^{\alpha/n}\gamma(z)$ for simplicity of the notation. Now the equation \eqref{s10:flow-kn2} differs from \cite[Equ. (13)]{And00} only by the term $\phi(t)\gamma(z)$.

Since the matrix $A_{\gamma}(z)$ is uniformly positive definite, the two-sided positive bound on $E_n(\kappa)$ is equivalent to
\begin{equation}\label{s10:En-bd}
  0< C_1\leq E_n(\mathfrak{r})\leq C_2
\end{equation}
for some positive constants $C_1$ and  $C_2$. Then we can apply the argument in Theorem 10 of \cite{And00} to conclude that the principal radii $\mathfrak{r}_i$ are uniformly bounded from above. In fact, noting that the term $\phi(t)\gamma(z)$ in \eqref{s10:flow-kn2} only contributes some lower order terms in the evolution of $\mathfrak{r}_{ij}$, which are uniformly controlled,  we can apply the calculation in Theorem 10 of \cite{And00} to obtain that
\begin{align*}
  \frac{\partial }{\partial t}\mathfrak{r}_{ij} \leq & \frac{\alpha}n\rho(z)E_n(\mathfrak{r})^{-(1+\frac{\alpha}n)}\dot{E}_n^{k\ell}\nabla^{\mathbb{S}}_k\nabla^{\mathbb{S}}_\ell\mathfrak{r}_{ij} - \frac{\alpha}n\rho(z)E_n(\mathfrak{r})^{-(1+\frac{\alpha}n)}\dot{E}_n^{k\ell}(\mathfrak{r})\sigma_{k\ell}\mathfrak{r}_{ij}\\
   & +CE_n(\mathfrak{r})^{-\alpha/n}\sigma_{ij}+\phi(t)A_{\gamma}(z)_{ij},
\end{align*}
where $C$ depends on $\gamma$ and $\alpha$. The second term here is crucial. By the generalized Newton's inequality, we have
\begin{align*}
  \dot{E}_n^{k\ell}(\mathfrak{r})\sigma_{k\ell}= & nE_{n-1}(\mathfrak{r})\geq nE_{n}(\mathfrak{r})^{\frac{n-2}{n-1}}E_1(\mathfrak{r})^{\frac 1{n-1}} \geq C E_{n}(\mathfrak{r})^{\frac{n-2}{n-1}} \mathfrak{r}_{\max}^{\frac 1{n-1}}.
\end{align*}
Now work at a space-time point where a maximum eigenvalue of $\mathfrak{r}_{ij}$ is attained. We have
\begin{align*}
  \frac{\partial }{\partial t}\mathfrak{r}_{\max} \leq & - \frac{\alpha}n\rho(z)E_n(\mathfrak{r})^{-(\frac 1{n-1}+\frac{\alpha}n)}\mathfrak{r}_{\max}^{\frac n{n-1}}+C\left(E_n(\mathfrak{r})^{-\alpha/n}+\phi(t)\right)
\end{align*}
for a constant $C$ depending only on $\alpha$ and $\gamma$. Given the bound \eqref{s10:En-bd} on $E_n(\mathfrak{r})$ and Corollary~\ref{cor-phi-bd} on $\phi(t)$, the parabolic maximum principle implies that
\begin{equation*}
  \mathfrak{r}_{\max} \leq C\left(1+\frac 1{t^{n-1}}\right)
\end{equation*}
for some positive constant $C$ depending only on $C_1,C_2$ and $n,\alpha,k, \gamma, \mathrm{Vol}(K_0)$ and $V_{n+1-k}(K_0)$.

The upper bound on $\mathfrak{r}_i$ together with $E_n(\mathfrak{r})\geq C_1$ also implies the uniform positive lower bound of $\mathfrak{r}_i$.  Therefore the equation \eqref{s10:flow-kn2} is uniformly parabolic. The estimate on $\mathfrak{r}_i$ also implies the $C^2$ estimate of the isotropic support function $h(z,t)$. Since \eqref{s10:flow-kn2} is concave with respect to the second spatial derivatives of $h$, we apply Theorem 1.1 of \cite{TW13} to derive the $C^{2,\beta}$ estimate of $h(z,t)$. By the parabolic Schauder theory \cite{Lieb96}, we conclude that the $C^{\ell,\beta}$ norm of $h$ is uniformly bounded for any $\ell\geq 2$. This is equivalent to the uniform $C^{\ell,\beta}$ regularity estimate of the evolving hypersurfaces $M_t$.  This together with the Hausdorff convergence proved in Proposition~\ref{s10.lem1} implies that $M_t$ converges to a translated scaled Wulff shape in the $C^{\infty}$ topology as $t\to\infty$.

\subsection{Smooth convergence for $\alpha\geq k=1,\dots,n$}$\ $

We first prove the improved $C^0$ and $C^1$ estimates.
\begin{lem}\label{s12.lem1}
Let $M_t$, $t\in [0,\infty)$ be a smooth solution to the flow \eqref{flow-VMCF} with $k=1,\dots,n$ and $\alpha>0$. For any small $\varepsilon>0$, there exists a large time $N_0$ and a fixed constant $d>0$ such that for all $t_0>N_0$, if the origin $o$ is chosen with $|s(z,t_0)-\bar{r}|\leq \varepsilon \bar{r}$, then
\begin{equation}\label{s12.1}
  |s(z,t)-\bar{r}|\leq C\varepsilon \bar{r},\quad \forall~z\in \Sigma
\end{equation}
for all $t\in [t_0, t_0+d]$, where $s(z,t)$ is the anisotropic support function and $C$ depends on $d, \bar{r}$ but not on $t_0$. This combined with the $C^1$ estimate \eqref{s9:ds-1} implies that
\begin{equation}\label{s12.1-ds}
  |\bar{\nabla} s(z,t)|^2\leq C\varepsilon \bar{r}
\end{equation}
for all $t\in [t_0, t_0+d]$.
\end{lem}
\proof
By Proposition \ref{s10.lem1} and Corollary \ref{s10.cor}, for any given $0<\varepsilon<1/4$ there exists  $N_0$ such that for all time $t>N_0$, the anisotropic Hausdorff distance of $M_t$ to a scaled Wulff shape $\bar{r}\Sigma$ is less than $\varepsilon$, and
\begin{equation*}%\label{s12.1}
\int_t^{t+d}\bigg|\phi(x)-\frac 1{\bar{r}^\alpha}\bigg| dx\leq \frac {\varepsilon}{\bar{r}^\alpha}
\end{equation*}
for any fixed constant $d>0$. For each $t_0>N_0$, we choose the origin $o$ such that the anisotropic support function of $M_{t_0}$ with respect to $o$ satisfies
\begin{equation*}
  |s(z,t_0)-\bar{r}|\leq \varepsilon \bar{r}.
\end{equation*}
Since $s(z,t)$ satisfies the parabolic equation
 \begin{equation*}
   \frac{\partial }{\partial t}s=-F_*(\tau_{ij})^{-\alpha}+\phi(t),
 \end{equation*}
by adjusting $d$ if necessary, the uniform bounds on $F_*$ and $\phi(t)$ imply that
\begin{equation}\label{s12.2}
  |s(z,t)-\bar{r}|\leq \frac 12\bar{r}
\end{equation}
for all time $t\in [t_0, t_0+d]$,  where $d$ is independent of $t_0$.
%\begin{center}
%\begin{tikzpicture}
%\draw[->,thick] (-2,0) -- (7,0)  node[below]{$t$};
%\draw[red,thick] (3.7,-0.05) -- (3.7, 0.05) node[below]{$t_0$};
%\draw[red,thick] (5,-0.05) -- (5, 0.05) node[above]{$t_0+d$};
%\draw[red](-1,-0.05) -- (-1, 0.05) node[above]{$0$};
%\end{tikzpicture}
%\end{center}
Let $\underline{s}(t)=\min_{z\in\Sigma}s(z,t)$.  We have
\begin{align}\label{s12.3}
  \frac d{dt}\underline{s}(t)\geq & -\underline{s}(t)^{-\alpha} +\phi(t).
\end{align}
We claim that $\underline{s}(t)-\bar{r}\geq -C\varepsilon\bar{r}$ on $[t_0,t_0+d]$ for a suitable $d$. Without loss of generality, we assume $\underline{s}(t)<\bar{r}$. Integrating \eqref{s12.3} implies that
\begin{align}\label{s12.4}
  \underline{s}(t)-\underline{s}(t_0)\geq & \int_{t_0}^t(\phi(x)-\underline{s}(x)^{-\alpha})dx\nonumber\\
 \geq &\int_{t_0}^t\left((1-\varepsilon) \bar{r}^{-\alpha}-\underline{s}(x)^{-\alpha}\right)dx\nonumber\\
 =& \int_{t_0}^t\frac 1{\bar{r}^\alpha\underline{s}(x)^\alpha }\left( \underline{s}(x)^\alpha-\bar{r}^{\alpha}\right)dx-\varepsilon \bar{r}^{-\alpha} (t-t_0).
\end{align}
Note that
\begin{align*}
\underline{s}(t)^\alpha-\bar{r}^\alpha =&\alpha\int_0^1\left((1-x)\bar{r}+x\underline{s}(t)\right)^{\alpha-1}dx~(\underline{s}(t)-\bar{r})\\
\geq & C\bar{r}^{\alpha-1}(\underline{s}(t)-\bar{r}),
\end{align*}
where we used \eqref{s12.2}. We have
\begin{align}\label{s12.5}
 \underline{s}(t)-\bar{r}= & ~ \underline{s}(t)-\underline{s}(t_0)+\underline{s}(t_0)-\bar{r}\nonumber\\
 \geq &~ C\bar{r}^{-\alpha-1}\int_{t_0}^t\left( \underline{s}(x)-\bar{r}\right)dx-\varepsilon \bar{r}^{-\alpha} (t-t_0)+\underline{s}(t_0)-\bar{r}\nonumber\\
 \geq &~C\bar{r}^{-\alpha-1}\int_{t_0}^t\left( \underline{s}(x)-\bar{r}\right)dx-\varepsilon \bar{r}^{-\alpha} (t-t_0)-\varepsilon \bar{r},
\end{align}
where we used the fact that the origin $o$ is chosen such that $\underline{s}(t_0)-\bar{r}\geq -\varepsilon \bar{r}$. Then the integral form of Gronwall's inequality implies that
\begin{align*}%\label{s12.6}
 \underline{s}(t)-\bar{r}\geq & ~ -\varepsilon \left(\bar{r}^{-\alpha} (t-t_0)+\bar{r}\right) \left(1+C\bar{r}^{-\alpha-1}(t-t_0)e^{C\bar{r}^{-\alpha-1}(t-t_0)}\right)\nonumber\\
 \geq & -C\varepsilon \bar{r},
\end{align*}
where $C$ depends only on $\bar{r}$ and $t-t_0$. So we obtain the claim.

A similar argument gives the upper bound on $\bar{s}(t)=\max_{\Sigma}s(\cdot,t)$:
\begin{align*}%\label{s12.6}
 \bar{s}(t)-\bar{r}\leq & ~ C\varepsilon \bar{r}.
\end{align*}
So we have the improved $C^0$ estimate \eqref{s12.1}. The improved $C^1$ estimate \eqref{s12.1-ds} follows from \eqref{s12.1} and \eqref{s9:ds-1}.
\endproof

In the following, we focus on the case $\alpha\geq k=1,\dots,n$ and show that the anisotropic principal curvatures $\kappa_i$ of $M_t$ have a uniform positive lower bound.
\begin{lem}\label{s12.lem3}
Along the flow \eqref{flow-VMCF} with $\alpha\geq k=1,\dots,n$, the anisotropic principal curvatures $\kappa_i$ of the solution $M_t$ have a uniform positive lower bound for $t>0$.
\end{lem}
\proof
We have proved in Proposition \ref{s5:prop-conv} that $\kappa_i$ are positive on any finite time interval.  However, the lower bound on $\kappa_i$ may degenerate as the time approaches the infinity. To prove that $\kappa_i$ has a uniform positive lower bound for all time, it suffices to show this for sufficiently large time.

Recall that in the case $\alpha\geq k=1,\dots,n$, the global term $\phi(t)$ satisfies \eqref{s10.phi1}.  We employ the anisotropic Gauss map parametrization in \S \ref{sec:gauss map}.  We assume that $N_0>0$ is sufficiently large such that the anisotropic Hausdorff distance of $M_t$ to $\bar{r}\Sigma$ is less than $\varepsilon$, and
\begin{equation}\label{s12.phi1}
\phi(t)\geq (1-\varepsilon)\bar{r}^{-\alpha}
\end{equation}
for all time $t\in [N_0,\infty)$. By choosing $N_0$ sufficiently large, the constant $\varepsilon$ in \eqref{s12.phi1} can be any small positive constant.  For each fixed $t_0>N_0$, we choose the origin $o$ such that the anisotropic support function of $M_{t_0}$ with respect to $o$ satisfies
\begin{equation*}
  |s(z,t_0)-\bar{r}|\leq \bar{r}\varepsilon,\qquad \forall~z\in \Sigma.
\end{equation*}
By Lemma \ref{s12.lem1}, there exists a constant $d>0$ and $C=C(d,\bar{r})$ independent of $t_0$ such that \eqref{s12.1} and \eqref{s12.1-ds} hold for all  $t\in [t_0,t_0+d]$.

In the following, we only focus on the time $t\geq N_0$. To estimate the lower bound on the anisotropic principal curvatures $\kappa_i$, we can equivalently estimate the upper bound on the anisotropic principal radii of curvature $\tau_i$, which are eigenvalues of $(\tau_{ij})$ defined in \eqref{s4:tau-def}. We define
\begin{equation*}
  \zeta= \sup\{\tau_{ij}\xi^i\xi^j:~|\xi|=1\}
\end{equation*}
and for any $t_0>N_0$ we consider the function
\begin{equation}\label{s12.omeg-def}
  \omega(z,t) =\log\zeta(z,t)-\mu_1 s(z,t)+\mu_2 |\bar{\nabla}s|^2(z,t)
\end{equation}
on the interval $[t_0, t_0+d]$, where $\mu_1, \mu_2>0$ are positive constants to be determined later. We consider a point $(z_1,t_1)$ where a new maximum of the function $\omega$ is achieved. By rotation of the local orthonormal frame, we assume that $\xi=e_1$ and $(\tau_{ij})=\mathrm{diag}(\tau_1,\dots,\tau_n)$ is diagonal at $(z_1,t_1)$. First, at $(z_1,t_1)$ we have $\zeta=\tau_{11}$ and
\begin{align*}
  \frac{\partial }{\partial t}\omega =& \frac 1{\tau_{11}}\frac{\partial }{\partial t}\tau_{11} -\mu_1 \frac{\partial }{\partial t}s+\mu_2\frac{\partial }{\partial t}|\bar{\nabla}s|^2,\\
  \bar{\nabla}_\ell\omega =& \frac 1{\tau_{11}}\bar{\nabla}_\ell\tau_{11}-\mu_1 \bar{\nabla}_\ell s+\mu_2\bar{\nabla}_\ell|\bar{\nabla}s|^2,\\
  \bar{\nabla}_k\bar{\nabla}_\ell\omega =&\frac 1{\tau_{11}}\bar{\nabla}_k\bar{\nabla}_\ell\tau_{11}-\mu_1 \bar{\nabla}_k\bar{\nabla}_\ell s+\mu_2\bar{\nabla}_k\bar{\nabla}_\ell|\bar{\nabla}s|^2-\frac 1{(\tau_{11})^2}\bar{\nabla}_k\tau_{11}\bar{\nabla}_\ell\tau_{11}.
\end{align*}
Combining these equations with \eqref{s4:s-1}, \eqref{s8:ds-0} and \eqref{s4:flow-tau2} gives
\begin{align}\label{s12.evl-S}
   &\frac{\partial }{\partial t}\omega -\dot{\Psi}^{k\ell}\left(\bar{\nabla}_k\bar{\nabla}_\ell\omega-\frac 12Q_{k\ell p}\bar{\nabla}_p\omega\right)\nonumber\\
   =&\frac 1{\tau_{11}}\left(\frac{\partial }{\partial t}\tau_{11} -\dot{\Psi}^{k\ell}\left(\bar{\nabla}_k\bar{\nabla}_\ell\tau_{11}-\frac 12Q_{k\ell p}\bar{\nabla}_p\tau_{11}\right)\right)\nonumber\\
   & -\mu_1 \left(\frac{\partial }{\partial t}s -\dot{\Psi}^{k\ell}\left(\bar{\nabla}_k\bar{\nabla}_\ell s-\frac 12Q_{k\ell p}\bar{\nabla}_ps\right)\right)\nonumber\\
   &+\mu_2 \left(\frac{\partial }{\partial t}|\bar{\nabla}s|^2 -\dot{\Psi}^{k\ell}\left(\bar{\nabla}_k\bar{\nabla}_\ell|\bar{\nabla}s|^2-\frac 12Q_{k\ell p}\bar{\nabla}_p|\bar{\nabla}s|^2\right)\right)\nonumber\\
   &+\frac 1{(\tau_{11})^2}\dot{\Psi}^{k\ell}\bar{\nabla}_k\tau_{11}\bar{\nabla}_\ell\tau_{11}\nonumber\\
\leq & \frac 1{\tau_{11}}\left(\ddot{\Psi}^{k\ell,pq}\bar{\nabla}_1\tau_{k\ell}\bar{\nabla}_1\tau_{pq}+(1-\alpha)\Psi+\phi(t)+C\dot{\Psi}^{k\ell}\bar{g}_{k\ell}\tau_{11}\right)\nonumber\\
  &-\mu_1\left(\phi(t)+(1+\alpha)\Psi+s\dot{\Psi}^{k\ell}\bar{g}_{k\ell}\right) +\frac 1{(\tau_{11})^2}\dot{\Psi}^{k\ell}\bar{\nabla}_k\tau_{11}\bar{\nabla}_\ell\tau_{11}\nonumber\\
  &+\mu_2\dot{\Psi}^{k\ell}\biggl(-2\tau_{ik}\tau_{i\ell}-2s^2\bar{g}_{k\ell}+4s\tau_{k\ell}-2\tau_{ik}Q_{i\ell p}s_p+2sQ_{k\ell p}s_p\biggr)\nonumber\\
 &+\mu_2\dot{\Psi}^{k\ell}\biggl(2s_ks_\ell -\frac 12(2Q_{i\ell m}Q_{mkp}-Q_{ipm}Q_{mk\ell})s_is_p-\bar{\nabla}_iQ_{k\ell p}s_ps_i\biggr)\nonumber\\
 \leq &\frac 1{\tau_{11}}\ddot{\Psi}^{k\ell,pq}\bar{\nabla}_1\tau_{k\ell}\bar{\nabla}_1\tau_{pq} +\frac 1{(\tau_{11})^2}\dot{\Psi}^{k\ell}\bar{\nabla}_k\tau_{11}\bar{\nabla}_\ell\tau_{11}\nonumber\\
 &-\mu_1\left(\phi(t)+(1+\alpha)\Psi+s\sum_k\dot{\psi}^k\right)-2\mu_2\sum_k\dot{\psi}^k\tau_k^2\nonumber\\
 &+(1-\alpha)\frac 1{\tau_{11}}\Psi+\frac 1{\tau_{11}}\phi(t)-\alpha\mu_2C\Psi +C(1+\mu_2\varepsilon)\sum_k\dot{\psi}^k,
\end{align}
where in the last inequality we used the $C^0$, $C^1$ estimates of $s$ and the homogeneity of $\Psi$
\begin{align*}
  \dot{\Psi}^{k\ell}\tau_{k\ell} =& -\alpha \Psi \\
  -2\dot{\Psi}^{k\ell}\tau_{ik}Q_{i\ell p}s_p\leq &C\sum_k\dot{\psi}^k\tau_k=-C\alpha \Psi.
\end{align*}
The constants $C$ in the last line of \eqref{s12.evl-S} depend on $\gamma$ and the initial hypersurface $M_0$. We next apply the maximum principle to \eqref{s12.evl-S} to obtain a uniform upper bound on $\tau_{11}$. Note that $\tau_{11}=\tau_1$ is the largest eigenvalue of $(\tau_{ij})$ at the point $(z_1,t_1)$.

We first estimate the gradient terms in \eqref{s12.evl-S}. By \eqref{s2:F-ddt} and the fact that $\Psi=-F_*^{-\alpha}$ is concave, we have
\begin{align}\label{s12.lem3-pf1}
  \ddot{\Psi}^{k\ell,pq}\bar{\nabla}_1\tau_{k\ell}\bar{\nabla}_1\tau_{pq}= &~\ddot{\psi}^{k\ell}\bar{\nabla}_1\tau_{kk}\bar{\nabla}_1\tau_{\ell\ell}+2\sum_{k>\ell}\frac{\dot{\psi}^k-\dot{\psi}^\ell}{\tau_k-\tau_\ell}(\bar{\nabla}_1\tau_{k\ell})^2\nonumber\\
   \leq &~ -2\sum_{k>1}\frac 1{\tau_1}(\dot{\psi}^k-\dot{\psi}^1)(\bar{\nabla}_1\tau_{k1})^2.
\end{align}
Codazzi equation \eqref{s4:Codaz} implies that
\begin{align}\label{s12.lem3-pf2}
  (\bar{\nabla}_1\tau_{k1})^2= & \left( \bar{\nabla}_k\tau_{11}+\frac 12Q_{11p}\tau_{kp}-\frac 12Q_{k1p}\tau_{1p}\right)^2
  \geq   \frac 12(\bar{\nabla}_k\tau_{11})^2-C(\tau_{11})^2.
\end{align}
Substituting \eqref{s12.lem3-pf2} into \eqref{s12.lem3-pf1}, and noting that
\begin{equation*}
  \frac 1{\tau_{11}}\bar{\nabla}_k\tau_{11}=\mu_1\bar{\nabla}_ks-\mu_2 \bar{\nabla}_k|\bar{\nabla}s|^2
\end{equation*}
holds at $(z_1,t_1)$,   we have
\begin{align*}%\label{s5:LTE3}
   & \frac 1{\tau_{11}}\ddot{\Psi}^{k\ell,pq}\bar{\nabla}_1\tau_{k\ell}\bar{\nabla}_1\tau_{pq}+\frac 1{(\tau_{11})^2}\sum_k\dot{\psi}^k(\bar{\nabla}_k\tau_{11})^2\nonumber\\
   \leq &~-2\sum_{k>1}\frac 1{\tau_{11}^{2}}(\dot{\psi}^k-\dot{\psi}^1)\left(\frac 12(\bar{\nabla}_k\tau_{11})^2-C(\tau_{11})^2\right)+\frac 1{(\tau_{11})^2}\sum_k\dot{\psi}^k(\bar{\nabla}_k\tau_{11})^2\nonumber\\
= &~\dot{\psi}^1\tau_{11}^{-2}\sum_k(\bar{\nabla}_k\tau_{11})^2+2C\sum_{k>1}(\dot{\psi}^k-\dot{\psi}^1)\nonumber\\
   \leq &~\dot{\psi}^1\sum_k\left(\mu_1\bar{\nabla}_ks-\mu_2 \bar{\nabla}_k|\bar{\nabla}s|^2\right)^2+2C\sum_{k>1}(\dot{\psi}^k-\dot{\psi}^1)
\end{align*}
at $(z_1,t_1)$. By the definition \eqref{s4:tau-def} of $\tau_{ij}$ and the improved $C^0$, $C^1$ estimates of $s$ in Lemma \ref{s12.lem1}, we have
\begin{align*}
  \left(\mu_1\bar{\nabla}_ks-\mu_2 \bar{\nabla}_k|\bar{\nabla}s|^2\right)^2=& \left(\mu_1\bar{\nabla}_ks-2\mu_2 \bar{\nabla}_k\bar{\nabla}_\ell s\bar{\nabla}_\ell s\right)^2\\
  =& \left(\mu_1\bar{\nabla}_ks-2\mu_2 \bar{\nabla}_\ell s(\tau_{k\ell}-s\bar{g}_{k\ell}+\frac 12Q_{k\ell p}\bar{\nabla}_ps)\right)^2\\
  \leq & C\varepsilon \left(\mu_1^2+\mu_2^2(\tau_1^2+1)\right),\quad k=1,\dots,n
\end{align*}
for some constant $C$ depending only on $M_0$ and $\gamma$. In summary, the gradient terms on the right-hand side of \eqref{s12.evl-S} satisfies
\begin{align}\label{s12.gradt}
& \frac 1{\tau_{11}}\ddot{\Psi}^{k\ell,pq}\bar{\nabla}_1\tau_{k\ell}\bar{\nabla}_1\tau_{pq}+\frac 1{(\tau_{11})^2}\dot{\Psi}^{k\ell}\bar{\nabla}_k\tau_{11}\bar{\nabla}_\ell\tau_{11}\nonumber\\
 \leq  & C\varepsilon\dot{\psi}^1\left(\mu_1^2+\mu_2^2(\tau_1^2+1)\right)+C\sum_k\dot{\psi}^k
\end{align}
at $(z_1,t_1)$.

Since the left hand side of \eqref{s12.evl-S} is nonnegative at $(z_1,t_1)$, rearranging the terms in \eqref{s12.evl-S} and applying \eqref{s12.gradt}, we have
\begin{align}\label{s12.lem3-pf3}
&\mu_1\left(\phi(t)+(1+\alpha)\Psi+s\sum_k\dot{\psi}^k\right)+2\mu_2\sum_k\dot{\psi}^k\tau_k^2\nonumber\\
 \leq &C\varepsilon \dot{\psi}^1\left(\mu_1^2+\mu_2^2(\tau_1^2+1)\right)+C(\mu_2\varepsilon+1)\sum_k\dot{\psi}^k\nonumber\\
 &+(1-\alpha)\frac 1{\tau_{11}}\Psi+\frac 1{\tau_{11}}\phi(t)-\alpha\mu_2C\Psi
\end{align}
 at $(z_1,t_1)$. Note that
\begin{equation*}
  \Psi=-F_*^{-\alpha},\quad \psi^k=\alpha F_*^{-\alpha-1}\dot{f}^k_*,
\end{equation*}
where $F_*(\tau_{ij})=f_*(\tau)$. Rewriting the terms in \eqref{s12.lem3-pf3} and multiplying by $F_*^{\alpha+1}$ both sides give that
\begin{align}\label{s12.lem3-pf4}
&\mu_1\left(\phi(t)F_*^{\alpha+1}-(1+\alpha)F_*+\alpha s\right)+\mu_1\alpha s\left(\sum_k\dot{f}^k_*-1\right)+2\mu_2\alpha \sum_k\dot{f}_*^k\tau_k^2\nonumber\\
 \leq &C\varepsilon\alpha \dot{f}_*^1\left(\mu_1^2+\mu_2^2(\tau_1^2+1)\right)+C(\mu_2\varepsilon+1)\alpha \sum_k\dot{f}^k_*\nonumber\\
 &+(\alpha-1)\frac 1{\tau_{11}}F_*+\frac 1{\tau_{11}}\phi(t)F_*^{\alpha+1}+\alpha\mu_2CF_*
\end{align}
at $(z_1,t_1)$. Using the estimate \eqref{s12.phi1} and Lemma \ref{s12.lem1}, the first term on the left side of \eqref{s12.lem3-pf4} satisfies
\begin{align}\label{s12.lem3-pf5}
   & \mu_1\left(\phi(t)F_*^{\alpha+1}-(1+\alpha)F_*+\alpha s\right) \nonumber\\
   \geq & \mu_1\bar{r}\left((1-\varepsilon)\left(\frac{F_*}{\bar{r}}\right)^{\alpha+1}-(1+\alpha)\frac {F_*}{\bar{r}}+\alpha \right)+\alpha \mu_1 (s-\bar{r})\nonumber\\
  \geq & \mu_1\bar{r}\left(\left(\frac{F_*}{\bar{r}}\right)^{\alpha+1}-(1+\alpha)\frac {F_*}{\bar{r}}+\alpha \right)-\mu_1\varepsilon \bar{r}^{-\alpha}F_*^{\alpha+1}-C\varepsilon\alpha \mu_1\bar{r}\nonumber\\
  \geq & -\mu_1\varepsilon \bar{r}^{-\alpha}F_*^{\alpha+1}-C\varepsilon\alpha \mu_1\bar{r},
\end{align}
where $C$ is the constant in Lemma \ref{s12.lem1} depending only on $d$ and $\bar{r}$. Substituting \eqref{s12.lem3-pf5} into \eqref{s12.lem3-pf4}, we have the following estimate
\begin{align}\label{s12.lem3-pf6}
&2\mu_2\alpha \sum_k\dot{f}_*^k\tau_k^2+\mu_1\alpha s\left(\sum_k\dot{f}^k_*-1\right)\nonumber\\
 \leq &C\varepsilon\alpha \dot{f}_*^1\mu_2^2\tau_1^2+C\varepsilon\alpha \dot{f}_*^1\left(\mu_1^2+\mu_2^2\right)+C(\mu_2\varepsilon+1)\alpha \sum_k\dot{f}^k_*\nonumber\\
 &+C(\mu_2+1)+C\mu_1\varepsilon+C\varepsilon\alpha \mu_1\bar{r}
\end{align}
holding at $(z_1,t_1)$, where the constants $C$ are uniform and independent of the time $t_0$ and the length $d$ of the time interval.

Now, we apply an observation in \cite{Guanbo15} and \cite{xia-2} which relies on the property that $f_*$ is increasing and concave in $\Gamma_+$ and vanishes on $\partial\Gamma_+$. Precisely, for $f_*=\left(\frac {E_{n}}{E_{n-k}}\right)^{1/k}$, we have the following property.
\begin{prop}[Lemma 2.2 in \cite{Guanbo15}]\label{s12.prop-Guan}
Let $K\subset \Gamma_+$ be a compact set and $\varrho>0$ a constant. Then there exists a constant $\theta>0$ depending only on $K$ and $\varrho$ such that for $\tau\in \Gamma_+$ and $\mu\in K$ satisfying $|\nu_{\tau}-\nu_\mu|\geq \varrho$, we have
\begin{equation*}
  \sum_{k}\dot{f}_*^k(\tau)\mu_k-f_*(\mu)\geq \theta \left(\sum_k\dot{f}_*^k(\tau)+1\right),
\end{equation*}
where $\nu_\tau$ denotes the unit normal vector of the level set $\{x\in \Gamma_+:~f_*(x)=f_*(\tau)\}$ at the point $\tau$, i.e., $\nu_\tau=Df_*(\tau)/|Df_*(\tau)|$.
\end{prop}
Let $\mu=2^{-1}\bar{r}(1,\dots,1)\in \Gamma_+$ and $K=\{\mu\}$ in Proposition \ref{s12.prop-Guan}. Then there exists a constant $\varrho=\bar{r}/8>0$ such that $\nu_\mu-2\varrho (1,\dots,1)\in \Gamma_+$. We have two cases. First, if the anisotropic principal radii of curvature $\tau$ at the point  $(z_1,t_1)$  satisfy $|\nu_{\tau}-\nu_\mu|\geq \varrho$, Proposition \ref{s12.prop-Guan} implies that
\begin{equation}\label{s12.guan-pf1}
 \sum_k\dot{f}_*^k(\tau)-1\geq \frac {2\theta} {\bar{r}} \left(\sum_k\dot{f}_*^k(\tau)+1\right).
\end{equation}
Substituting \eqref{s12.guan-pf1} into \eqref{s12.lem3-pf6}, we have
\begin{align}\label{s12.guan-pf2}
&2\mu_2\alpha \sum_k\dot{f}_*^k\tau_k^2+\mu_1\alpha s\frac {2\theta} {\bar{r}} \left(\sum_k\dot{f}^k_*+1\right)\nonumber\\
 \leq &C\varepsilon\alpha \dot{f}_*^1\mu_2^2\tau_1^2+C\varepsilon\alpha \dot{f}_*^1\left(\mu_1^2+\mu_2^2\right)+C(\mu_2\varepsilon+1)\alpha \sum_k\dot{f}^k_*\nonumber\\
 &+C(\mu_2+1)+C\mu_1\varepsilon+C\varepsilon\alpha \mu_1\bar{r}
\end{align}
at $(z_1,t_1)$. Based on the above inequality, we can assume that $\varepsilon$ is chosen with
\begin{equation*}
  C\varepsilon +C\varepsilon\alpha \bar{r}\leq \frac 12\alpha \theta.
\end{equation*}
Note that $\varepsilon$ can be chosen to depend only on $M_0, \gamma$ and $\theta$. Now we first choose $\mu_2=\mu_2(\varepsilon, \alpha, \gamma, M_0)$ small enough such that
\begin{equation*}
C\varepsilon\alpha \dot{f}_*^1\mu_2^2\tau_1^2\leq \mu_2\alpha \sum_k\dot{f}_*^k\tau_k^2
\end{equation*}
and then choose $\mu_1=\mu_1(\alpha, \gamma, M_0)$ large enough such that
\begin{equation*}
  C(\mu_2\varepsilon+1)\alpha \sum_k\dot{f}^k_*+C(\mu_2+1)\leq \frac 12 \mu_1\alpha s\frac {2\theta} {\bar{r}} \left(\sum_k\dot{f}^k_*+1\right).
\end{equation*}
Then \eqref{s12.guan-pf2} becomes
\begin{align*}
\mu_2\alpha \dot{f}_*^1\tau_1^2 \leq & C\varepsilon\alpha \dot{f}_*^1\left(\mu_1^2+\mu_2^2\right),
\end{align*}
which gives a uniform upper bound on $\tau_1$ at the point $(z_1,t_1)$:
\begin{equation*}
  \tau_1\leq \left(\frac 1{\mu_2}C\left(\mu_1^2+\mu_2^2\right)\right)^{1/2}.
\end{equation*}

Second, if the anisotropic principal radii of curvature $\tau$ at the point  $(z_1,t_1)$  satisfy $|\nu_{\tau}-\nu_\mu|\leq \varrho$, it follows from $\nu_\mu-2\varrho (1,\dots,1)\in \Gamma_+$ that $\nu_\tau-\varrho (1,\dots,1)\in \Gamma_+$ and $\dot{f}_*^i\geq C\sum_k\dot{f}_*^k$ for some constant $C$ and all $i=1,\dots,n$. Then
\begin{equation}\label{s12.guan-pf3}
 C\leq  F_*=\sum_k\dot{f}_*^k\tau_k\leq (\sum_k\dot{f}_*^k)\tau_1\leq C\dot{f}_*^1\tau_1.
\end{equation}
We can still choose $\mu_2=\mu_2(\varepsilon,\alpha,\gamma,M_0)$ small enough such that
\begin{equation}\label{s12.guan-pf31}
C\varepsilon\alpha \dot{f}_*^1\mu_2^2\tau_1^2\leq \mu_2\alpha \sum_k\dot{f}_*^k\tau_k^2.
\end{equation}
The function $f_*$ is concave, we always have $\sum_k\dot{f}^k_*\geq 1$. Substituting \eqref{s12.guan-pf3} and \eqref{s12.guan-pf31} into \eqref{s12.lem3-pf6} and discarding the second term on the left-hand side of \eqref{s12.lem3-pf6}, we have
\begin{align}\label{s12.guan-pf4}
\mu_2\alpha \dot{f}_*^1\tau_1^2\leq &C\varepsilon\alpha \dot{f}_*^1\left(\mu_1^2+\mu_2^2\right)+C(\mu_2\varepsilon+1)\alpha \dot{f}^1_*\nonumber\\
 &+C(\mu_2+1)+C\mu_1\varepsilon+C\varepsilon\alpha \mu_1\bar{r}\nonumber\\
 \leq &C\varepsilon\alpha \dot{f}_*^1\left(\mu_1^2+\mu_2^2\right)+C(\mu_2\varepsilon+1)\alpha \dot{f}^1_*\nonumber\\
 &+C\biggl((\mu_2+1)+\mu_1\varepsilon+C\varepsilon\alpha \mu_1\bar{r}\biggr)\dot{f}_*^1\tau_1
\end{align}
at $(z_1,t_1)$. Multiplying by $1/{\dot{f}_*^1}$ both sides of \eqref{s12.guan-pf4} we have
\begin{align*}
\mu_2\alpha \tau_1^2\leq &C\varepsilon\alpha \left(\mu_1^2+\mu_2^2\right)+C(\mu_2\varepsilon+1)\alpha \nonumber\\
 &+C\biggl((\mu_2+1)+\mu_1\varepsilon+C\varepsilon\alpha \mu_1\bar{r}\biggr)\tau_1,
\end{align*}
which gives a uniform upper bound on $\tau_1$ at $(z_1,t_1)$.

Therefore, we can choose a small constant $\mu_2$ and a large constant $\mu_1$ in the definition \eqref{s12.omeg-def} of $\omega(z,t)$ such that in both the above two cases, the largest anisotropic principal radius of curvature $\tau_1$ has a uniform upper bound at $(z_1,t_1)$. This in turn implies that $\omega(z_1,t_1)\leq C$ for some uniform positive constant $C$ depending only on the initial data $M_0$ and $\gamma,\alpha, n$. By the definition of $\omega(z,t)$ and the $C^0$, $C^1$ estimates, we conclude that $\tau_1$ is bounded from above uniformly on the time interval $[t_0, t_0+d]$. Note that the constants $\mu_1$, $\mu_2$ are chosen uniformly and depend only on $M_0$, $\alpha$, $\gamma$. Since $t_0>N_0$ is arbitrary and $d>0$ is a fixed constant,  repeating the above argument completes the proof of Proposition \ref{s12.lem3}.
\endproof

Combining Proposition \ref{s5:thm-Ek-ub} and Proposition \ref{s12.lem3}, we obtain the uniform curvature estimate of the solution $M_t$ to the flow \eqref{flow-VMCF}. In fact, since
\begin{align*}
 E_k\geq &{\binom nk}^{-1}\kappa_n\cdots \kappa_{n-k+1}  \geq  {\binom nk}^{-1}\kappa_n \kappa_1^{k-1},
\end{align*}
the uniform upper bound on $E_k(\kappa)$ and the uniform positive lower bound on the anisotropic principal curvature $\kappa_i$  imply that
\begin{equation*}
  0<\frac 1C\leq \kappa_i\leq C,\qquad i=1,\dots,n
\end{equation*}
for some constant $C$ depending only on $\gamma$, $\alpha$ and $M_0$. Since the dual function $F_*$ is concave, we apply Theorem 1 of \cite{TW13} to the equation \eqref{flow-gauss} of $s$ and derive the $C^{2,\beta
}$ estimate of $s(z,t)$ for some $\beta\in (0,1)$. Then the parabolic Schauder theory \cite{Lieb96} implies that the $C^{\ell,\beta}$ norm of $s(z,t)$ is bounded for any $\ell\geq 2$. This means that the solution $M_t$ of the flow \eqref{flow-VMCF} has uniform regularity estimates for all time $t\in [0,\infty)$. Taking into account the Hausdorff convergence of $M_t$, we conclude that $M_t$ converges smoothly to a scaled Wulff shape as $t\to\infty$.

\subsection{Exponential convergence}$\ $

The stronger convergence results can be obtained by studying the linearization of the flow around the Wulff shape (cf. \cite[\S 12]{And01}). The goal is to prove that the linearization of the flow about the Wulff shape $\Sigma$ is strictly stable. Without loss of generality, we assume that the hypersurfaces near $\Sigma$ can be written as graphs of smooth functions over $\Sigma$
\begin{equation*}
M_t=\{r(z,t)z:z\in \Sigma\},
\end{equation*}
where $r(z,t)=1+\varepsilon \eta(z,t)$. By adding a smooth reparametrization to preserve this graphical representation, the flow \eqref{flow-VMCF} is equivalent to the following parabolic equation for the function $r$:
 \begin{equation}\label{s10:r-dt}
   \frac{\partial r}{\partial t}=(\phi(t)-E_k^{\alpha/k}(\kappa))\left(G(\nu_\gamma,z)-\frac{1}{r}G^{ij}G(\nu_\gamma,\partial_i{z})\frac{\partial r}{\partial z^j}\right),
 \end{equation}
where $G$ is the metric at $z$ on $\mathbb{R}^{n+1}$ described in \eqref{s2:G} and $G^{ij}$ stands for the inverse of the matrix $(G_{ij})_{i,j=1}^n=(G(\partial_iz,\partial_jz))_{i,j=1}^n$.  Note that on $\Sigma$ we have $G(\nu_\gamma,z)=G(z,z)=1$, $G(\nu_\gamma,\partial_i{z})=G(z,\partial_iz)=0$ and $\phi(t)-E_k^{\alpha/k}(\kappa)=0$. Differentiating \eqref{s10:r-dt} with respect to $\epsilon$ and letting $\epsilon=0$, we have
 \begin{equation}\label{s10:r-dt-1}
   \frac{\partial }{\partial t}\eta=\frac{\pt }{\pt \varepsilon}\bigg|_{\varepsilon=0}(\phi(t)-E_k^{\alpha/k}(\kappa)).
 \end{equation}
On $\Sigma$ we have $\hat{h}_i^j=\delta_i^j$. The variation equations \eqref{s2:dmu} and \eqref{eq2.h} imply that
\begin{align}
\frac{\pt }{\pt \varepsilon}\bigg|_{\varepsilon=0}E_k^{\alpha/k}(\kappa )&=\frac{\alpha}kE_k^{\alpha/k-1}(I)E_k^{ij}(I)\left( -\bar{\nabla}^i\bar{\nabla}_j \eta-{A}_j{}^{pi}\pt_p \eta-\eta \hat{h}^p_j\hat{h}_p^i\right)\nonumber\\
=&\frac{\alpha}n\left(-\bar{\Delta}\eta-{A}_{pi}{}^i\bar{\nabla}^p\eta-n\eta\right),\label{s10:r-dt2}
\end{align}
and
\begin{align}
  \frac{d }{d \varepsilon}\bigg|_{\varepsilon=0}\phi(t) =& \frac 1{|\Sigma|_{\gamma}}\int_{\Sigma} \left(\frac{\pt }{\pt \varepsilon}\bigg|_{\varepsilon=0}E_k^{\alpha/k}(\kappa) d\mu_\gamma+E_k^{\alpha/k}(I)\frac{\pt }{\pt \varepsilon}\bigg|_{\varepsilon=0}d\mu_\gamma\right)\nonumber\\
   &\quad -\frac 1{|\Sigma|^2_{\gamma}}\int_{\Sigma} E_k^{\alpha/k}(I)d\mu_\gamma\frac{\pt }{\pt \varepsilon}\bigg|_{\varepsilon=0}|M_t|_{\gamma}\nonumber\\
   =&-\frac{\alpha}n\frac 1{|\Sigma|_\gamma}\int_{\Sigma}n\eta d\mu_{\gamma},\label{s10:r-dt3}
\end{align}
where $I=(1,\dots,1)$ and we used $\int_\Sigma (\bar{\Delta}\eta+{A}_{pi}{}^i\bar{\nabla}^p\eta)d\mu_\gamma=0$ by virtue of Lemma~2.8 in \cite{Xia13}.  Substituting \eqref{s10:r-dt2} and \eqref{s10:r-dt3} into \eqref{s10:r-dt-1} we get
\begin{equation*}
\frac{\pt }{\pt t}\eta~=~\frac{\alpha}n\left(\bar{\Delta}\eta+{A}_{pi}{}^i\bar{\nabla}^p\eta+n\eta-\frac 1{|\Sigma|_\gamma}\int_{\Sigma}n\eta d\mu_{\gamma}\right)~=:\mathcal{L}\eta,
\end{equation*}
which is the linearized equation of the flow \eqref{flow-VMCF} about $\Sigma$. As in Proposition 12.1 of \cite{And01}, the operator $\mathcal{L}$ maps the space $W^{2,2}({\Sigma})\cap \{f:\int_{\Sigma} fd\mu_\gamma=0\}$ to the space $L^2({\Sigma})\cap \{f:\int_{\Sigma} fd\mu_\gamma=0\}$. Moreover, $\mathcal{L}$ is non-positive and self-adjoint with respect to the inner product $\langle f_1,f_2\rangle=\int_{\Sigma} f_1f_2 d\mu_\gamma$, and has kernel consisting precisely of the generators of translations. In other words, the linearization of the flow \eqref{flow-VMCF} about the Wulff shape is strictly stable modulo the action of the translation group. Then Theorem 9.1.2 of \cite{Lun95} applies to conclude that initial hypersurfaces sufficiently close to the Wulff shape converge modulo translations to the Wulff shape, exponentially fast in the $C^\infty$ topology. Since the smooth convergence of $M_t$ is already established, we conclude that the whole family of the solution $M_t$ of the flow \eqref{flow-VMCF}  converge modulo translations to a scaled Wulff shape $M_\infty=\bar{r}\Sigma$ exponentially fast in $C^\infty$ topology. In particular, the speed of the flow satisfies
\begin{equation*}
  |\phi(t)-E_k^{\alpha/k}(\kappa)|\leq C e^{-\delta t}
\end{equation*}
for some constants $C$ and $\delta$, and for all $t>0$. Then we have for any $t_2>t_1>0$,
\begin{align*}
  |X(\cdot,t_2)-X(\cdot,t_1)| \leq & \int_{t_1}^{t_2}\max_{M}\biggl|\frac{\partial X}{\partial t}\biggr|dt \\
   \leq & \int_{t_1}^{t_2}\max_{M}|\phi(t)-E_k^{\alpha/k}(\kappa)|dt\\
   \leq &\frac C{\delta}(e^{-\delta t_1}-e^{-\delta t_2}).
\end{align*}
By the Cauchy criterion, $X(\cdot,t)$ converges to a limit embedding $X_{\infty}(\cdot)$ of $M_{\infty}$ as $t\to\infty$. Hence, $M_t$ converges exponentially and smoothly to a scaled Wulff shape $M_\infty=\bar{r}\Sigma$ without any correction by translations. This completes the proof of the convergence stated in Theorem \ref{thm-2}.

%For the intermediate $k=2,\cdots,n-1$:
%\begin{lem}
%If $\kappa_{n+1-k}\geq C>0$ uniformly, then the flow \eqref{flow-VMCF} is uniformly parabolic.
%\end{lem}
%\proof
%We already proved that ${E}_k$ satisfies
%\begin{equation*}
%  0<1/C\leq {E}_k\leq C
%\end{equation*}
%for some constant $C>0$ along the flow \eqref{flow-VMCF}. Since ${E}_k^{1/k}$ is concave with respect to its variables $\kappa_i$, we have
%\begin{equation*}
%\frac{\partial {E}_k^{1/k}}{\partial \kappa_1}\geq \cdots\geq \frac{\partial {E}_k^{1/k}}{\partial \kappa_n}.
%\end{equation*}
%We first have
%\begin{align*}
%  \frac{\partial {E}_k^{1/k}}{\partial \kappa_n} =& \frac 1k {E}_k^{\frac 1k-1} {E}_{k-1;n} \\
%  \geq  & \frac 1k {E}_k^{\frac 1k-1} \kappa_{n+1-k}\cdots \kappa_{n-1}\\
%  \geq &\frac 1k {E}_k^{\frac 1k-1} (\kappa_{n+1-k})^{k-1}\geq C>0.
%\end{align*}
%On the other hand,
%\begin{align*}
%  \frac{\partial {E}_k^{1/k}}{\partial \kappa_1} =& \frac 1k {E}_k^{\frac 1k-1} {E}_{k-1;1} \\
%  \leq  & \frac 1k {E}_k^{\frac 1k-1} \binom{n-1}{k-1}\kappa_{n+2-k}\cdots \kappa_{n}\\
%  \leq &\frac 1k {E}_k^{\frac 1k-1} \binom{n-1}{k-1}\frac{{E}_k}{\kappa_{n+1-k}}\leq C.
%\end{align*}
%This proves the lemma.
%\endproof

\bibliographystyle{amsplain}

%------------------------------------
\end{document}